\documentclass[11pt,leqno]{amsart}
\usepackage[dvips]{epsfig}
\usepackage{graphics}
\usepackage{latexsym}
\usepackage{verbatim}
\usepackage{amsmath}
\usepackage{amsthm}
\usepackage{amssymb}
\usepackage{empheq}
\usepackage{hyperref}
\usepackage[]{hyperref}
\usepackage{float}
\usepackage{enumerate}
\hypersetup{
    colorlinks=true,%
    filecolor=black,%
    linkcolor=black,%
    urlcolor=black  %
}
\usepackage [table]{xcolor}
\usepackage{multirow}
\usepackage{float}
\usepackage{tikz}
\usepackage{subfig}

\newcommand\br{\begin{remark}}
\newcommand\er{\end{remark}}
\newcommand\bp{\begin{pmatrix}}
\newcommand\ep{\end{pmatrix}}
\newcommand{\be}{\begin{equation}}
\newcommand{\ee}{\end{equation}}
\newcommand{\ba}[1]{\begin{array}{#1}}
\newcommand{\ea}{\end{array}}
\DeclareMathOperator{\beD}{\mathbf{e}}
\DeclareMathOperator\dD{d}

\DeclareMathOperator{\Tr}{Tr}
\newcommand{\mypar}{{\mkern3mu\vphantom{\perp}\vrule depth 0pt\mkern2mu\vrule depth 0pt\mkern3mu}}

\newcommand\eDr{\beD_r}

\newcommand\eDt{\beD_\theta}
\newcommand\eDp{\beD_\varphi}

\newcommand\eDpar{\beD_\mypar}
\newcommand\eDperp{\beD_\perp}

\newcommand\vpar{v_\mypar}
\newcommand\bvperp{\bv_\perp}
\newcommand\buperp{\bu_\perp}
\newcommand\vr{v_r}

\newcommand\vperp{v_\perp}
\newcommand\uperp{u_\perp}

\newcommand\Epar{E_\mypar}
\newcommand\bEperp{\bE_\perp}
\newcommand\bGperp{\bG_\perp}
\newcommand\Fpar{F_\mypar}

\newcommand\bCperp{\bC_\perp}

\newcommand\Eperp{E_\perp}
\newcommand\Gperp{G_\perp}
\newcommand\Cr{C_r}

\newcommand\Gr{G_r}

\newcommand\Cperp{C_\perp}

\newcommand\Er{E_r}

\newcommand\bUperp{\bU_\perp}
\newcommand\Uperp{U_\perp}
\newcommand\Ur{U_r}

\newcommand\cL{{\mathcal L}}
\newcommand\cO{{\mathcal O}}

\newtheorem{theorem}{Theorem}[section]

\newtheorem{proposition}[theorem]{Proposition}

\newtheorem{remark}[theorem]{Remark}

\newcommand\CC{\hbox{C\kern -.58em {\raise .54ex \hbox{$\scriptscriptstyle |$}}\kern-.55em {\raise .53ex \hbox{$\scriptscriptstyle |$}}}}
\newcommand\NN{\hbox{I\kern-.2em\hbox{N}}}
\newcommand\RR{\mathbb{R}}

\newcommand\ZZ{\mathbb{Z}}

\newcommand\Div{ \textrm{div}}

\newcommand\ds{ \displaystyle }
\renewcommand\d{\partial}

\newcommand{\Id}{{\rm Id}}

\newcommand\bA{{\mathbf A}}
\newcommand\bB{{\mathbf B}}
\newcommand\bC{{\mathbf C}}

\newcommand\bE{{\mathbf E}}
\newcommand\bF{{\mathbf F}}
\newcommand\bG{{\mathbf G}}

\newcommand\bJ{{\mathbf J}}

\newcommand\bL{{\mathbf L}}

\newcommand\bR{{\mathbf R}}

\newcommand\bU{{\mathbf U}}
\newcommand\bV{{\mathbf V}}

\newcommand\bX{{\mathbf X}}
\newcommand\bY{{\mathbf Y}}
\newcommand\bZ{{\mathbf Z}}

\newcommand\bu{{\mathbf u}}
\newcommand\bv{{\mathbf v}}

\newcommand\bx{{\mathbf x}}

\newcommand{\EcB}{\bU_{\bE\times\bB}}
\newcommand{\gradB}{\bU_{\nabla b\times\bB}}
\newcommand{\curvB}{\bU_{\rm curv}}

\def\eps{\varepsilon }

\def\signFF{\bigskip\bigskip\hspace{80mm}
\vbox{{\sc Francis Filbet\par\vspace{3mm}
Universit\'e de Toulouse III \par
UMR5219, IMT,\par
118, route de Narbonne\par
F-31062 Toulouse cedex,  FRANCE
\par\vspace{3mm}e-mail:} francis.filbet@math.univ-toulouse.fr }}

\def\signLMR{\bigskip\bigskip\hspace{80mm}
\vbox{{\sc  Luis Miguel Rodrigues \par\vspace{3mm}
Univ Rennes \& IUF,\par
CNRS, IRMAR - UMR 6625,\par
F-35000 Rennes, FRANCE
\par\vspace{3mm}e-mail:}  luis-miguel.rodrigues@univ-rennes1.fr }}

\begin{document}

\title[Particle methods for  strongly magnetized
plasmas]{Asymptotically preserving particle methods for strongly magnetized
plasmas in a torus}

\author{Francis Filbet and Luis Miguel Rodrigues}

\maketitle

\begin{abstract}
We propose and analyze a class of particle methods for the Vlasov equation with a strong  external magnetic field in a torus configuration. In this regime, the time step can
be subject to stability constraints related to the smallness of Larmor radius. To avoid this limitation, our approach is based on higher-order semi-implicit numerical
schemes already validated on dissipative systems \cite{BFR:15} and 
for 
 magnetic fields pointing in a fixed direction \cite{FR1, FR2, FR3}.
It hinges on asymptotic insights gained in \cite{FR:JEP} at the continuous level. 
Thus, when the magnitude of the external magnetic field is large, this scheme provides a consistent approximation of the guiding-center system taking into account curvature and variation of the magnetic field.  Finally, we carry out a theoretical proof of consistency and perform several numerical experiments that establish a solid validation of the method and its underlying concepts.
 \end{abstract}

\vspace{0.1cm}


\tableofcontents

\section{Introduction} 
\setcounter{equation}{0}
\label{sec:1}
The main concern of the present paper is the study of plasma  confined by a strong external nonconstant magnetic
field, where the charged particles evolve under an electrostatic  and
intense confining magnetic field.  This configuration is typical of a
tokamak plasma \cite{bellan_2006_fundamentals, miyamoto_2006_plasma}
where the magnetic field is used to confine particles inside the core
of the device. Kinetic models, based on a mesoscopic description of the various
particles constituting a plasma, and coupled to Maxwell's equations
for the computation of the electromagnetic fields, are very precise
approaches for the study of such thermonuclear fusion
plasmas Here we suppose that collective effects are
dominant and the plasma is entirely modelled with transport
equations, where the unknown is the number density of particles $f\equiv
f(t,\bx,\bv)$  depending on time $t\geq 0$, position
$\bx\in\Omega\subset \RR^3$ and velocity $\bv\in\RR^3$. The transport
equation is written in adimensional form as
\be\label{eq:vlasov}
\d_t f^\eps\,+\,\Div_\bx(f^\eps\,\bv)
\,+\,\Div_\bv\left(f^\eps\,\left(\frac{\bv\wedge \bB(t,\bx)}{\eps}\,+\,\bE(t,\bx)\right)\right)\,=\,0\,,
\ee
where the parameter $\eps$ accounts for the high intensity of the
external magnetic field, $1/\eps$  being related to the so-called gyro-frequency.  The  characteristic flow associated to this
transport equation is encoded by the following ODEs
\be
\label{eq:xv}
\left\{
\ba{l}
\ds\frac{\dD\bx}{\dD t}\,=\,\bv\,,
\\[0.9em]
\ds\frac{\dD\bv}{\dD t}\,=\,\frac{\bv\wedge \bB(t,\bx)}{\eps} \,+\,\bE(t,\bx)\,,
\ea
\right.
\ee
where $\wedge$ denotes the standard vector product on $\RR^3$, $\bB$ stands for the external magnetic field, $\bE$ for an electric field, either external or obtained by solving a field equation. This kinetic equation provides an appropriate description of turbulent transport  in a  general context,  but it also requires to solve a high dimensional problem which leads to a huge computational cost.  One approach consists in  reducing the cost of numerical simulations, by deriving asymptotic models with a smaller number of variables than the kinetic description. Indeed, large magnetic fields  usually lead to the so-called drift-kinetic limit \cite{haz_ware_78,haz_mei_03}. We refer to \cite{HanKwan_PhD,Lutz_PhD,Herda_PhD} and \cite{FR:JEP} for a detailed account of mathematical results on this topic and relevant entering gates to the extensive physical literature. Besides those, on the physical side, we only point out \cite{Burby2020}, as posterior to the references that may found there, and closer to, but distinct from, \cite{FR:JEP} that inspires the numerical methods introduced in the present contribution.


Other approaches are based on the construction of efficient particle
solvers for the original dynamics, to be used as a piece of a PIC scheme. Over the last decade, considerable efforts have been devoted to the design of such solvers and we refer the reader to \cite{boris,vu1995,Cohen2007,Webb2014,FR1,FR2,Hairer2017,cruz2,cruz1,chacon2020,Hairer2020,Wang2020,FR3,wang2021,hairer2022,chen2022} for both significant contributions and relevant entering gates to the now abundant literature. Along these years, roughly speaking, two kind of goals were assigned to the built numerical schemes. On one hand, one may wish to enforce the preservation of some of the geometrical structures of the original system (symplecticity, conservation of the total energy, or of a momentum associated with some group of symmetry,...). On the other hand, one may try to ensure that the designed schemes are consistent with the above-mentioned  asymptoptic reduction, or at least with some of its consequences (approximate conservation of adiabatic invariants, effective spatial drifts,...). Schemes satisfying (some of) the former conditions are typically called structure preserving schemes, whereas those satisfying a version of the latter are named asymptotic preserving. Note that the asymptotic preserving property includes that in the limit $\eps\rightarrow 0$, schemes do capture accurately the non stiff part of the evolution while allowing for coarse discretization parameters. 

A feature of the present evolution, or more generally of rapidly oscillating dynamics, that makes it difficult to capture numerically is that the two kinds of requirement may lead to conflicting choices. Indeed, an apparent paradox to solve is that, in the regime when $\eps$ is small and $(\bx,\bv)$ solve \eqref{eq:xv}, the transverse microscopic kinetic energy $\|\bvperp\|^2/2$, entering in many conserved quantities, evolves slowly and remains of size $1$ whereas the transverse velocity $\bvperp$, that is, the part of the velocity orthogonal to $\bB$, converges (weakly) to an effective drift of size $\eps$, by rapidly oscillating about it, the latter oscillatory convergence being the core of the gyro-kinetic asymptotic reduction. For this reason, many of the structure preserving schemes built with classical tools from geometric numerical integration, such as multi-step schemes \cite{Hairer2017}, variational schemes \cite{Webb2014,Hairer2020,Wang2020}, or splitting schemes \cite{wang2021}, fail to capture accurately the correct asymptotic behavior, and, even worse, many of them are only known to provide the desired structure preservation under upper size constraints on $\Delta t/\eps$, $\Delta t$ denoting the numerical time step. Unfortunately, so far proposed fixes for these geometric schemes, such as the introduction of $\eps$-dependent filters \cite{hairer2022}, are consistent with the exact dynamics, as $\Delta t$ goes to zero, only under lower size restrictions on $\Delta t/\sqrt{\eps}$ so that for the moment none of these provide a satisfactory behavior for the whole range of relevant physical and numerical parameters.

A much less standard class of schemes developed in \cite{cruz2, cruz1}, consists of explicitly doubling time variables, going from $(t,\bx, \bv)$ to $(t, \tau, \bx, \bv)$, where $\tau$ is a periodic time, the original system being recovered at the $\eps$-diagonal $(t, \tau )=(t, t/\eps)$. The corresponding methods are extremely good at
capturing oscillations, and some of them do preserve parts of the relevant geometric structures. Yet their design requires a deep a priori understanding of the detailed structure of oscillations and they seem hard to implement efficiently and to combine with standard field solvers used to compute electromagnetic fields. 

The class of semi-implicit schemes proposed in \cite{FR1,FR2,FR3} focuses on the less ambitious goal of guaranteeing only the asymptotic preserving property, thus the recovering of both the exact dynamics as $\Delta t$ goes to zero, uniformly with respect to $\eps$, and the slow reduced dynamics, including the guiding-center particle motion, when $\eps$ goes to zero, uniformly with respect to $\Delta t$. We stress that by many respects those schemes are remarkably natural and simple, and may easily be inserted in a standard PIC code. For complex geometries, our schemes are designed as high-order semi-implicit schemes \cite{BFR:15} applied to an augmented formulation.  Our supporting strategy is quite systematic and versatile, but so far we have implemented it only for homogeneous magnetic fields \cite{FR1,FR3} and magnetic fields pointing in a fixed direction \cite{FR2}.

Incidentally, we point out that similar restrictions hold for all the schemes described so far, in the sense that mathematical guarantees for asymptotic preservation have been provided only for magnetic fields pointing in a fixed direction, as in \cite{FR2}, and for magnetic fields with a homogeneous intensity, that is, with $\|\bB\|$ constant, as in \cite{cruz1}.  The former configuration precludes magnetic curvature drifts, that are notoriously difficult to capture, whereas the latter ensure that particles share the same period thus remain synchronized, a situation much easier to analyze with filtering or averaging techniques. 

In the present article, we show how our approach may be extended to genuinely three
dimensional magnetic fields, with symmetries of a torus configuration. The geometric framework is thought as a toy model, mimicking realistic configurations used in tokamak devices. However we restrain from designing schemes that provide a second-order description of the full slow dynamics (as in \cite{FR2}) so as to gain, in the trade-off, a structure preservation property, the approximate conservation of the total energy, but also to maintain the complexity of the designed schemes to a bare minimum. With this respect, our present goals are similar to those underlying the design of the fully-implicit numerical schemes in \cite{chacon2020,chen2022}.

\medskip

The rest of the paper is organized as follows. In Section \ref{sec:2},
we present the geometric framework of our study and reformulate the equation of motion
\eqref{eq:xv} in order to identify carefully fast and slow scales. In Section~\ref{sec:2bis}, we derive the expected asymptotic behavior in the regime $\eps \ll 1$ by applying, at the continuous level, the arguments devised in \cite{FR:JEP}.  Then in Section \ref{sec:3}, we present several time discretization techniques and we prove uniform consistency of
the schemes in the limit $\eps\rightarrow 0$. 
 Finally, Section \ref{sec:4} is then devoted to numerical simulations of particle motion in various regimes, including  $\eps \approx 1$ and $\eps\ll 1$.

\section{Toroidal configuration}
\setcounter{equation}{0}
\label{sec:2}

Firstly, we identify adapted coordinates. 
Hence, we pick some radius $R_0>0$ for the torus and introduce toroidal coordinates through $\bx=\bX(r,\theta,\varphi)$, where
\[
\bX(r,\theta,\varphi)\,=\,
\left(
\begin{array}{l}
  R(r,\theta)\cos(\varphi)
  \\[0.9em]
  R(r,\theta)\sin(\varphi)
  \\[0.9em]
  r\sin(\theta)
\end{array}
  \right),
\]
with $R(r,\theta):=R_0+r\cos(\theta)$ and $(r,\theta,\varphi)$ varying in $(0,R_0)\times(\RR/2\pi\ZZ)\times (\RR/2\pi\ZZ)$. To prepare the corresponding change of variables for the velocity we also introduce the orthonormal basis  $(\eDr,\eDt,\eDp)$ by
\[
\left\{
\begin{array}{ll}
\eDr(\theta,\varphi)&=(\cos(\theta)\,\cos(\varphi),\,\cos(\theta)\,\sin(\varphi),\,\sin(\theta))\,,\\[0.9em]
\eDt(\theta,\varphi)&=(-\sin(\theta)\,\cos(\varphi),\,-\sin(\theta)\,\sin(\varphi),\,\cos(\theta))\,,\\[0.9em]
\eDp(\varphi)&=(-\sin(\varphi),\,\cos(\varphi),\,0)\,.
\end{array}
\right.
\]

Then we replace original coordinates $(\bx,\bv)$ with $(r,\theta,\varphi,v_r,v_\theta,v_\varphi)$ where
\[
v_\alpha = \langle \bv,\beD_\alpha\rangle, \quad{\rm for} \quad\alpha \in\{r, \,\theta, \,\varphi\}
\]
In the new coordinates, the equations of motion \eqref{eq:xv} become
\be
\label{eq:xv2}
\left\{
\ba{l}
\ds\frac{\dD r}{\dD t}\,=\,v_r\,,
\hspace{2em} 
r\,\frac{\dD \theta}{\dD t}\,=\,v_\theta\,,
\hspace{2em} 
R(r,\theta)\,\frac{\dD \varphi}{\dD t}\,=\,v_\varphi\,,\\[0.9em]
\ds\frac{\dD v_r}{\dD t}\,=\,\frac{v_\varphi\,B_\theta-v_\theta\,B_\varphi}{\eps} \,+\,E_r\,+\,\frac{v_\theta^2}{r}\,+\,\cos(\theta)\,\frac{v_\varphi^2}{R} ,
\\[0.9em]
\ds\frac{\dD v_\theta}{\dD t}\,=\,\frac{v_r\,B_\varphi-v_\varphi\,B_r}{\eps} \,+\,E_\theta\,-\,\frac{v_r\,v_\theta}{r}\,-\,\sin(\theta)\,\frac{v_\varphi^2}{R},\\[0.9em]
\ds\frac{\dD v_\varphi}{\dD t}\,=\,\frac{v_\theta\,B_r-v_r\,B_\theta}{\eps} \,+\,E_\varphi\,-\,\cos(\theta)\,\frac{v_r\,v_\varphi}{R}\,+\,\sin(\theta)\,\frac{v_\theta\,v_\varphi}{R},
\ea
\right.
\ee
where consistently we have set 
\[
E_\alpha \,=\, \langle \bE,\beD_\alpha\rangle,\qquad  B_\alpha \;=\; \langle \bB,\beD_\alpha\rangle,\quad{\rm for} \quad \alpha\in\{r,\varphi,\theta\}.
\]

Now we make the assumption that the magnetic field is steady and axi-symmetric, that is, that the magnetic field components $(B_r,B_\varphi,B_\theta)$ do not depend on $t$ and $\varphi$. Moreover we also assume that magnetic field lines are contained in $r$ level surfaces, that is, that $B_r\equiv 0$. Hence the magnetic field $\bB$ takes the form
\begin{equation}
  \label{hyp:0}
\bB(r,\theta,\varphi)=b(r,\theta)\eDpar(r,\theta,\varphi)
\end{equation}
with $b$ positive-valued and the unit vector $\eDpar$ defined by
\be
\label{hyp:1}
\eDpar(r,\theta,\varphi)
\,=\, \cos(\omega)\,\eDp(\varphi)\,+\,
\sin(\omega)\,\eDt(\theta,\varphi)\,,
\ee
where $\omega$ 
is a real valued function depending only on $(r,\theta)$.

There are two key features in the present form of $\bB$. On one hand, in suitable coordinates, the dependence on the angle $\varphi$ of the dynamics of other variables occurs only through the possible non axi-symmetry of the electric field $\bE$. On the other hand, independently of any further assumption on $\omega$, magnetic field lines are everywhere orthogonal to $\eDr$. In particular, introducing $\eDperp$ through
\be
\label{hyp:2}
\eDperp(r,\theta,\varphi)
\,=\,\sin(\omega)\,\eDp(\varphi) \,-\, \cos(\omega)\,\eDt(\theta,\varphi)\,,
\ee
one obtains $(\eDr,\eDperp,\eDpar)$ valued in direct orthonormal bases.

\begin{remark}
The analysis of \cite{FR:JEP} applies to an arbitrary geometry and, in principal, the corresponding numerical approach described here could be used in such a generality. Yet, in full generality, some of the involved computations turn out to be rather cumbersome. We use axi-symmetry and the orthogonality of $\eDr$ and $\bB$ to simplify the latter. We would like to point out a simple way to relax the latter assumption so as to include concrete applications without dramatically increasing computational complexity. On one hand, at a similar price, one may replace $(r,\theta)$ with another set of two-dimensional coordinates $(\psi,\chi)$ generating an orthonormal frame $(\beD_\psi,\beD_\chi,\eDp)$ such that $\beD_\psi$ and $\bB$ are orthogonal. This allows to treat geometries that are axi-symmetric toroidal-like.  
\end{remark}

Therefore, since the leading-order dynamics is a rotation of the velocity in the plane orthogonal to $\eDpar$, it is convenient to use for $\bv$ a frame adapted to $\eDpar$, and we shall use $(\eDr,\eDperp,\eDpar)$ to provide such a frame. Accordingly we introduce
\[
v_\alpha \,=\, \langle \bv,\beD_\alpha\rangle,
\qquad
E_\alpha \,=\, \langle \bE,\beD_\alpha\rangle, \quad{\rm for} \quad\alpha \in\{r,\perp, \mypar\}
\]
and
\begin{align*}
\bvperp&=(\vr,\vperp)\,,&\bEperp&=(\Er,\Eperp)\,.
\end{align*}
Therefore from the definition of $(\eDpar,\eDperp)$ in \eqref{hyp:1} and \eqref{hyp:2}, we notice that
\be
\label{def:0}
\left\{\ba{l}
\vpar \,=\,\cos(\omega)\,v_\varphi\,+\,\sin(\omega)\,v_\theta,
\\[0.9em]
\vperp \,=\,\sin(\omega)\,v_\varphi\,-\,\cos(\omega)\,v_\theta
\ea\right.
\ee
and conversely,
$$
\left\{\ba{l}
v_\varphi \,=\,\cos(\omega)\,\vpar\,+\,\sin(\omega)\,\vperp,
\\[0.9em]
v_\theta \,=\,\sin(\omega)\,\vpar\,-\,\cos(\omega)\,\vperp.
\ea\right.
$$
Hence we may now write  System~\eqref{eq:xv} in terms of $(r,\theta,\varphi,\bvperp,\vpar)$ as, 
\begin{align}
  \label{eq:tor-rtp1}
\frac{\dD}{\dD t}\bp r\\[0.9em] \varphi\\[0.9em] \theta \ep
&\,=\,
\bp
v_r
\\[0.9em]
\ds\frac{\cos(\omega)\,\vpar \,+\,\sin(\omega) \,\vperp }{R}
\\[0.9em]
\ds\frac{\sin(\omega)\,\vpar\,-\,\cos(\omega) \,\vperp}{r}
\ep\,,
\end{align}
whereas the equation for $\vpar$ is obtained using \eqref{eq:xv2} and \eqref{def:0},
\begin{eqnarray*}
  \frac{\dD \vpar}{\dD t} &=&  - \vperp \,\left(\partial_r\omega \,v_r\,+  \frac{\partial_\theta \omega}{r} \,v_\theta \right) \,+\,  \cos(\omega) \frac{\dD v_\varphi}{\dD t} \,+\,\sin(\omega)\, \frac{\dD v_\theta}{\dD t}
\\
&=& - \vperp \,\left( \partial_r\omega \,v_r\,+  \frac{\partial_\theta \omega}{r} \,v_\theta\right) \,+\, \Epar \,-\,\frac{\sin(\theta)}{R}\, v_\varphi \,v_\perp
\\
&&-\,v_r \,\left(\frac{\cos(\theta)}{R}\,\cos(\omega)\, v_\varphi \,+\,\frac{1}{r}\,\sin(\omega)\,v_\theta\right).
\end{eqnarray*}
Then replacing $(v_\theta,v_\varphi)$ with their expression with respect to $(\vpar,\vperp)$ yields
\be
\label{eq:vpar}
\left\{
\ba{l}
\ds\frac{\dD \vpar}{\dD t} \,=\, \Epar + \Fpar, 
 \\[0.9em]
\ds
\textrm{with }\qquad\Fpar \,:=\, \gamma \,\vpar\,v_r \,+\, \alpha\,\vpar\,\vperp \,+\, \left(\delta-\partial_r\omega\right)\,\vperp\,v_r \,+\, \beta \;\vperp^2
\ea\right.
\ee
 where the parameters $\alpha$, $\beta$, $\gamma$ and $\delta$ depend
 only on  $(r,\theta)$,
 \be
 \left\{
 \ba{l}
 \ds\alpha \,=\,-\frac{\partial_\theta\omega}{r} \,\sin(\omega) \,-\,  \frac{\sin(\theta)}{R}\,\cos(\omega)\,,
 \\[0.9em]
 \ds\beta \,=\,\frac{\partial_\theta\omega}{r} \,\cos(\omega) \,-\,  \frac{\sin(\theta)}{R}\,\sin(\omega)\,,
\\[0.9em]
 \ds\gamma \,=\, -\frac{\sin^2(\omega)}{r}  \,-\, \frac{\cos(\theta)}{R} \,\cos^2(\omega)\,,
 \\[0.9em]
 \ds\delta \,=\,-\left(\frac{\cos(\theta)}{R}\,-\,  \frac{1}{r}\right)\,\sin(\omega)\cos(\omega)\,.
 \ea\right.
\label{def:alpha_et_al}
\ee
Likewise, we obtain an equation for $\bvperp=(v_r,\vperp)$ as
\be
\label{eq:vperp}
\left\{\ba{l}
\ds\frac{\dD v_r}{\dD t} \,=\, E_r \,+\, \Cr \,+\,\frac{b\,\vperp}{\eps}\,,
 \\[0.9em]
\ds\frac{\dD \vperp}{\dD t} \,=\, \Eperp \,+\, \Cperp
\,-\, \frac{b\,v_r}{\eps}\,, 
  \ea\right.
 \ee
  where $\bCperp=(\Cr,\Cperp)$,
 $$
 \left\{
 \ba{l}
\ds\Cr \,:=\,  -\gamma\, \vpar^2\,-\,2\,\delta\,\vpar\,\vperp\,+\, \zeta\,\vperp^2\,,
\\[0.9em]
\ds\Cperp \,:=\,  -\alpha\, \vpar^2\,+\,\left(\partial_r\omega + \delta \right)\,\vpar\,v_r \,-\, \beta\,\vpar\,\vperp\,-\, \zeta\,v_r\,\vperp
\ea\right.
$$
and
\be
\label{def:zeta}
 \zeta \,=\, \frac{\cos^2(\omega)}{r} \,+\, \frac{\cos(\theta)}{R}\,\sin^2(\omega)\,.
\ee
This formulation allows to split the parallel and perpendicular
directions with respect to the magnetic field. In particular, we get that at leading-order $\bvperp$ is oscillating at a
frequency of order $1/\eps$ whereas $(r,\theta,\varphi,\vpar)$ exhibit
a slower dynamics. Furthermore, it allows to identify another slow
variable $\|\bvperp\|^2$, hence we introduce the new variable $b\mu
\,:=\, {\|\bvperp\|^2}/{2}$ and  write its slow dynamic from
\eqref{eq:vperp}  as
\begin{align}
  \label{eq:bmu}
\frac{\dD b\mu}{\dD t}\,=\, \langle\bEperp,\,\bvperp\rangle \,-\,\vpar\,\Fpar\,.
\end{align}
Note that our convoluted notation $b\,\mu$ is chosen to respect standard notational conventions of the physical literature devoted to gyro-kinetic reductions; see the related discussion in \cite{FR:JEP}.

Our final step is motivated by the simple observation that when deriving asymptotic models the stiff part of \eqref{eq:vperp} is used to replace, in the equations of slower components, $\bvperp$ with slower or lower-order terms. To prepare such eliminations, it is convenient to divide System \eqref{eq:vperp} by $b$. This hints at the introduction of the new variable
\[
\buperp=\frac{\bvperp}{b}.
\]
As a side effect, this will exhibit the effects of the gradients of the intensity of the
magnetic field $b$ on the reduced asymptotic dynamics. With $\buperp$ in hands, we observe that \eqref{eq:vperp} may be replaced with
\be
\label{eq:uperp}
\left\{\ba{l}
\ds\eps\,\frac{\dD u_r}{\dD t} \,=\, \eps\,\frac{E_r + \Cr + \Gr}{b} \,+\,\vperp\,,
 \\[0.9em]
\ds\eps\,\frac{\dD \uperp}{\dD t} \,=\, \eps\,\frac{\Eperp +\Cperp + \Gperp}{b}
\,-\, v_r\,, 
  \ea\right.
 \ee
where $\bGperp =(\Gr,\Gperp)$ is given by
 $$
 \left\{
 \ba{l}
\ds\Gr \, := \, \eta\, v_r\,\vpar \,+\, \kappa\,v_r\,\vperp \,+\, \lambda\,v_r^2,
\\[0.9em]
\ds\Gperp \, := \,  \eta\,\vpar\,\vperp\,+\, \lambda\,v_r\,\vperp  \,+\, \kappa\,\vperp^2
\ea\right.
$$
and $(\alpha,\beta,\gamma,\delta)$ is defined in
\eqref{def:alpha_et_al} and the triplet $(\eta,\kappa,\lambda)$ is 
\be
\left\{
 \begin{array}{l}
\ds\eta \,=\,  -\frac{\sin(\omega)}{r}\, \frac{\partial_\theta b}{b},
\\[0.9em]
 \ds\kappa \,=\, \frac{\cos(\omega)}{r}\, \frac{\partial_\theta b}{b},
   \\[0.9em]
\ds\lambda \,=\, -\frac{\partial_r b}{b}.
 \end{array}\right.
 \label{def:eta_kappa}
 \ee
 
In \eqref{eq:uperp}, $\bEperp$ takes into account the effects of the electric field, whereas $\bCperp$ describes the curvature effects and $\bGperp$ those of the gradient of the magnetic field intensity. Let us observe that the structure of \eqref{eq:tor-rtp1}-\eqref{def:eta_kappa} strongly echoes the one of the two-dimensional inhomogeneous situation dealt with in \cite{FR2}, with $\bvperp$ playing the role of the two-dimensional velocity and $(r,\theta,\varphi,\vpar, b\,\mu)$ playing the role of the spatial position and the microscopic kinetic energy.
 
To prepare further manipulations, we write \eqref{eq:uperp} in a more compact form as
\be
\label{def:uperp2}
\bvperp \,=\, \eps\,\bUperp \,\,-\,\eps
\frac{\dD\,}{\dD t}\left(\bJ_0\buperp\right)\,,
\ee
where
\[
\bJ_0\,:=\,\bp 0&1\\-1&0 \ep\,,
\]
and
\be
\label{def:Uperp}
\bUperp:= \frac{\bJ_0}{b}\left(\bEperp \,+\, \bCperp \,+\, \bGperp\right).
\ee 
 
 \section{Asymptotic dynamics}
 \label{sec:2bis}
 
 We now consider the asymptotic regime $\eps\ll 1$. When doing so, our focus is two-fold. On one hand, we are interested in the identification of the asymptotic reduction by itself, so as to know what are the objectives for the numerical schemes introduced below. On the other hand, we are also interested in unraveling the algebraic identities that supports the asymptotic reduction, since a convenient way to ensure that the asymptotic reduction takes also place at the discrete level is precisely to enforce discrete counterparts to such identities in the numerical schemes.
 
 We could adapt or apply the arguments of \cite{FR:JEP} and obtain a fully rigorous mathematical analysis. Yet this would lead us too far beyond the main scope of the present paper. Instead, we borrow mostly the algebraic part of \cite{FR:JEP}.
 
\subsection{First order asymptotics}

To begin with, as in \cite[Lemmas~3.3~\&~4.3]{FR:JEP}, we observe that Equation~\eqref{def:uperp2} may be combined with slower equations to eliminate not only $\bvperp$ but any expression linear in $\bvperp$ with dependence on time and slow variables. Indeed if $t\mapsto \cL(t)$ is a smooth map valued in a space of linear operators, there exist two functions $(\chi_\bL,\eta_\bL)\in L^\infty_t$ such that the equations of motion imply
\be
\label{linear:est}
\bL\,\bvperp \,=\, -\eps\,\frac{\dD \chi_\bL}{\dD t}\,+\,\eps\, \eta_\bL\,.
\ee
In particular, from this one expects that $\bL\,\bvperp $ converge to zero
when $\eps$ tends to zero in some sense\footnote{To be precise, in the $W^{-1,\infty}$ topology for functions of the time variable. Let us stress again that most of nonlinear transformations are not continuous for such a weak topology.}. The underlying algebra is completely constructive but we refrain from giving it wherever it is not strictly needed. 

Now we show how to apply the foregoing first principles 
to the study of the asymptotic dynamics of the slow variables
$(r,\theta,\varphi,\vpar,b\mu)$.

The dependence of System~\eqref{eq:tor-rtp1} on $\bvperp$ is linear so that the foregoing argument is sufficient to reveal the first order part of it. We begin with the study of the time derivative of the variable $r$. Here the algebraic manipulation consists in using the second equation of \eqref{eq:uperp} so as to eliminate the right hand side $\vr$ in \eqref{eq:tor-rtp1}. In this way, one derives 
\[
\frac{\dD\,}{\dD t}\left(r+\eps\,\frac{\vperp}{b}\right)
\,=\,\eps\,\frac{\Eperp+\Cperp}{b}
+\eps\,\left(\vr\frac{\d_rb}{b}+\left(\vperp\cos(\omega)+\vpar\sin(\omega)\right)\,\frac{\d_\theta b}{b\,r}\right)\,\frac{\vperp}{b},
\]
which suggests that $r$ converges to the constant $r(0)$ when $\eps$ tends to
zero. We proceed in the same manner for the angle variables
$(\varphi, \theta)$ using the property \eqref{linear:est} to study the
asymptotic behavior of  linear terms with respect to $\bvperp$. It
gives that, for some $(\chi_\varphi,\eta_\varphi)$ and
 $(\chi_\theta,\eta_\theta)$,
\[
\frac{\dD}{\dD t}\left(\varphi \,+\, \eps\,\chi_{\varphi} \right)
\,=\, \frac{\cos(\omega)\,\vpar}{R} \,+\,
\eps\, \eta_{\varphi}
\]
and
\[
\frac{\dD}{\dD t}\left(\theta \,+\, \eps\,\chi_{\theta} \right)
 \,=\, \frac{\sin(\omega)\,\vpar}{r}  \,+\,
 \eps\, \eta_{\theta},
 \]
 allowing to characterize the limit equation on $(\varphi,\theta)$.
 when $\eps\rightarrow 0$.
  
Note that the equation for $\vpar$ given in \eqref{eq:vpar} contains both linear and quadratic terms in $\bvperp$. Hence, to go on, we need
to determine which quadratic terms may be eliminated. 
Proceeding as in \cite[Lemmas~3.8 and~4.4]{FR:JEP}, we extract slow components from
expressions quadratic in $\bvperp$. As a result, for any smooth map $\bA$ valued in linear maps on $\RR^2$, there exist two functions $(\chi_\bA,\eta_\bA)$ such that from the equations of motion stem
\be
\label{quad:est}
\langle\bA\,\bvperp,\,\bvperp\rangle \,=\, b\,\mu \,\Tr(\bA) \,-\,\eps\frac{\dD \chi_\bA}{\dD t}\,+\,\eps\, \eta_\bA,
\ee
where $b\mu$ is given by \eqref{eq:bmu}. Note that only trace-free quadratic expressions become negligible.

Consistently, we split  the right hand side of \eqref{eq:vpar} into its slow component and a trace-free expression, as
\[
\beta\,\vperp^2 \,=\, \beta\,\left[ b\mu  \,+\, \frac{1}{2} \left( \vperp^2 - v_r^2\right)  \right]\,.
\]
We then apply the linear and quadratic abstract eliminations with respectively
\[
\bL \,=\, \vpar\,\left(\gamma, \,\alpha \right) \qquad\textrm{ and }\qquad
\bA \,=\, \left( \ba{ll} -\beta/2 & \delta-\partial_r\omega \\ 0 & \beta/2
  \ea \right).
\]
Since $\Tr(\bA)=0$, there exists $(\chi_{\vpar},\eta_{\vpar})$ such that
\[
\frac{\dD}{\dD t}\left(\vpar \,+\, \eps\,\chi_{\vpar} \right) \,=\, \Epar +
\beta\,b\,\mu \,+\, \eps\, \eta_{\vpar}.
\]
We treat the time derivatives of  $b\,\mu$ in the same
manner and get a right hand side with a zero-th order term with
respect to $\eps$ and a correction of order $\eps$
\[
\frac{\dD}{\dD t}\left(b\mu \,+\, \eps\,\chi_{b\mu} \right) \,=\, -\vpar\,\beta\,b\,\mu \,+\, \eps\, \eta_{b\mu}.
\]

Gathering the latter results we receive
\[
\left\{
  \ba{l}
 \ds\frac{\dD\,}{\dD t}\left(r\,+\,\eps\,\frac{\vperp}{b}\right)
 \,=\, \eps\,\frac{\Eperp+\Cperp}{b}
+\eps\,\left(\vr\frac{\d_rb}{b}+\left(\vperp\cos(\omega)+\vpar\sin(\omega)\right)\,\frac{\d_\theta b}{b\,r}\right)\,\frac{\vperp}{b},
 \\[0.9em]
\ds\frac{\dD}{\dD t}\left(\varphi \,+\, \eps\,\chi_{\varphi} \right)
\,=\, \frac{\cos(\omega)\,\vpar}{R} \,+\,
\eps\, \eta_{\varphi},
\\[0.9em]
 \ds\frac{\dD}{\dD t}\left(\theta \,+\, \eps\,\chi_{\theta} \right)
 \,=\, \frac{\sin(\omega)\,\vpar}{r}  \,+\,
 \eps\, \eta_{\theta},
 \\[0.9em]
 \ds\frac{\dD}{\dD t}\left(\vpar \,+\, \eps\,\chi_{\vpar} \right) \,=\, \Epar +
 \beta\,b\,\mu \,+\, \eps\, \eta_{\vpar},
 \\[0.9em]
 \ds\frac{\dD}{\dD t}\left(b\mu \,+\, \eps\,\chi_{b\mu} \right) \,=\, -\vpar\,\beta\,b\,\mu \,+\, \eps\, \eta_{b\mu},
 \ea\right.
\]
where for any $s\in\{r,\,\varphi,\,\theta,\,\vpar,\,b\,\mu\}$, the
functions  $\eta_s$ and $\chi_s$ can be computed explicitly and do depend on $(v_r,\vperp)$. When the four first equations of the latter system are supplemented with System~\eqref{eq:vperp} on $(v_r,\vperp)$, it is of course  equivalent to  the initial one given by \eqref{eq:xv2} and the last equation of System~\eqref{eq:vperp} comes as a consequence. But System~\eqref{eq:vperp} is also well adapted to capture the leading order terms with respect to $\eps$. 

After these algebraic manipulations, applying the analytic arguments \cite{FR:JEP} indeed proves that solutions to the following closed\footnote{There is no dependence on  $(v_r,\vperp)$ anymore.} system
\be
\label{order:1}
\left\{
  \ba{l}
 \ds\frac{\dD r}{\dD t}
\,=\,0,
 \\[0.9em]
\ds\frac{\dD \varphi}{\dD t} \,=\, \frac{\cos(\omega)\,\vpar}{R},
\\[0.9em]
 \ds\frac{\dD \theta}{\dD t} \,=\, \frac{\sin(\omega)\,\vpar}{r},
 \\[0.9em]
 \ds\frac{\dD \vpar}{\dD t} \,=\, \Epar +
 \beta\,b\,\mu,
 \\[0.9em]
 \ds\frac{\dD b\mu}{\dD t} \,=\, -\vpar\,\beta\,b\,\mu,
 \ea\right.
\ee
provide an approximation of the slow variables up to $\cO(\eps)$ errors.

\medskip

We list now a few properties of the first-order asymptotic system. Obviously $r$ is constant along the flow. 
A few more conservations may be obtained if one assumes classical extra structure on electromagnetic fields.

\begin{proposition}
\label{prop:cons1}
Assume that $\bE=-\nabla_\bx\phi$ with $\phi$ not depending on time and that the confining magnetic field satisfies the Gauss law 
\[
\Div_\bx\bB\,=\, 0.
\]
Then solutions to the asymptotic model \eqref{order:1} satisfy
\begin{itemize}
\item the conservation of energy
\[
\frac{\dD}{\dD t} \left(\frac{\vpar^2}{2}+b\mu +\phi\right)   \,=\,0\,;
\]
\item the conservation of the classical adiabatic invariant 
\[
\frac{\dD\mu}{\dD t} \,=\, 0.
\]
\end{itemize} 
\end{proposition}

We stress that the conservation of energy already holds for the original equations of motion whereas conservations of $r$ and $\mu$ hold only for the first-order asymptotic model.

\begin{proof}
Let us suppose that the electric field $\bE$ derives from a potential $\phi$, that is, $\bE=-\nabla\phi$, 
the corresponding balance law for the total energy of the asymptotic model is
\begin{eqnarray*}
\frac{\dD}{\dD t} \left(\frac{\vpar^2}{2}\,+\,b\,\mu\,+\,\phi\right)
  &=&  \Epar\,\vpar \,+\,\left( \d_t\phi  \,+\,\partial_\varphi \phi\,
    \frac{\dD \varphi}{\dD t}  + \,+\,\partial_\theta \phi\,
      \frac{\dD \theta}{\dD t} \right).
\end{eqnarray*}
Then, we observe that
\[
 \Epar\,=\, \langle -\nabla\phi, \eDpar\rangle \,=\, - \cos(\omega)\,
 \frac{\partial_\varphi \phi}{R}  - \sin(\omega)\,  \frac{\partial_\theta \phi}{r}
 \]
 and, by using the equations on $(\varphi,\theta)$ from \eqref{order:1}, that the
 energy balance law of energy  reduces to the claimed conservation law
 when $\d_t\phi\equiv0$
 \[
\frac{\dD}{\dD t} \left(\frac{\vpar^2}{2}\,+\,b\,\mu\,+\,\phi\right)
  \,=\, 0. 
\]

Note moreover that from the axi-symmetric assumption
\eqref{hyp:0}-\eqref{hyp:1} on the magnetic
field follows that the Gauss law is
\[
0\,=\,\Div_\bx\bB\,=\, \frac{1}{R\,r} \partial_\theta\left( R\,B_\theta\right)
\,=\, \frac{\partial_\theta B_\theta}{r} \,-\
\frac{B_\theta\,\sin\theta}{R} \,,
\]
which in terms of $(b,\omega)$ is written as
\[
\frac{\sin(\omega) }{r}\,\partial_\theta b \,+\,
b\,\left(\frac{\cos(\omega)}{r}\, \partial_\theta \omega   \,-\,
  \frac{\sin(\omega)}{R}\,\sin(\theta)\right)  \,=\,0\,.
\]
Hence, the Gauss law implies that solutions to the asymptotic model \eqref{order:1} satisfy
\begin{eqnarray*}
\frac{\dD \mu}{\dD t} 
  &=& -\left(\frac{\partial_\theta\omega}{r} \,\cos(\omega) \,-\,
      \frac{\sin(\theta)}{R}\,\sin(\omega) \right)\,\vpar \,\mu  \,-\, \frac{\mu}{b}\,
      \partial_\theta b\,\frac{\dD \theta}{\dD t} 
      \\
  &=& -\left(\frac{\partial_\theta b}{r}\,\sin(\omega) \,+\, \frac{\partial_\theta\omega}{r} \,b\,\cos(\omega) \,-\,
      \frac{\sin(\theta)}{R}\,b\,\sin(\omega) \right)\,\frac{\vpar
      \,\mu}{b}  
 \,=\,0\,,
\end{eqnarray*}
which is indeed a conservation law.
\end{proof}


\subsection{Second order corrections}
The previous asymptotics is not completely  satisfactory since it does not take into account 
drifts in the $r$ direction, which is perpendicular to the magnetic field lines. Capturing such drifts is key to the computational examination of confinement properties of fusion devices and, in \cite{FR1,FR2,FR3}, this is precisely the numerical computation of the corresponding perpendicular second-order drifts that is enforced by designing asymptotic preserving schemes.

The strategy sketched in the foregoing section, and fully worked out in \cite{FR:JEP}, may be pursued one step further so as to derive a second order system encoding the dynamics of all the slow variables up to $\cO(\eps^2)$ errors. In particular, this includes a leading order description of spatial trajectories in all directions. Indeed, such a task is carried out in \cite{FR:JEP} for arbitrary geometries. See \cite[Figure~2.2]{FR:JEP} and the surrounding discussion for an illustration of the dramatic effects of including such second order corrections on the correct prediction of the shape of spatial trajectories.

Yet for our practical purposes, we anticipate that the specialization of \cite{FR:JEP} to our current coordinates would result in cumbersome formula. Instead, we focus on providing second order descriptions only for variables that are so slow that the first order model \eqref{order:1} cannot capture their leading order dynamics, that is, only for $r$ and $\mu$, or more precisely for relevant first order corrections of those. To achieve this goal, one could follow the same process used to derive System~\eqref{order:1} and obtain 
\[
\frac{\dD\,}{\dD t}\left(r+\eps\,\frac{\vperp}{b}+\eps^2\,\chi_r\right)
\,=\,\eps\,\frac{\Eperp-\alpha\,\vpar^2+\kappa\,b\mu}{b}\,+\,\eps^2\,\eta_r,
\]
for some $(\chi_r,\eta_r)$. In terms of asymptotic reduced model, this suggests to either replace the first equation of \eqref{order:1} with 
\[
\frac{\dD\,r}{\dD t}
\,=\,\eps\,\frac{\Eperp-\alpha\,\vpar^2+\kappa\,b\mu}{b}\,,
\]
or to first solve \eqref{order:1} and then solve the latter equation as a slaved equation to determine a corrected dynamics for $r+\eps\,\vperp/b$ (identified with the $r$ of the reduced model). A similar treatment could be applied to $\mu$.

However, we find a slightly different derivation of the second order corrections for spatial positions to be better suited to our numerical purposes. Thus we now provide some details on it. The starting point is that \eqref{def:uperp2} may be combined with
\[
\frac{\dD\,\bx}{\dD t}
\,=\,\vpar\,\eDpar+\vr\,\eDr+\vperp\,\eDperp\,,
\]
to yield 
\[
\frac{\dD\,}{\dD t}\left(\bx+\eps\,(\uperp\eDr-u_r\eDperp)\right)
\,=\,\vpar\,\eDpar\,+\,\eps\,(\Ur\,\eDr+\Uperp\,\eDperp)\,+\,\eps\,\bR\,,
\]
with $\bUperp=(\Ur,\Uperp)$ explicitly defined in \eqref{def:Uperp} and
\begin{align*}
\bR&:=\uperp\ (\bv\cdot\nabla_{\bx})\eDr\,-\,u_r\ (\bv\cdot\nabla_{\bx})\eDperp\,,
\end{align*}
which can be decomposed as $\bR = \bR_1 + \bR_2$ with
\[
\left\{
  \begin{array}{l}
    \ds\bR_1  \,=\, \bR \,-\, \left(\uperp\,\vperp\
    (\eDperp\cdot\nabla_{\bx})\eDr-u_r\,v_r\
    (\eDr\cdot\nabla_{\bx})\eDperp\right)
    \\[1.1em]
    \ds \bR_2 \,=\,\left(\uperp\,\vperp\ (\eDperp\cdot\nabla_{\bx})\eDr-u_r\,v_r\ (\eDr\cdot\nabla_{\bx})\eDperp\right).
    \end{array}\right.
\]
Let us recall that expressions that, as functions of $\bvperp$, or equivalently of $\buperp$, are either linear or quadratic and trace-free may be eliminated at main order. The term $\bR_1$ fits directly in this category whereas
\begin{align*}
\bR_2 &=-\eDpar\,\left(u_r\,v_r\,\d_r\omega+\uperp\,\vperp\,\delta\right)
+\eDperp\,\uperp\,\vperp\,\zeta
\end{align*} 
thus differs from $\mu\,(\d_r\omega+\delta)\,\eDpar+b\,\uperp^2\,\delta\,\eDperp$ by a trace-free quadratic term. Therefore for some $(\chi_{\bx},\eta_{\bx})$,
\[
\frac{\dD\,}{\dD t}\left(\bx+\eps\,(\uperp\eDr-u_r\eDperp)
+\eps^2\,\chi_{\bx}\right)
\,=\,(\vpar+\eps\,\mu\,(\d_r\omega+\delta))\,\eDpar+\eps\,(\Ur\,\eDr+(\Uperp+b\,\uperp^2\,\delta)\,\eDperp)+\eps^2\,\eta_{\bx}\,
\]
so that now we only need to extract the leading order contributions of $\Ur\,\eDr+(\Uperp+b\,\uperp^2\,\delta)\,\eDperp$. By discarding again linear and trace-free quadratic terms, on the rewriting
\[
\left\{
 \ba{ll}
\ds\Ur \,= & \ds\frac{\Eperp}{b} \,-\, \frac{\alpha}{b}\,\vpar^2\,+\, \left(\partial_r\omega +\delta \right)\,\vpar\,u_r \,+\, (\eta-\beta)\,\vpar\,\uperp\,+\,b\,\left(\lambda-
  \zeta\right)\,u_r\,\uperp \,+\, \kappa\,\left(\mu +  b\,\frac{\uperp^2-u_r^2}{2}\right),
 \\[0.9em]
\ds\Uperp \,= & \ds-\frac{\Er}{b} + 2\,\delta\,\vpar\uperp\,-\,
\zeta\,b\,\uperp^2 \,+\, \frac{\gamma}{b}\,\vpar^2 \,-\,\eta\, u_r\,\vpar \,-\, \kappa\,u_r\,\uperp
\,-\, \lambda\,\left(\mu -  b\,\frac{\uperp^2-u_r^2}{2}\right),
\ea\right.
\]
obtained from
\begin{align*}
u_r^2&\,=\,\frac{\mu}{b}-\frac{\uperp^2-u_r^2}{2}\,,&
\uperp^2&\,=\,\frac{\mu}{b}+\frac{\uperp^2-u_r^2}{2}\,,
\end{align*}
we see that the expected main contribution of $(\Ur\,\eDr+(\Uperp+b\,\uperp^2\,\delta)\,\eDperp)$ in the regime $\eps\ll 1$ is $\overline{U}_r\,\eDr+\overline{U}_\perp\,\eDperp$ with 
\begin{align*}
\overline{U}_r&:=\Eperp\,-\, {\alpha}\,\vpar^2 \,+\, \kappa\,b\mu\,,&
\overline{U}_\perp&:=-\Er\,+\, {\gamma}\,\vpar^2\,-\, \lambda\,b\mu\,.
\end{align*}
More explicitly, for some $(\chi_{\bx},\eta_{\bx})$ (different form the above ones)
\[
\frac{\dD\,}{\dD t}\left(\bx+\eps\,(\uperp\eDr-u_r\eDperp)
+\eps^2\,\chi_{\bx}\right)
\,=\,(\vpar+\eps\,\mu\,(\d_r\omega+\delta))\,\eDpar+\eps\,(\overline{U}_r\,\eDr+\overline{U}_\perp\,\eDperp)+\eps^2\,\eta_{\bx}\,.
\]
This recovers in particular the expression for the main contribution in the direction $r$,
\[
\frac{\dD\,}{\dD t}\left(r+\eps\,\frac{\vperp}{b}+\eps^2\,\chi_r\right)
\,=\,\eps\,\overline{U}_r\,+\,\eps^2\,\eta_r\,.
\]

The main upshot of the latter considerations is that we want to enforce that, in a suitable sense, $(\vr\eDr+\vperp\eDperp)/\eps$ converges to a vector of the form
\begin{equation}\label{extra}
-\frac{\dD\,}{\dD t}\left(\uperp\eDr-u_r\eDperp\right)\,+\,*\,\eDpar\,+\,\overline{U}_r\,\eDr\,+\,\overline{U}_\perp\,\eDperp
\end{equation}
 with $(\overline{U}_r,\overline{U}_\perp)$ as above, and $*$ some irrelevant coefficient that we do not try to recover accurately at the numerical level (when using coarse meshes). Note that focusing only on a better reconstruction of $r$ would only require to capture the shape 
 \[
-\frac{\dD\,}{\dD t}\left(\uperp\eDr\,+\,*\eDperp\right)\,+\,**\,\eDpar\,+\,\overline{U}_r\,\eDr\,+\,***\,\eDperp
\]
for some $*$, $**$ and $***$. The approach we choose to implement at the numerical level, that enforces the convergence to \eqref{extra}, improves the capture of drifts in all directions perpendicular to $\bB$, hence is expected to be more robust to coordinatization. Numerical comparisons, not reproduced here, show that this choice brings a dramatic improvement in numerical accuracy.

\begin{remark}
For the sake of comparison with both \cite{FR:JEP} and the classical gyro-kinetic theory, let us comment on the kind of drifts arising from $\bUperp$ (defined in \eqref{def:Uperp}). One immediately identifies the electric drift $\EcB$ given by
\[
\EcB \ds\,:=\, \frac{\bJ_0\,\bEperp}{b}\,=\,\frac{\bE \wedge \bB}{\|\bB\|^2}\,.
\]
As already mentioned, the term
\[
\curvB\,:=\,
\frac{\bJ_0\bCperp}{b}
\]
contains curvature effects, whereas, applying the uncoupling strategy of \cite{FR:JEP}, expounded in the foregoing subsection, to the study of 
\[
\gradB \,:=\, \frac{\bJ_0\,\bGperp}{b}\,.
\]
suggests that it converges, in a suitable sense, to
\[ \frac{\mu}{b}\begin{pmatrix}\ds
  \frac{\cos(\omega)}{r}\partial_\theta b \\[0.9em] \partial_r b \end{pmatrix} 
  \,=\,\frac{\mu}{b^2}\,\nabla b\wedge \bB
\]
when $\eps$ goes to zero.
\end{remark}

As already mentioned we could perform a similar analysis for $\mu$. Yet, firstly we believe that it is less physically significant than the capture of drifts; secondly since we are not enforcing a reconstruction of other variables up to $\cO(\eps^2)$ errors this would be hardly compatible with the conservation of a total energy $\vpar^2/2+b\mu +\phi$ --- even up to $\cO(\eps^2)$ errors --- when $\bE$ derives from a potential $\phi$. Implicitly, we choose here to prioritize the latter.

\section{A particle method for axi-symmetric strongly magnetized plasmas}
\setcounter{equation}{0}
\label{sec:3}

We now turn to the introduction of numerical schemes for the particle evolution. We refer the reader to \cite{FR1} for a brief description and thorougher references on how this fits in a complete PIC code.

As in \cite{BFR:15,FR1,FR2,FR3}, we want to apply semi-implicit schemes. This requires a preliminary identification of stiff and slow terms in the evolution. As in \cite{FR2}, we apply this strategy to an augmented formulation of the original equations of motion \eqref{eq:xv}, where, at the discrete level, the evolutions of $b\,\mu=\|\bvperp\|^2/2$ and $\bvperp$ are allowed to be uncoupled when the oscillations of $\bvperp$ are too fast to be captured by the coarse time discretization. 

Our first task is to design a suitable augmented formulation. However a natural choice for the latter stems readily from considerations of the foregoing section. Incidentally, we point out that many other choices would do a reasonable job and that the choice made in \cite{FR2} was indeed much more artificial. 

\subsection{Augmented formulation}

With variables $\bZ =(r,\varphi,\theta,\vpar,b\,\mu)$ and $\buperp=(u_r,\uperp)$ we consider
\be
\label{eq:1}
\left\{
  \ba{rl}
\ds\frac{\dD \bZ}{\dD t} &\ds\,=\, \bF(t,\bZ,\buperp),
\\[0.9em]
\ds
\frac{\dD }{\dD t}\left(\bJ_0\buperp\right)& \ds\,=\, \bU_\perp(t,\bZ,\buperp) \,-\, \frac{b(\bZ)\,\buperp}{\eps}, 
\ea\right.\ee
where  $\bF$ is given by
\begin{align}
\label{eq:num1}
\bF(t,\bZ,\buperp)
&\,:=\,
\bp
b\,u_r
\\[0.9em]
\ds\frac{\cos(\omega)\,\vpar \,+\,b\sin(\omega) \,\uperp}{R}
\\[0.9em]
\ds\frac{\sin(\omega)\,\vpar\,-\, b\cos(\omega) \,\uperp }{r}
\\[0.9em]
\ds \Epar + \Fpar
\\[0.9em]
\ds  \,-\, \vpar\,\Fpar \,+\,b\,\langle \bEperp,\,\buperp\rangle
  \ep\,,
\end{align}
with
\be
 \label{eq:num2}
\Fpar(t,\bZ,\buperp)\,:=\, b\,\left[\left(\gamma \,\,u_r \,+\, \alpha\,\uperp\right) \,\vpar\,+\, b\left(\delta-\partial_r\omega\right)\,\uperp\,u_r \,+\,
\beta \;\left( \mu +b\,\frac{\uperp^2 - u_r^2}{2} \right)\right],
\ee
and $\bUperp(t,\bZ,\buperp)=(\Ur,\Uperp)(t,\bZ,\buperp)$ is now given by
\be
\label{eq:num4}
 \left\{
 \ba{ll}
\ds\Ur \,:= & \ds\frac{\Eperp}{b} \,-\, \frac{\alpha}{b}\,\vpar^2\,+\, \left(\partial_r\omega +\delta \right)\,\vpar\,u_r \,+\, (\eta-\beta)\,\vpar\,\uperp\,+\,b\,\left(\lambda-
  \zeta\right)\,u_r\,\uperp 
\\[0.9em]
& \ds\,+\, \kappa\,\left(\mu +  b\,\frac{\uperp^2-u_r^2}{2}\right),
 \\[0.9em]
\ds\Uperp \,:= & \ds-\frac{\Er}{b} + 2\,\delta\,\vpar\uperp\,-\,
\zeta\,b\,\uperp^2 \,+\, \frac{\gamma}{b}\,\vpar^2 \,-\,\eta\, u_r\,\vpar \,-\, \kappa\,u_r\,\uperp
\\[0.9em]
& \ds\,-\, \lambda\,\left(\frac{b\mu}{b} -  b\,\frac{\uperp^2-u_r^2}{2}\right).
\ea\right.
\ee

System~\eqref{eq:1} is indeed an augmented formulation of \eqref{eq:xv} in the sense that if a solution to \eqref{eq:1} satisfies the constraint $(b\mu)/b^2=\|\uperp\|^2$ at some time it satisfies this constraint at any time and that in this case, on the associated constrained manifold, System~\ref{eq:1} reduces to \eqref{eq:xv} (in suitable coordinates). Yet a key point of our schemes, introduced below, is that they do not maintain the constraint in regimes where $\eps$ is much smaller than time steps, but instead they damp the oscillating $\uperp$ while keeping at order $1$ the slow $(b\mu)/b^2$. The latter provides consistency with the first order reduced model \eqref{order:1} since System~\eqref{order:1} is equivalently written as 
\[
\frac{\dD \bZ}{\dD t}\,=\, \bF(t,\bZ,0)\,.
\]

Before introducing numerical schemes for System~\eqref{eq:1}, in order to discuss second order properties, let us consider the effective second order system 
\be
\label{order:2}
\frac{\dD \bZ}{\dD t}\,=\, \bF\left(t,\bZ,\frac{\eps}{b}\,\overline{\bU}_\perp(t,\bZ)\right)
\ee
with $\overline{\bU}_\perp(t,\bZ)=(\overline{U}_r,\overline{U}_\perp)(t,\bZ)$ given by
\begin{align*}
\overline{U}_r&:=\Eperp\,-\, \alpha\,\vpar^2 \,+\, \kappa\,b\mu\,,&
\overline{U}_\perp&:=-\Er\,+\, \gamma\,\vpar^2\,-\, \lambda\,b\mu\,.
\end{align*}
Note that $\overline{\bU}_\perp(t,\bZ)=\bUperp(t,\bZ,0)$. For the sake of concreteness and to facilitate later computations, we point out that System~\eqref{order:2} is more explicitly written as
\[
\left\{
  \ba{l}
 \ds\frac{\dD r}{\dD t}
\,=\, \eps\,\overline{U}_r
 \\[0.9em]
\ds\frac{\dD \varphi}{\dD t} \,=\, \frac{\cos(\omega)\,\vpar \,+\, \eps\,\sin(\omega)\,\overline{U}_\perp}{R}\,,
\\[0.9em]
 \ds\frac{\dD \theta}{\dD t} \,=\, \frac{\sin(\omega)\,\vpar\,-\,\eps\, \cos(\omega)\,\overline{U}_\perp}{r}\,,
 \\[0.9em]
 \ds\frac{\dD \vpar}{\dD t} \,=\, \Epar \,+\,\beta\,b\mu
 \,+\,\eps\,\left(\gamma\,\overline{U}_r +\alpha\,\overline{U}_\perp    \right) \,\vpar
 \\[0.9em]
 \ds\frac{\dD b\mu}{\dD t} \,=\, 
 \,-\,\vpar\,\left(\beta\,b\mu
 \,+\,\eps\,\left(\gamma\,\overline{U}_r +\alpha\,\overline{U}_\perp
 \right) \,\vpar\right) \,+\, \eps\,\langle \bEperp, \,\overline{\bU}_\perp\rangle\,.
 \ea\right.
\]
We strongly emphasize that System~\eqref{order:2} does not provide a second order reduced dynamics for the original system \eqref{eq:xv} in the regime $\eps\ll 1$, but its suitable discretizations shall provide numerical solutions $\cO(\eps^2)$ close to solutions of our numerical schemes in the regime when $\eps$ is much smaller than discretization parameters. Roughly speaking, we use it to analyze the limit $\eps\to0$ of our schemes in the same way as standard effective ODEs or PDEs are used to analyze schemes in the limit when discretization steps go to zero.

With this in mind, let us observe that by design System~\eqref{order:2} reproduces accurately the expected drift in the $r$ direction. Moreover, though this was not explicitly part of our initial concern, it does preserve total energy, as shown in the following proposition. In the forthcoming subsections, we shall use comparison with System~\eqref{order:2} as a way to validate these properties. To be more precise on the latter, we shall prove below that solutions to our schemes of order $m$ are $\eps^2$-close to the solutions of a numerical scheme of order $m$ for System~\eqref{order:2}, which implies, together with the following proposition, that at the discrete level the total energy is conserved at least at order $\eps^2+(\Delta t)^m$ even with coarse meshes.

\begin{proposition}
\label{prop:cons2}
Assume that $\bE=-\nabla_\bx\phi$ with $\phi$ not depending on time.
Then solutions to the effective model \eqref{order:2} satisfy the conservation of energy
\[
\frac{\dD}{\dD t} \left(\frac{\vpar^2}{2}+b\mu +\phi\right)\,=\,0\,.
\]
\end{proposition}

\begin{proof}
We follow the lines of the proof of Proposition \ref{prop:cons1}
  with the additional terms of order $\eps$. The evolution of
  the total energy obeys
\begin{equation*}
\frac{\dD}{\dD t} \left(\frac{\vpar^2}{2}\,+\,b\,\mu\,+\,\phi\right)
\,=\,  \Epar\,\vpar \,+\,\eps\,\bEperp\cdot\overline{\bU}_\perp\,+\,\left( \d_t\phi  \,+\,\partial_r \phi
    \,\frac{\dD r}{\dD t} \,+\,\partial_\varphi \phi
    \,\frac{\dD \varphi}{\dD t}  
    \,+\,\partial_\theta \phi
      \,\frac{\dD \theta}{\dD t} \right).
\end{equation*}
Now using the equations on $(r,\varphi,\theta)$ from System~\eqref{order:2} and that
\begin{align*}
\Er&=-\d_r\phi\,,&
\Eperp&=\frac{\cos(\omega)}{r}\d_\theta\phi
-\frac{\sin(\omega)}{R}\d_\varphi\phi\,,&
 \Epar&\,=\, - \cos(\omega)\,
 \frac{\partial_\varphi \phi}{R}  - \sin(\omega)\,  \frac{\partial_\theta \phi}{r}\,,
\end{align*}
we get that
\begin{align*}
\frac{\dD}{\dD t} \left(\frac{\vpar^2}{2}\,+\,b\,\mu\,+\,\phi\right)
  &\,=\,
\d_t\phi+\eps\,\left(\bEperp\cdot\overline{\bU}_\perp
\,+\,\partial_r \phi\,\overline{U}_r
\,+\,\partial_\varphi \phi\,\frac{\sin(\omega)\,\overline{U}_\perp}{R} 
\,-\,\partial_\theta \phi\,\frac{\cos(\omega)\,\overline{U}_\perp}{r}\right)\,,
  \\
  & \,=\,\d_t\phi\,,
\end{align*}
which  reduces to the claimed conservation law when $\d_t\phi\equiv0$. 
\end{proof}

\medskip

We now describe various semi-implicit numerical schemes for System~\eqref{eq:1}.

\subsection{A first-order semi-implicit scheme}

We begin with the simplest semi-implicit scheme for \eqref{eq:1}, which is a combination of the backward and forward Euler schemes. For a fixed time step $\Delta t>0$ it is given by
\begin{subequations}\label{scheme:0}
\begin{empheq}[left=\empheqlbrace]{align}
&\ds\frac{\bZ^{n+1} - \bZ^n }{\Delta t} \,\,=\,  \bF\left(t^n, \bZ^n,\buperp^{n+1}\right),
\\[0.9em]
&\ds\bJ_0\frac{\buperp^{n+1} - \buperp^n }{\Delta t} \,\,=\,
  \bUperp(t^n,\bZ^n,\buperp^n) \,-\,  \frac{b(\bZ^n)\,\buperp^{n+1}}{\eps}\,.
  \label{scheme:0v}
\end{empheq}
\end{subequations}
Notice that only the second equation on $\buperp^{n+1}$ is really
implicit and it only requires the resolution of a two-dimensional
linear system. Then, once the value of $\buperp^{n+1}$ has been
computed the first equation provides explicitly the values of $\bZ^{n+1}$.

We recall that implicitly throughout the analysis we assume that fields are boundedly smooth and that $b$ does not vanish. Likewise we implicitly assume everywhere that $\eps$ is upper bounded, say by $1$.

\begin{proposition}[Consistency in the limit $\eps\rightarrow 0$ for a fixed $\Delta t$]
\label{prop:1}
Let us consider a time step $\Delta t>0$, a final time $T>0$  and
set $N_T:=\lfloor T/\Delta t\rfloor$. 
\begin{enumerate}[(i)]
\item Assume that $(\bZ^0(\eps),\buperp^0(\eps))$ is such that $\left(\bZ^0(\eps),\sqrt{\eps}\buperp^0(\eps)\right)_{\eps>0}$ converges in the limit $\eps\rightarrow 0$ to $(\bZ^0,0)$ for some $\bZ^0$. Consider $(\bZ^n(\eps),\buperp^n(\eps))_{0\leq
  n\leq N_T}$, the sequence obtained from  $(\bZ^0(\eps),\buperp^0(\eps))$  by \eqref{scheme:0}.\\
Then, for any $1\leq n\leq N_T$, $(\bZ^n(\eps),\buperp^n(\eps))_{\eps>0}$ converges to $(\bZ^n,0)$ as $\eps\rightarrow 0$ where
\begin{equation}
  \label{sch:y0}
  \left\{\ba{l}
\ds\frac{\bZ^{n+1} - \bZ^n}{\Delta t} \,=\, \bF\left(t^n,\bZ^n,0\right),\qquad
0\leq n\leq N_T-1\,,
\\[1.1em]
\ds\bZ^0 = \bZ^0,
\ea\right.
\end{equation} 
which provides a consistent first-order approximation with respect to $\Delta t$ of the
gyro-kinetic system~\eqref{order:1}.
\item Alternatively make the stronger assumption that for some $M>0$, $(\bZ^0(\eps),\buperp^0(\eps))$ is such that for some $\bZ^0$, $\left((\bZ^0(\eps)-\bZ^0)/\eps,\buperp^0(\eps)\right)_{\eps>0}$ is bounded by $M$ uniformly with respect to $\eps>0$.\\
Then, for some $C_M$ depending only on $M$, $\Delta t$ and $T$, for any $1\leq n\leq N_T$,
\[
\|\bZ^{n}(\eps)-\bZ^{n}\|
\,+\,\|\buperp^{n}(\eps)\|
\leq C_M\,\eps\,,
\]
where $(\bZ^n(\eps))_{0\leq
  n\leq N_T}$ is obtained from \eqref{sch:y0}, and, for any $2\leq n\leq N_T$,
\[
\|\bZ^{n}(\eps)-\bY^{n}(\eps) \|
\,+\,\left\|\buperp^{n}(\eps)-\eps\frac{\overline{\bU}_\perp(t^{n-1},\bY^{n-1}(\eps))}{b(\bY^{n-1}(\eps))}\right\| \,\leq\, C_M\, \eps^2
\]
where $(\bY^n(\eps))_{1\leq n\leq N_T}$ is obtained from
\begin{equation}
   \label{sch:y1}
   \left\{\ba{l}
\ds\frac{\bY^{n+1}(\eps)- \bY^n(\eps)}{\Delta t} \,=\, \bF\left(t^n,\bY^n(\eps),\eps\frac{\overline{\bU}_\perp(t^n,\bY^n(\eps))}{b(\bY^n(\eps))} \right),\qquad
1\leq n\leq N_T-1
\\[1.1em]
\bY^1(\eps)= \bZ^1(\eps),
\ea\right.
\end{equation}
which provides a consistent first-order approximation with respect to $\Delta t$ of System~\eqref{order:2}.
\end{enumerate}
\end{proposition}

\begin{proof}
We begin by proving the first point. The key stability observation is that \eqref{scheme:0v} is alternatively written as 
\[
\buperp^{n+1}(\eps) \,=\, \left( \Id \,
 -\, b(\bZ^n(\eps))\,\frac{\Delta t}{\eps}\,\bJ_0\right)^{-1}\,\left( \buperp^n(\eps)
  \, -\, \Delta t\, \bJ_0 \bUperp(t^n,\bZ^n(\eps),\buperp^n(\eps))\right),
\]
where $\left( \Id \, -\, b(\bZ^n(\eps))\,\frac{\Delta t}{\eps}\,\bJ_0\right)^{-1}$ is well-defined and bounded by a multiple of $\eps/\Delta t$. Since $\bUperp$ is at most quadratic in $\buperp$ and $(\sqrt{\eps}\buperp^0(\eps))_{\eps>0}$ converges to $0$ as $\eps\to0$, this implies that $(\buperp^1(\eps))_{\eps>0}$ converges to $0$ as $\eps\to0$ and then arguing recursively that $(\buperp^n(\eps))_{\eps>0}$, $2\leq n\leq N_T$, is bounded uniformly by a multiple of $\eps$ (with a factor depending on $\Delta t$ and $T$). In turn, for some $C'$, for any $0\leq n \leq N_T-1$,
\begin{align*}
\ds \|\bZ^{n+1}(\eps)\,-\, \bZ^{n+1} \|
&\leq \ds (1+C'\,\Delta t)\,\|\bZ^{n}(\eps)\,-\,\bZ^{n} \|
+C'\,\Delta t \,\|\buperp^{n+1}(\eps)\|\,,
\end{align*}
so that the convergence of the $\bZ$-variable stems recursively from the one of the $\buperp$-variable.

We turn to the proof of the second point. We omit to give details on the bound on $(\bZ^{n}(\eps)-\bZ^n,\buperp^n)$ since they are redundant with the ones sketched above. We focus on second-order estimates. With first-order bounds in hands we now use \eqref{scheme:0v} in the form
\[
\buperp^{n+1}(\eps)\,=\,
-\eps\,\bJ_0\,\frac{\buperp^{n+1}(\eps) - \buperp^n(\eps) }{b\left(\bZ^n(\eps)\right)\,\Delta
  t} \,+\,  \eps\,\frac{   \bUperp\left(t^n,\bZ^n(\eps),\buperp^n(\eps)\right)}{b\left(\bZ^n(\eps)\right)}\,,
\]
so as to derive
\be\label{luft}
\left\|\buperp^{n}(\eps)-\eps\frac{\overline{\bU}_\perp(t^{n-1},\bZ^{n-1}(\eps))}{b(\bZ^{n-1}(\eps))}\right\| \,\leq\, C'\, \eps^2\,,
\qquad 2\leq n\leq N_T
\ee
for some $C'$ (depending on $\Delta t$ and $T$). To conclude, it is thus sufficient to prove the bound on $\bZ^n(\eps)-\bY^n(\eps)$. However we already have by using \eqref{luft} that for some $C''$
\[
\ds \|\bZ^{n+1}(\eps)\,-\, \bY^{n+1}(\eps)\|
\leq \ds C''\,\|\bZ^{n}(\eps)\,-\,\bY^{n} \|+C''\,\eps^2\,,\qquad 
1\leq n\leq N_T-1\,.
\]
One may then conclude the proof arguing inductively.
\end{proof}

\begin{remark}
The consistency provided by the latter result is far from being uniform with respect to the time step $\Delta t$. However, though we restrain from doing so here in order to keep technicalities to a bare minimum, we expect that an analysis similar to the one carried out in \cite{FR3} could lead to uniform estimates, proving uniform stability and  consistency with respect to both $\Delta t$ and $\eps$. 
\end{remark}

Of course, a first-order scheme may fail to be accurate enough to describe correctly the long time behavior of the solution, but it has the advantage of simplicity. In the following we show how to generalize our approach to second-order schemes. Though we do not discuss such schemes here, we also recall that our approach is compatible with even higher order schemes and we refer to \cite{FR1,FR2} for third-order examples.

\subsection{Second-order semi-implicit Runge-Kutta schemes}
We now consider a second-order scheme with two stages.  More explicitly, the scheme we introduce is a combination of a Runge-Kutta method for the explicit part and of an $L$-stable second-order SDIRK method for the implicit part.

To describe the scheme, we introduce $\gamma>0$ the smallest root of the polynomial $X^2 - 2X + 1/2$, {\it  i.e.} $\gamma = 1 - 1/\sqrt{2}$. Then the scheme is given by the
following two stages. First,
\begin{subequations}\label{scheme:3-1}
\begin{empheq}[left=\empheqlbrace]{align}
&\ds\frac{\bZ^{(1)} - \bZ^n }{\Delta t} \,=\,  \gamma\,\bF\left(t^n, \bZ^n,\buperp^{(1)}\right),
\label{scheme:3-1Z}
\\[1.1em]
&\ds\bJ_0\frac{\buperp^{(1)} - \buperp^n }{\Delta t} \,=\, \gamma\,\left[\bUperp(t^n,\bZ^n,\buperp^n) \,-\,  \frac{b\left(\bZ^n\right)\,\buperp^{(1)}}{\eps}\right]\,.
\label{scheme:3-1v}
\end{empheq}
\end{subequations}
Before the second stage, we first introduce  $\hat{t}^{(1)} \,=\, t^n + {\Delta t}/{(2\gamma)}$ and explicitly compute $(\hat{\bZ}^{(1)},\hat{\bu}_\perp^{(1)})$ from
\begin{subequations}\label{tard2}
\begin{empheq}[left=\empheqlbrace]{align}
&\ds\hat{\bZ}^{(1)} \,=\, \left( 1
  -\frac{1}{2\gamma^2}\right)\,\bZ^{n}\,+\, \frac{1}{2\gamma^2}\,
  \bZ^{(1)},
  \\[1.1em]
&\ds\hat{\bu}_\perp^{(1)} \,=\, \left( 1 -\frac{1}{2\gamma^2}\right)\,\buperp^{n}\,+\, \frac{1}{2\gamma^2}\, \buperp^{(1)}\,.
\label{tard2v}
\end{empheq}
\end{subequations}
Then  the solution of the second stage $\left(\bZ^{n+1},\buperp^{n+1}\right)$ is given by
\begin{subequations}\label{scheme:3-2}
\begin{empheq}[left=\empheqlbrace]{align}
\ds\frac{\bZ^{n+1} - \bZ^n }{\Delta t} \,=\,&\ds
  (1-\gamma)\,\bF\left(t^n, \bZ^n,\buperp^{(1)}\right) \,+\, \gamma\,\bF\left(t^n, \hat{\bZ}^{(1)},\buperp^{n+1}\right)\,,
  \label{scheme:3-2Z}
\\[1.1em]
\ds\bJ_0\frac{\buperp^{n+1} - \buperp^n }{\Delta t} \,=\,&
\ds  (1-\gamma)\,\left[ \bUperp(t^n,\bZ^n,\buperp^n) \,-\,
  \frac{b\left(\bZ^{n}\right)\,\buperp^{(1)}}{\eps}\right]
  \label{scheme:3-2v}
  \\[1.2em]\nonumber
  \,&\ds\,+\, \gamma\,\left[ \bUperp\left(\hat{t}^{(1)},\hat{\bZ}^{(1)},\hat{\bu}_\perp^{(1)}\right) \,-\,\frac{b\left(\hat{\bZ}^{(1)}\right)\,\buperp^{n+1}}{\eps}\right]\,.
\end{empheq}
\end{subequations}

The following proposition provides consistency results in the limit
$\eps\rightarrow 0$ for the foregoing scheme.

\begin{proposition}[Consistency in the
  limit $\eps\rightarrow 0$ for a fixed $\Delta t$]
  \label{prop:3}
  Let us consider a time step $\Delta t>0$, a final time $T>0$  and
set $N_T:=\lfloor T/\Delta t\rfloor$. 
\begin{enumerate}[(i)]
\item Assume that $(\bZ^0(\eps),\buperp^0(\eps))$ is such that $\left(\bZ^0(\eps),\sqrt{\eps}\buperp^0(\eps)\right)_{\eps>0}$ converges in the limit $\eps\rightarrow 0$ to $(\bZ^0,0)$ for some $\bZ^0$. Consider $(\bZ^n(\eps),\buperp^n(\eps))_{0\leq n\leq N_T}$, the sequence obtained from  $(\bZ^0(\eps),\buperp^0(\eps))$  by \eqref{scheme:3-1}-\eqref{scheme:3-2}.\\
Then, for any $1\leq n\leq N_T$, $(\bZ^n(\eps),\buperp^n(\eps))_{\eps>0}$ converges to $(\bZ^n,0)$ as $\eps\rightarrow 0$ where
  \begin{equation}
  \label{sch:y2}
  \left\{\ba{l}
\ds\frac{\hat{\bZ}^{(1)} - \bZ^n}{\Delta t} \,=\, \frac{1}{2\gamma}\,\bF\left(t^n,\bZ^n,0\right),
\\[1.1em]
\ds\frac{\bZ^{n+1} - \bZ^n}{\Delta t} \,=\, (1-\gamma)\,\bF\left(t^n,\bZ^n,0\right)\,+\, \gamma\,\bF\left(\hat{t}^{(1)},\hat{\bZ}^{(1)},0\right),
\ea\right.
\end{equation}
which provides a consistent second-order approximation with respect to $\Delta t$ of the
gyro-kinetic system~\eqref{order:1}.
\item Alternatively make the stronger assumption that for some $M>0$, $(\bZ^0(\eps),\buperp^0(\eps))$ is such that for some $\bZ^0$, $\left((\bZ^0(\eps)-\bZ^0)/\eps,\buperp^0(\eps)\right)_{\eps>0}$ is bounded by $M$ uniformly with respect to $\eps>0$.\\
Then,  for some $C_M$ depending only on $M$, $\Delta t$ and $T$, for any $1\leq n\leq N_T$,
\[
\|\bZ^{n}(\eps)-\bZ^{n}\|
\,+\,\|\buperp^{n}(\eps)\|
\leq C_M\,\eps\,,
\]
where $(\bZ^n(\eps))_{0\leq
  n\leq N_T}$ is obtained from \eqref{sch:y2}, and, for any $2\leq n\leq N_T$,
\[
\|\bZ^{n}(\eps)-\bY^{n}(\eps) \|
\,+\,\left\|\buperp^{n}(\eps)-\eps\frac{\overline{\bU}_\perp(t^{n-1},\bY^{n-1}(\eps))}{b(\bY^{n-1}(\eps))}\right\| \,\leq\, C_M\, \eps^2
\]
where $(\bY^n(\eps))_{1\leq n\leq N_T}$ is obtained from
\begin{equation}
   \label{sch:y3}
   \left\{\ba{ll}
\ds\frac{\hat{\bY}^{(1)} - \bY^n}{\Delta t} \,= & \ds \frac{1}{2\gamma}\bF\left(t^n,\bY^n,\eps  \frac{\overline{\bU}_\perp(t^n,\bY^n)}{b(\bY^n)} \right),
\\[1.1em]
\ds\frac{\hat{\bY}^{n+1} - \bY^n}{\Delta t} \,= & \ds
(1-\gamma)\,\bF\left(t^n,\bY^n,\eps
  \frac{\overline{\bU}_\perp(t^n,\bY^n)}{b(\bY^n)} \right) \\
\\[1.1em]
& \ds +\, \gamma\, \bF\left(\hat{t}^{(1)},\hat{\bY}^{(1)},\eps
  \frac{\overline{\bU}_\perp(\hat{t}^{(1)},\hat{\bY}^{(1)})}{b\left(\hat{\bY}^{(1)}\right)} \right)\,,
\ea\right.
\end{equation}
which provides a consistent second-order approximation with respect to $\Delta t$ of System~\eqref{order:2}.
\end{enumerate}
\end{proposition}

\begin{proof}
The proof mainly follows the lines of the proof of Proposition~\ref{prop:1} and thus we only sketch a few details about the proof of the first point. To make arguments more precise we mark with a suffix ${}_{n}$ intermediate quantities involved in the step from $t^n$ to $t^{n+1}$.

From \eqref{scheme:3-1v} we deduce that  $(\bu^{(1)}_{\perp,0}(\eps))_\eps$ converges to zero, then from \eqref{tard2v} we derive that
$(\sqrt{\eps}\hat{\bu}^{(1)}_{\perp,0}(\eps))_\eps$ converges to zero and finally from \eqref{scheme:3-2v} that $(\bu^{1}_{\perp}(\eps))_\eps$ converges to zero. Likewise we infer that $(\bu^{(1)}_{\perp,1}(\eps)/\eps)_\eps$ is bounded, $(\hat{\bu}^{(2)}_{\perp,0}(\eps))_\eps$ converges to zero and $(\bu^{2}_{\perp}(\eps)/\eps)_\eps$ is bounded. At last, we deduce that for $2\leq n\leq N_T-1$, $(\bu^{(1)}_{\perp,n}(\eps)/\eps)_\eps$, $(\hat{\bu}^{(2)}_{\perp,0}(\eps)/\eps)_\eps$ and $(\bu^{n+1}_{\perp}(\eps)/\eps)_\eps$ are bounded. Inserting these bounds in \eqref{scheme:3-1Z} and \eqref{scheme:3-2Z} is sufficient to achieve the proof of the first point.
\end{proof}

\section{Numerical simulations}
\label{sec:4}
\setcounter{equation}{0}

In this section, we provide examples of numerical computations to validate and compare
the different time discretization schemes introduced in the previous section. We only consider the motion of individual particles under the effect of given
electromagnetic fields and investigate on it the accuracy and stability properties with respect to $\eps>0$ of the semi-implicit algorithms presented in Section~\ref{sec:3}. It allows us to  illustrate the ability of the semi-implicit schemes to capture in the limit $\varepsilon\rightarrow 0$ drift velocities due to variations of magnetic and electric fields, even with large time steps $\Delta t $ .

\subsection{Particle motion without electric field}
In this first subsection, numerical experiments are run with a zero electric field $\bE\equiv0$, and a time-independent external magnetic field $\bB$ corresponding to 
\[
\bB(r,\theta,\varphi)= B_\varphi(r,\theta)\,\eDp(\varphi)\,+\,
B_\theta(r,\theta)\,\eDt(\theta,\varphi)\,,
\]
where 
\[
B_\theta(r,\theta):=\frac{B_1\,r}{R}= \frac{B_1\,r}{R_0+ r\cos(\theta)}   \qquad
B_\varphi(r,\theta):=\frac{B_0}{R}=\frac{B_0}{R_0+ r\cos(\theta)},
\]
with $B_0=50$ and $B_1=10$. Observe that this choice guarantees that the magnetic field $\bB$ is divergence-free. We set torus radius to be $R_0=7/4$ and choose for all simulations the initial data so that at $t=0$ we have $r(0)=3/2$, $\theta(0)=\pi/6$ and $\varphi(0)=\pi/8$ whereas for the velocity we choose $\bV(0)=(10,10,5)$.

First, for comparison, we compute a reference solution $(\bX(\eps),\bV(\eps))_{\eps>0}$
to the initial problem \eqref{eq:xv2} up to a final time $T=0.5$ thanks to an explicit
fourth-order Runge-Kutta scheme used with a very small time step
$\Delta t=10^{-8}$ and a reference solution $\bY(\eps)$ to the (non
stiff) asymptotic model \eqref{order:2} obtained when $\eps
\ll 1$. Then we compute an approximate solution
$(\bX_{\Delta t}(\eps), \bV_{\Delta t}(\eps)):=(\bX_{\Delta t}^n(\eps), \bV^n_{\Delta t}(\eps))_{0\leq n\leq N_T}$  using \eqref{scheme:3-1}--\eqref{scheme:3-2} and a classical Boris scheme \cite{boris}, and we compare them to the reference solutions. Our goal is to evaluate the accuracy of the numerical solution $(\bX_{\Delta t}(\eps), \bV_{\Delta t}(\eps))$ for various regimes when both $\eps$ and $\Delta t$ vary, errors being measured on spatial positions in discrete $L^\infty$ norms
\[
\| \bX_{\Delta t}(\eps) - \bX(\eps) \| \,:=\,  \max_{n\in\{0,..,N_T\}}  \|\bX^{n}_{\Delta t}(\eps) - \bX(\eps)(t^n) \|\,.
\]

In Figure~\ref{fig:1}, we present the numerical error expressed with respect to the time step $\Delta t$ obtained using \eqref{scheme:3-1}--\eqref{scheme:3-2} and the Boris
scheme \cite{boris}. On the one hand, as expected, for a fixed  $\eps$ taken here between $10^{-5}$ and $10^{-1}$) and $\Delta t$ is small,  both schemes are quite accurate, but the amplitude of the error obtained using \eqref{scheme:3-1}--\eqref{scheme:3-2}
is much smaller than the one obtained using the Boris scheme when $\Delta t$ is significantly smaller than $\eps$. Accuracies are comparable only for some intermediate regimes. Indeed, it is also visible in Figure~\ref{fig:1}, that in the opposite regime, when $\Delta t \simeq 10^{-2}$ is fixed and $\eps$ is sent to zero, the error decreases with respect to $\eps$ when the IMEX scheme \eqref{scheme:3-1}--\eqref{scheme:3-2} is applied whereas it saturates with the Boris scheme. 
 
\begin{figure}[H]
\begin{center}
 \begin{tabular}{cc}
\includegraphics[width=7.75cm]{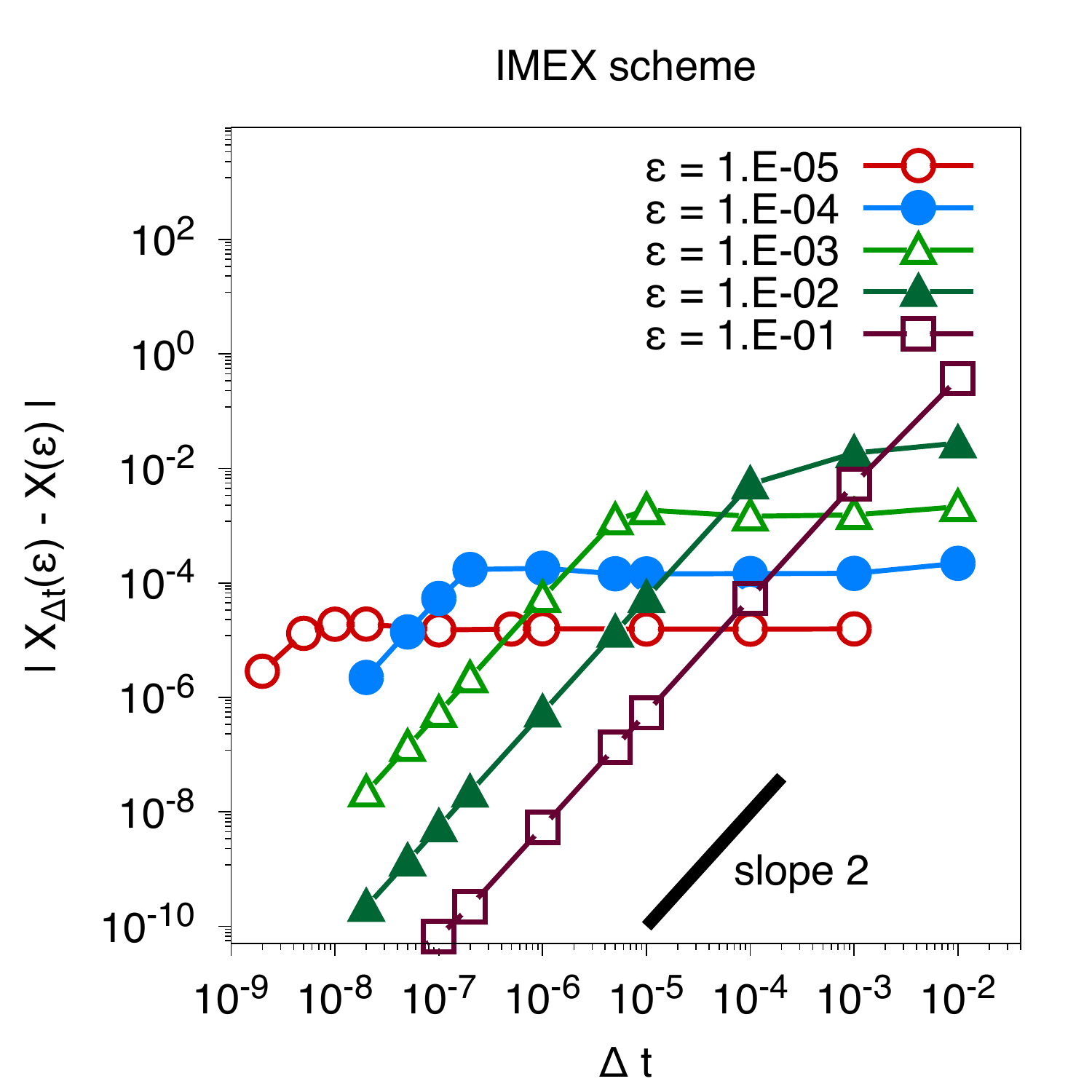} &    
\includegraphics[width=7.75cm]{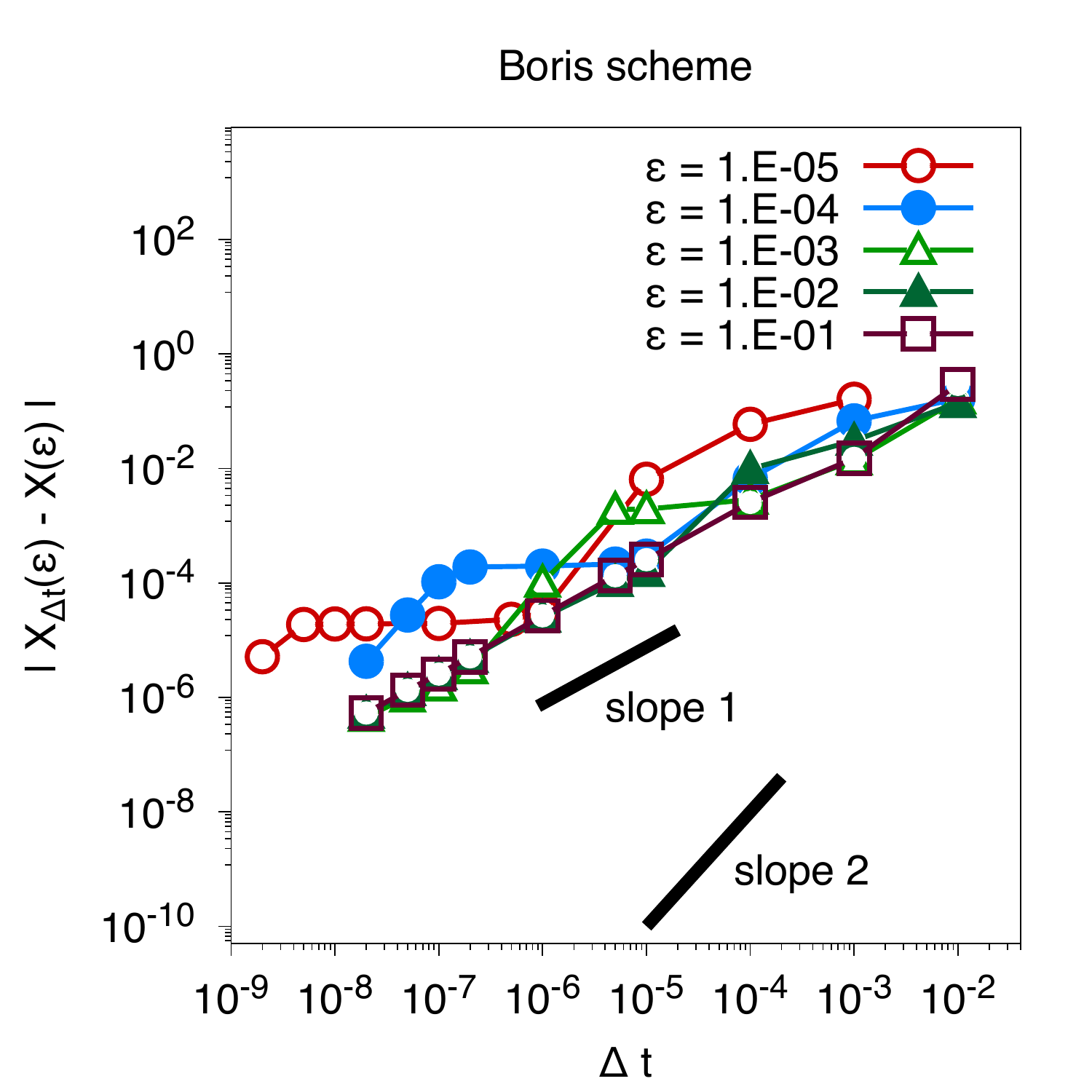} 
\\
(a)  & (b)  
\end{tabular}
\caption{\label{fig:1}
{\bf Particle motion without electric field.} Numerical errors 
$\|\bX^\eps_{\Delta t}-\bX^\eps\|$ obtained for different $\eps$ with (a)  the second order scheme
\eqref{scheme:3-1}-\eqref{scheme:3-2} and (b) the Boris scheme
\cite{boris}, plotted as functions of $\Delta t$.}
 \end{center}
\end{figure}

The latter observation is expected by design and consistent with Figure \ref{fig:2} where for the same set of simulations we plot errors with respect to the asymptotic
system~\eqref{order:2}. This illustrates that, for a fixed $\Delta t$,  the scheme
\eqref{scheme:3-1}--\eqref{scheme:3-2} captures correct first-order dynamics and correct second-order perpendicular drift velocities in the asymptotic $\eps\ll  1$. 

\begin{figure}[H]
\begin{center}
 \begin{tabular}{cc}
\includegraphics[width=7.75cm]{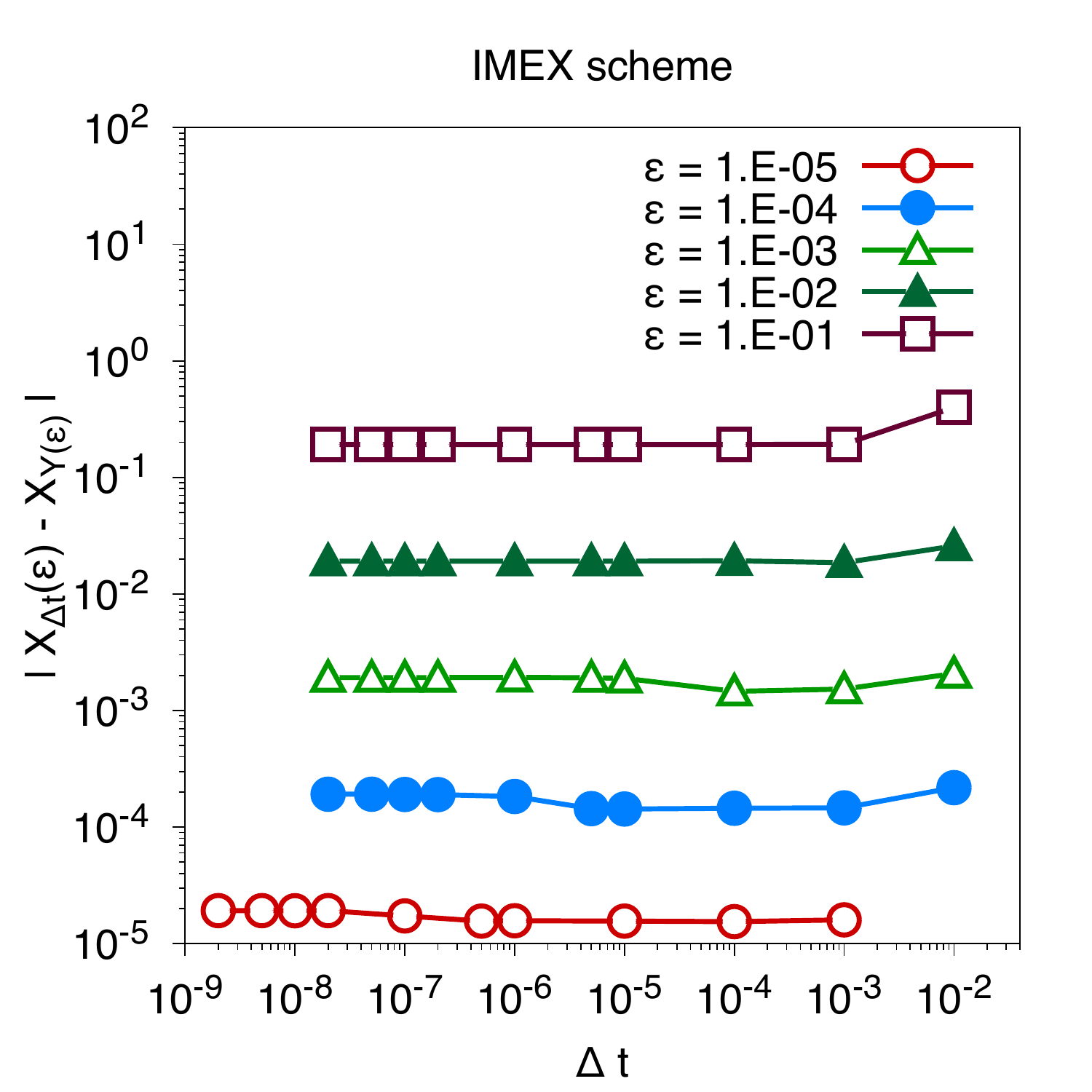} &    
\includegraphics[width=7.75cm]{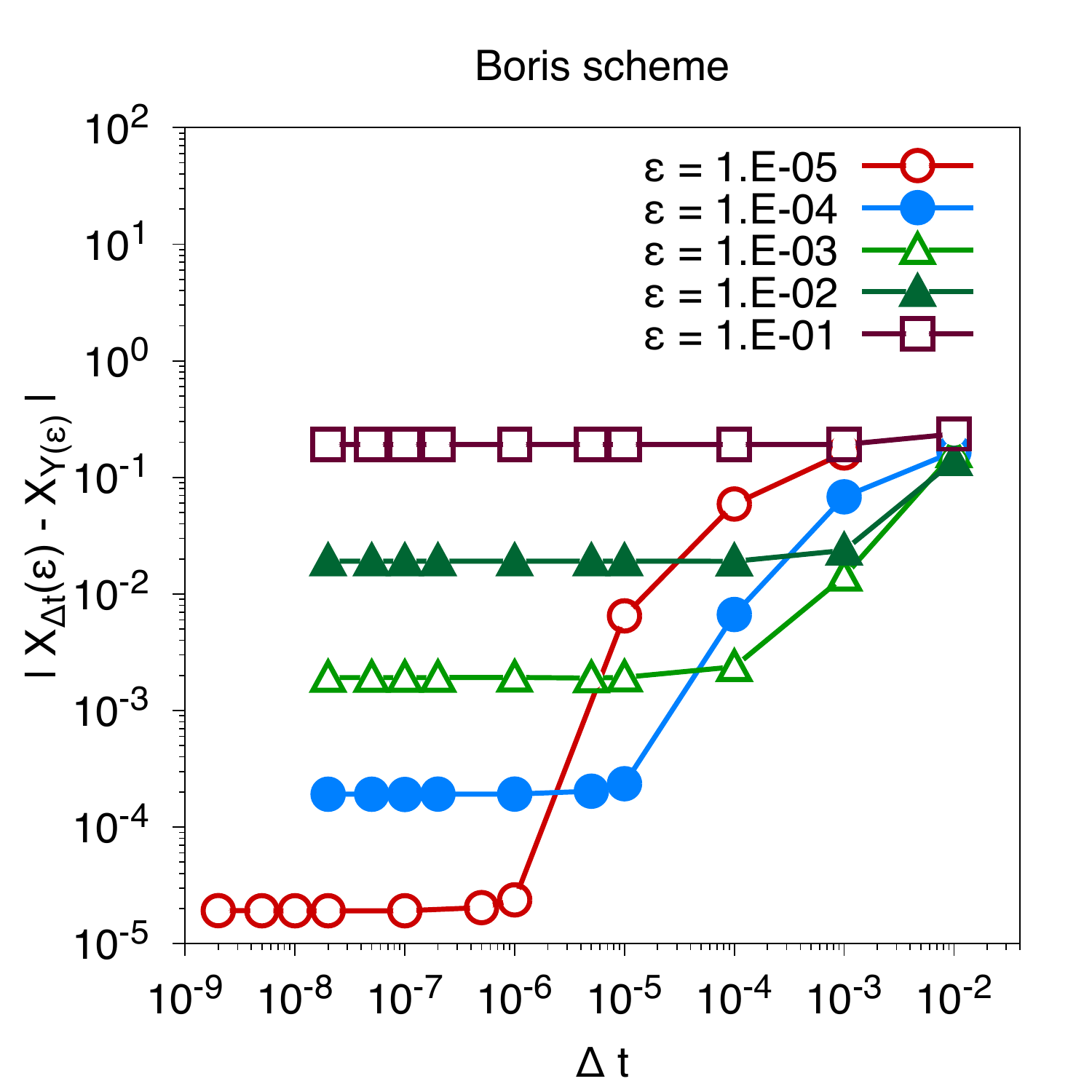} 
\\
(a)  & (b)  
\end{tabular}
\caption{\label{fig:2}
{\bf Particle motion without electric field.} Numerical errors 
$\|\bX_{\Delta t}(\eps)-\bX_{\bY(\eps)}\|$, $\bX_{\bY(\eps)}$ being the spatial position built from $\bY$,  obtained for different $\eps$ with (a) second order scheme
\eqref{scheme:3-1}-\eqref{scheme:3-2} and (b) the Boris scheme
\cite{boris}, plotted as functions of $\Delta t$.}
 \end{center}
\end{figure}

Finally, in Figures~\ref{fig:3} and \ref{fig:4} we show the time evolution of quantities related to slow\footnote{At least at first order.} components $(\mu,\vpar,r,z,\bx)$, 
obtained with the second-order scheme \eqref{scheme:3-1}-\eqref{scheme:3-2} and the Boris scheme, holding fixed both the time step and the stiffness parameter to $\Delta t = 0.02$ and  $\eps=10^{-2}$, which corresponds to an intermediate regime. Both results are compared with the reference solution computed with a small time step. On the one hand, we observe that the variations of the adiabatic invariant $\mu$ (shown in Figure~\ref{fig:3}$(a)$ and \ref{fig:4}$(a)$) are of order $\eps$ as it is expected but the Boris scheme slightly amplifies these variations whereas our IMEX scheme \eqref{scheme:3-1}-\eqref{scheme:3-2} overdamp them. The time evolution of $\vpar$ given by the IMEX scheme coincides with the one given by the reference solution whereas for large time the Boris scheme exhibits a small phase shift. On the other hand, we also present the projection in the $r$-$z$ plane and $3D$ trajectory in Figures~\ref{fig:3} and \ref{fig:4}, where $\bx=(x,y,z)$ denote there Cartesian coordinates. Both results are in good agreement with the ones corresponding to the reference solution. However, again the Boris scheme has some spuriously large oscillations whereas the trajectory obtained with the IMEX scheme is smooth and shows better agreement with spatial positions. 

Though we do not show numerical plots, it is worth mentioning that for
this test without electric field the kinetic energy is theoretically
preserved on the continuous system \eqref{eq:1} over time. However,
the numerical approximation provided by the IMEX scheme \eqref{scheme:3-1}-\eqref{scheme:3-2v} does not preserve
exactly the discrete kinetic energy  and its the fluctuations are of
order $10^{-6}$ (and do not increase over time)  whereas the  fluctuations to the Boris scheme are of order of the round-off error. 

\begin{figure}[H]
\begin{center}
 \begin{tabular}{cc}
\includegraphics[width=7.75cm]{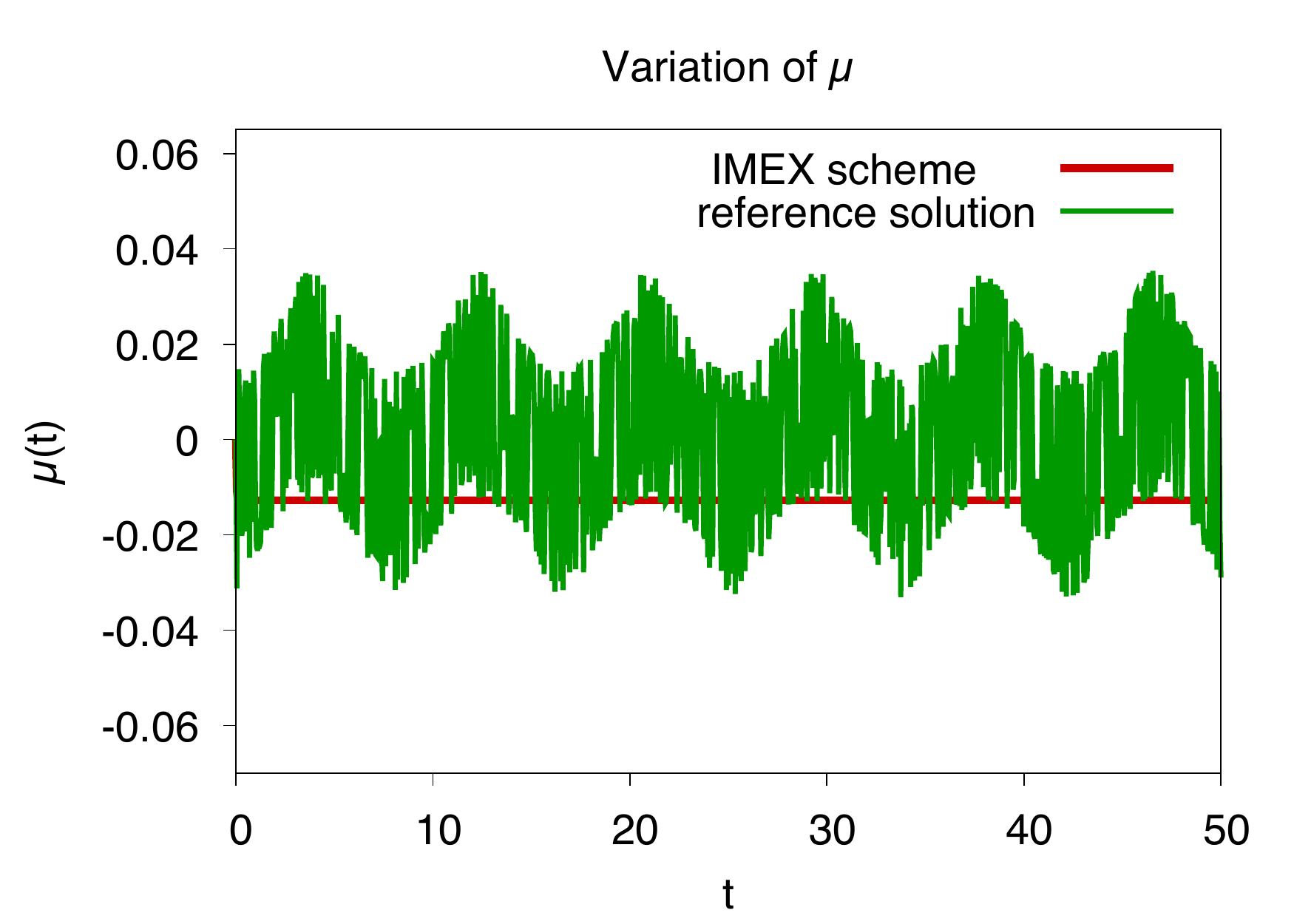} &    
\includegraphics[width=7.75cm]{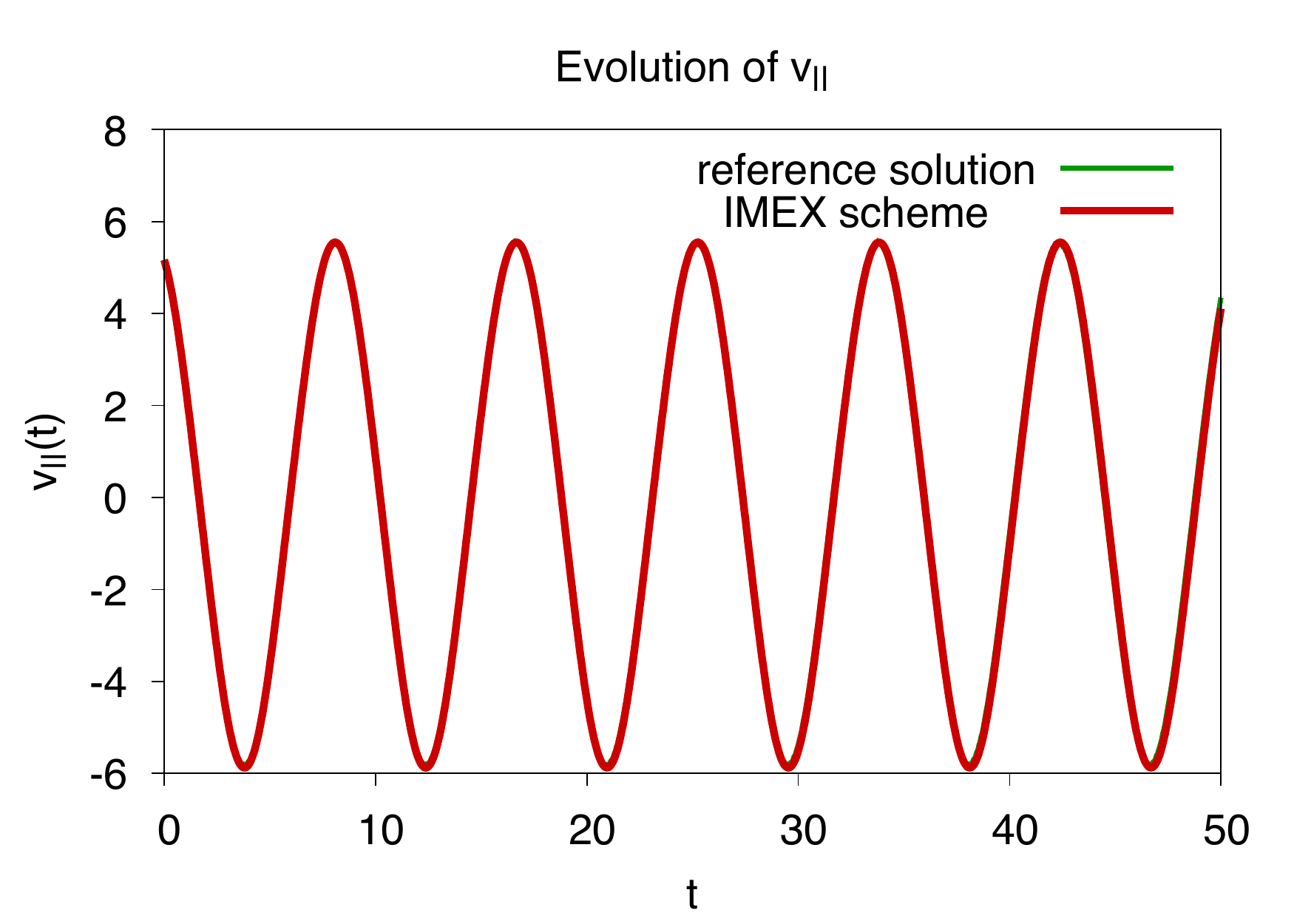}
  \\
   (a) & (b)
   \\
 \includegraphics[width=7.75cm]{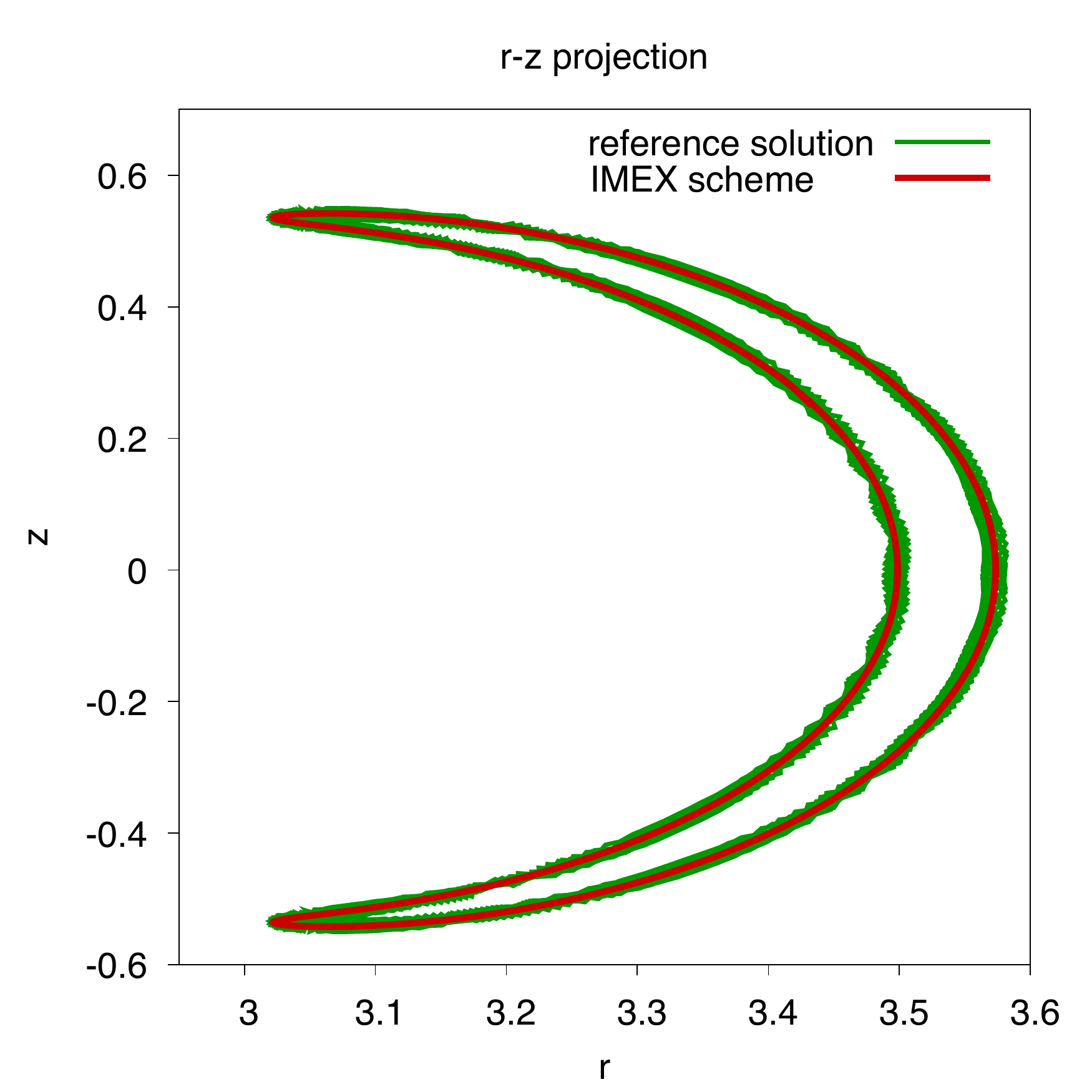}&
 \includegraphics[width=7.75cm]{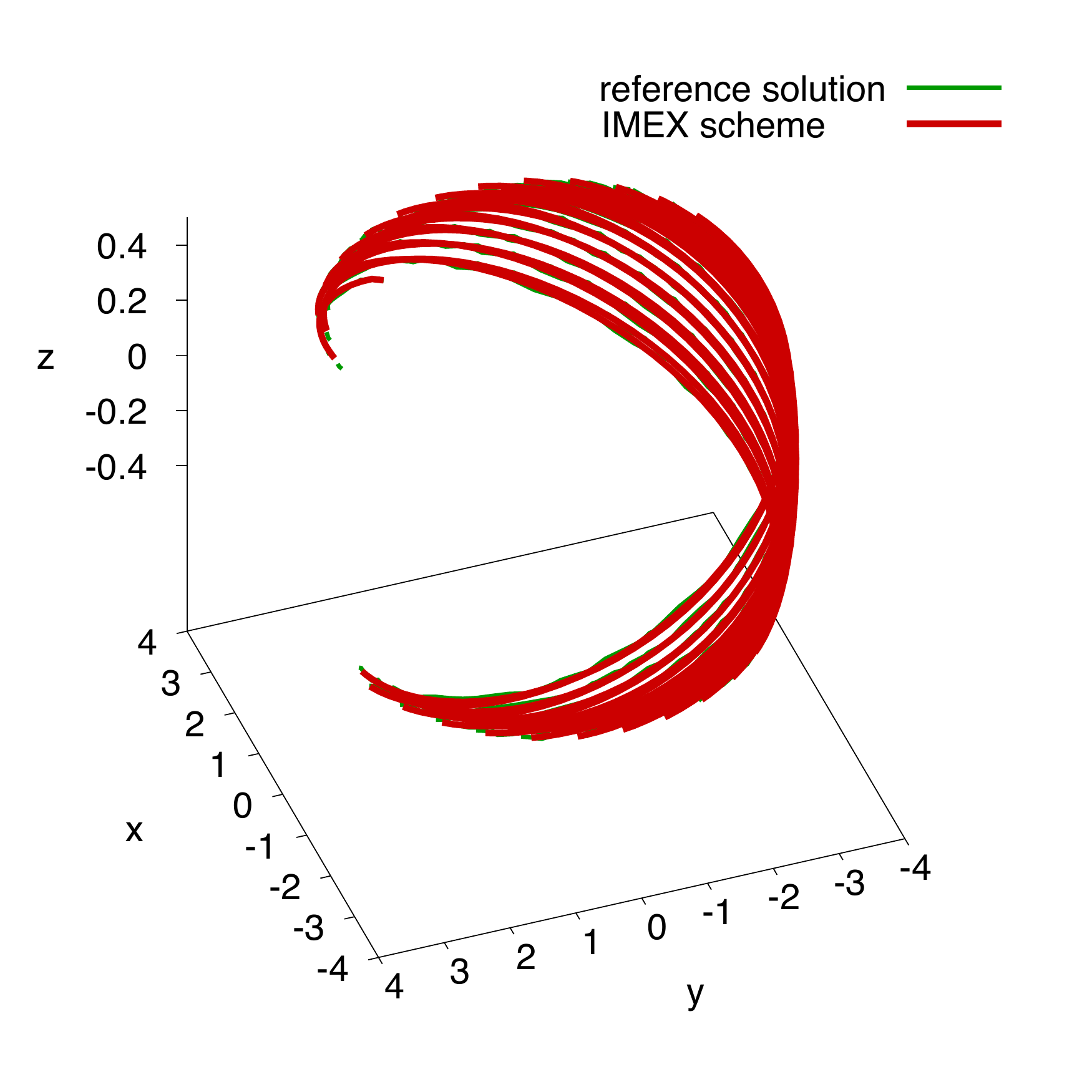}  \\
(c)  & (d)  
\end{tabular}
\caption{\label{fig:3} {\bf Particle motion without electric field.}
  Time evolution of  (a) $\mu$,  (b) $v_\parallel$,  (c) $r$-$z$ projection and (d) 3D trajectory obtained with the  second order scheme \eqref{scheme:3-1}-\eqref{scheme:3-2} with $\Delta t=2 \,10^{-3}$ and $\eps=10^{-2}$.}
 \end{center}
\end{figure}

\begin{figure}[H]
\begin{center}
 \begin{tabular}{cc}
\includegraphics[width=7.75cm]{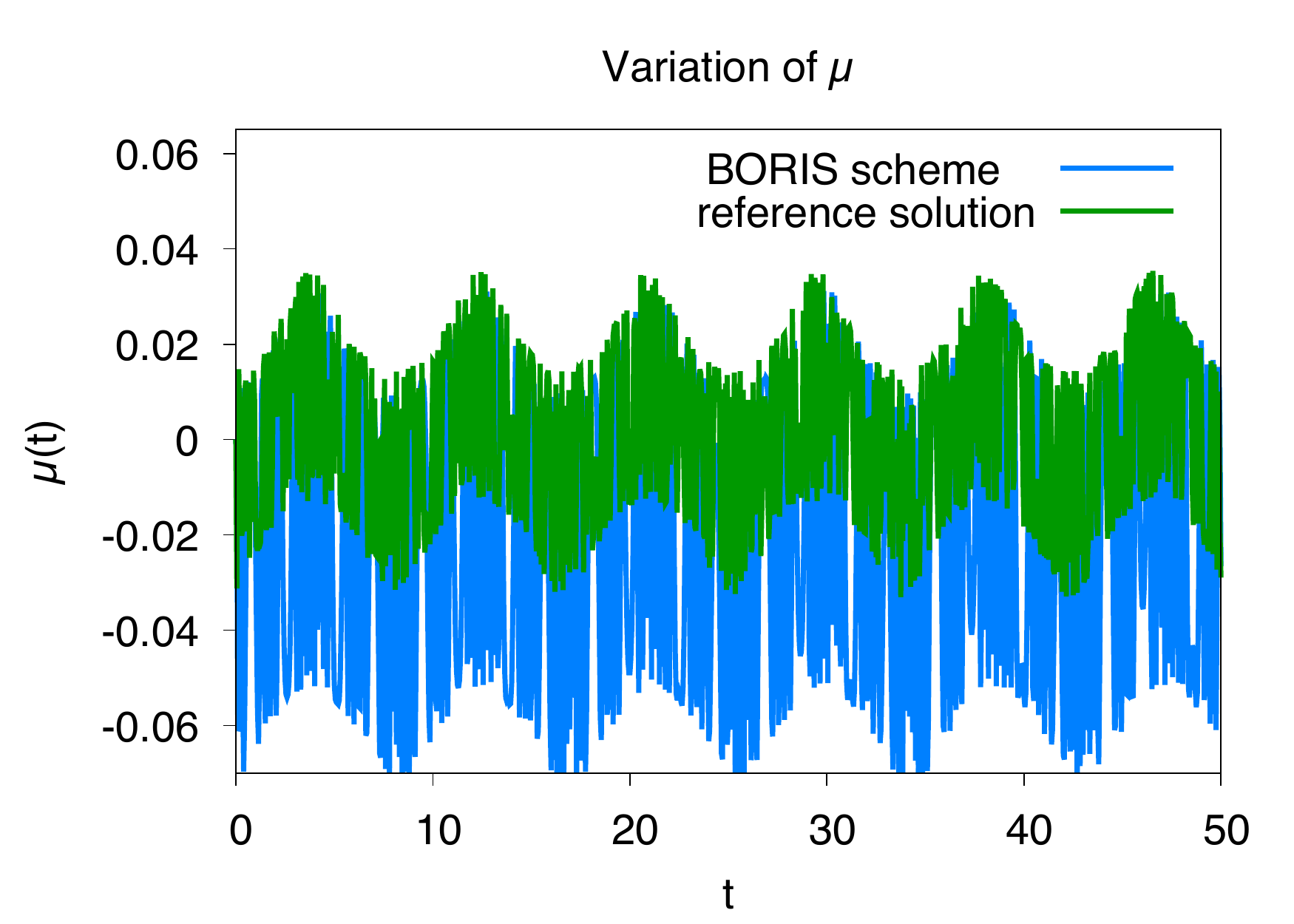} &    
\includegraphics[width=7.75cm]{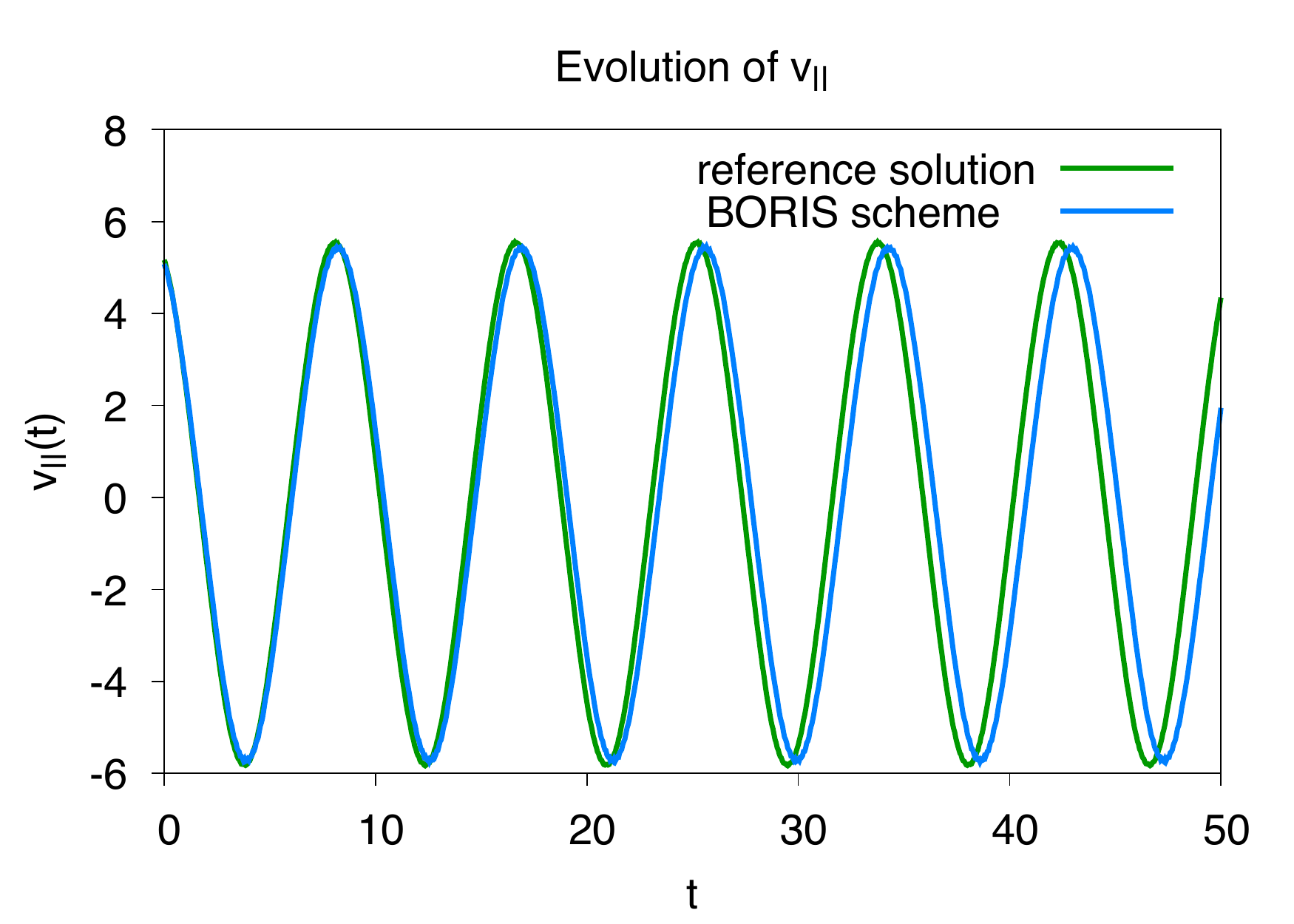} \\
  (a) & (b) \\
 \includegraphics[width=7.75cm]{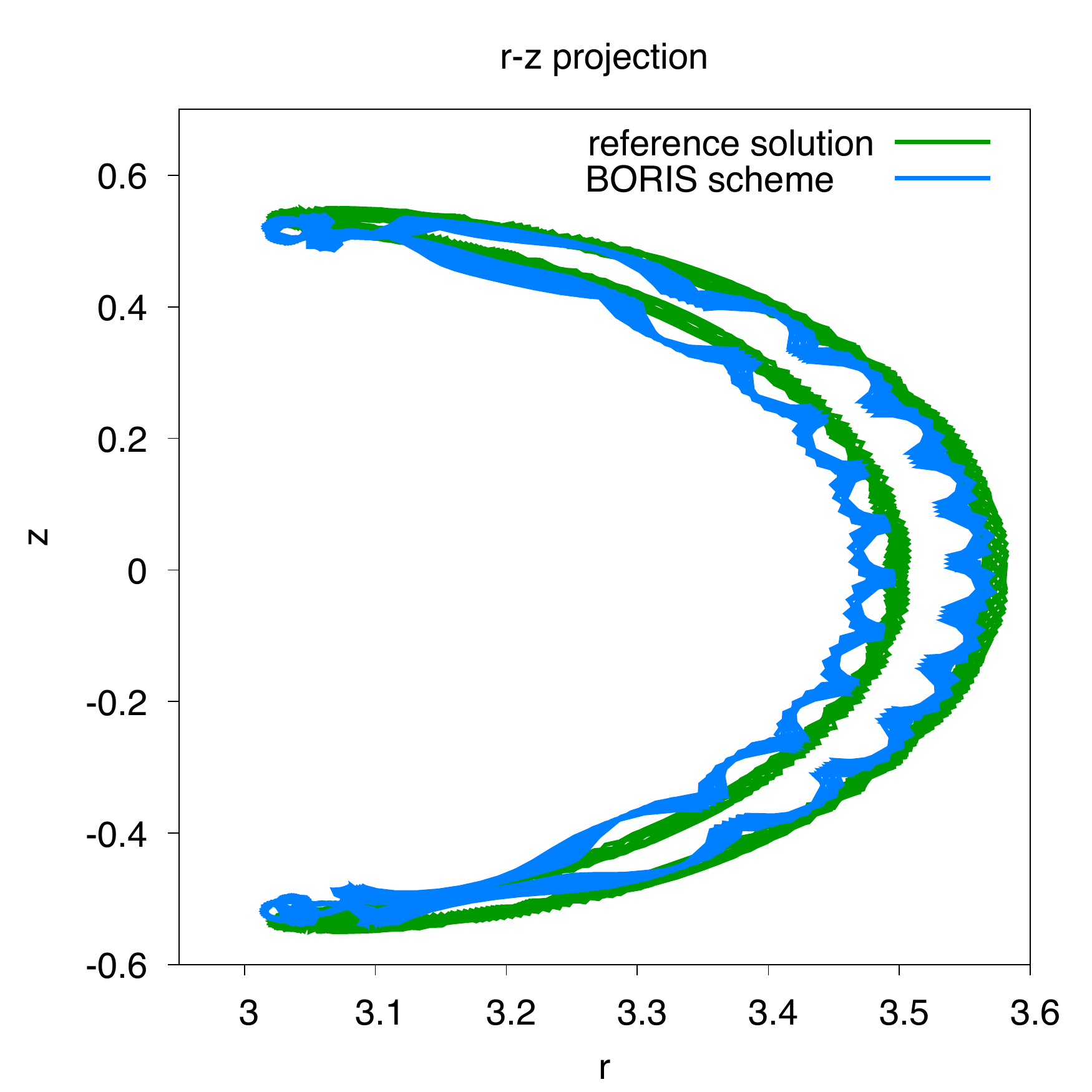}&
 \includegraphics[width=7.75cm]{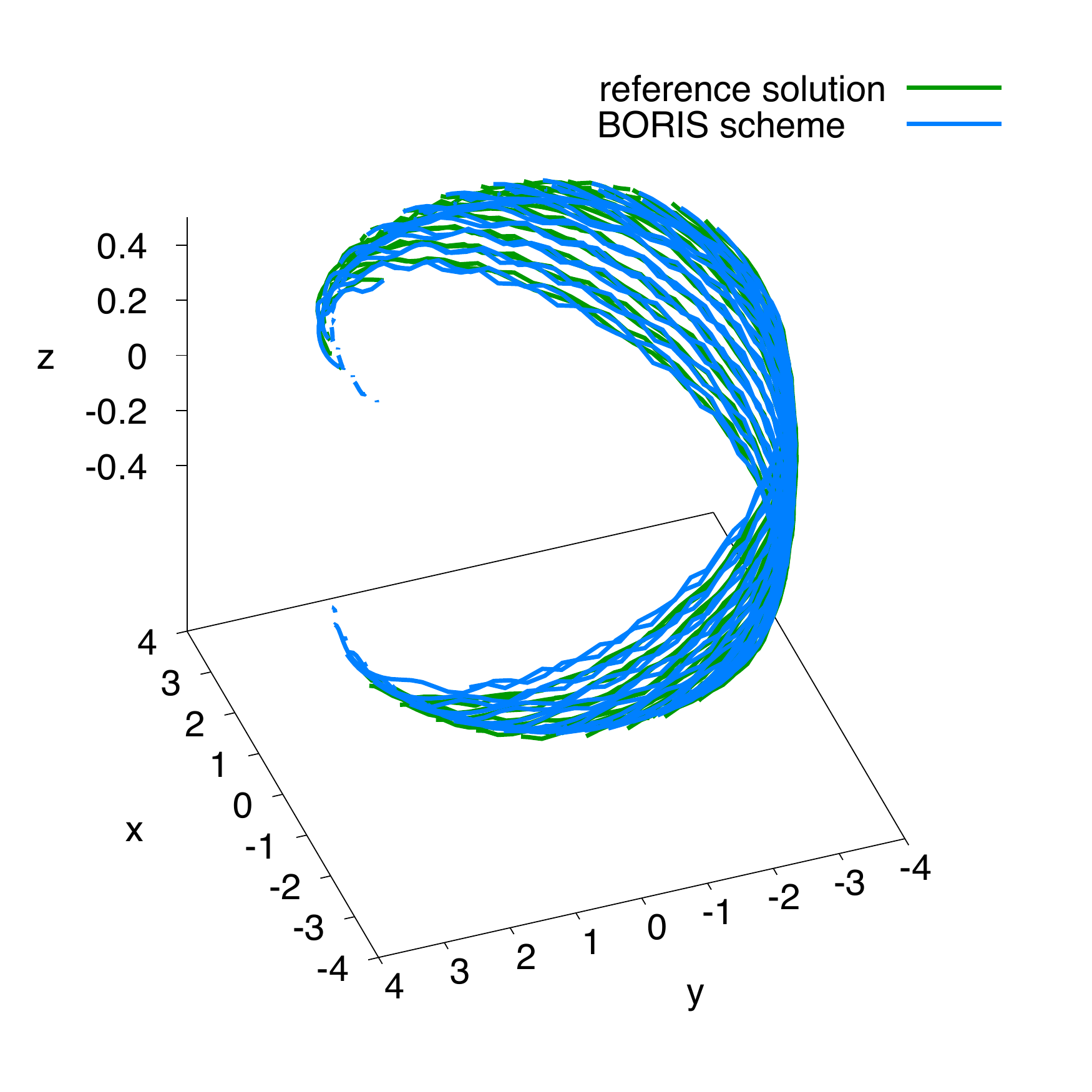}  \\
 (c)  & (d)  
\end{tabular}
\caption{\label{fig:4} {\bf Particle motion without electric field.}
  Time evolution of  (a) $\mu$, (b) $v_\parallel$,  (c) $r$-$z$ projection and (d) 3D trajectory obtained with the  second order Boris scheme \cite{boris} with $\Delta t=2 \,10^{-3}$ and $\eps=10^{-2}$.}
\end{center}
\end{figure}

\subsection{Tokamak-like Equilibrium}
Now we consider an equilibrium magnetic field in a tokamak-like
geometry \cite{cerfon2010}. Explicitly, a Solov'ev equilibrium solution of the Grad- Shafranov equation for the flux function $\psi$ can be written as
\[
\psi(r)  = 5\,\left( \frac{r^2}{2} \,-\, \frac{r^3}{3}\right)\,,
\]
and, correspondingly, we set
\[
\bB:=B_0\,\nabla\varphi+\nabla\psi\wedge\nabla\varphi
\]
with $B_0=50$. Therefore $\bB=B_\theta\,\eDt+B_\varphi\,\eDp$ with
\begin{align*}
B_\theta&:=\frac{5\, r\,(1-r)}{R} \,=\,\frac{5\, r\,(1-r)}{R_0 + r\cos(\theta)}\,,&
B_\varphi&:=\frac{B_0}{R} \,=\,\frac{B_0}{R_0 + r\cos(\theta)}\,.
\end{align*}
Additionally, we introduce a compatible electrostatic potential. 
Our choice is motivated by the fact that, in steady form, the ideal MHD Ohm's law $\bE\,+\,\bv \,\wedge \,\bB  \,=\, 0$ requires that $\bB\cdot \bE = 0$. We set $\bE=-\nabla\phi$, with electrostatic potential $\phi = -2\,\psi$.

We perform several numerical simulations for various parameters $\eps$
and time steps $\Delta t$. Again, as in Figures~\ref{fig:1} and~\ref{fig:2}, numerical errors show that a better accuracy of our approach based on the combination of augmented formulation \eqref{eq:1} with IMEX scheme \eqref{scheme:3-1}-\eqref{scheme:3-2}, allowing in particular the use of a larger time step $\Delta t$ than the ones compatible with the Boris scheme.  Since the results are globally the same as the ones shown
in the previous section, we do not report them here.

However we complete time evolutions for variables as in focus on the plots of the time evolution of physical quantities as in Figures~\ref{fig:3} and~\ref{fig:4}, with those of 
the potential and kinetic energy. Explicitly, in Figures~\ref{fig:5-1},  \ref{fig:6-1}  and~\ref{fig:7-1}, we show plots for energies and $\mu$, whereas in \ref{fig:5-2},  \ref{fig:6-2}  and \ref{fig:7-2} we show time evolutions of $\vpar$ and the $r$-$z$ projection of the spatial trajectory. In these experiments we fix $\Delta t=10^{-2}$ whereas the inverse of the amplitude of the magnetic field $\eps$ varies through $\eps=5\,10^{-2}$, $10^{-2}$, $10^{-3}$. 

First for $\eps=5\times 10^{-2}$ (and $\Delta t=10^{-2}$), we observe that both numerical outcome (obtained from the IMEX \eqref{scheme:3-1}-\eqref{scheme:3-2} and Boris schemes) are in good agreement with the reference solution, even if  the IMEX scheme \eqref{scheme:3-1}-\eqref{scheme:3-2} has a tendency to overdamp oscillations whereas the Boris scheme seems to slightly overamplify them. This point may be observed for instance on the time evolution of $\mu$ in Figure \ref{fig:5-1} or in the $r$-$z$ projection of the trajectory in Figure \ref{fig:5-2}, where particles oscillate following a banana trajectory.

\begin{figure}[H]
\begin{center}
 \begin{tabular}{cc}
\includegraphics[width=7.75cm]{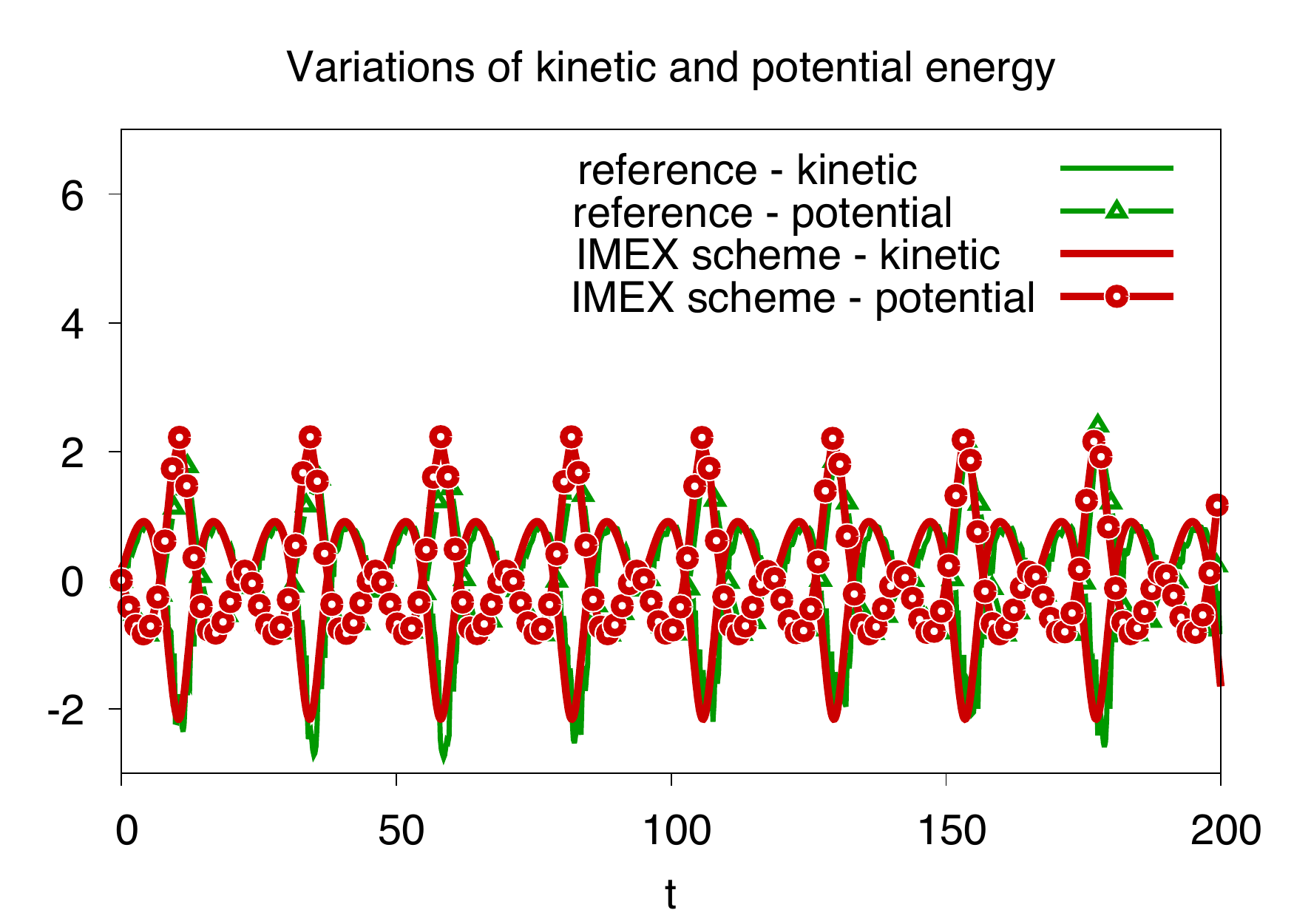} &    
\includegraphics[width=7.75cm]{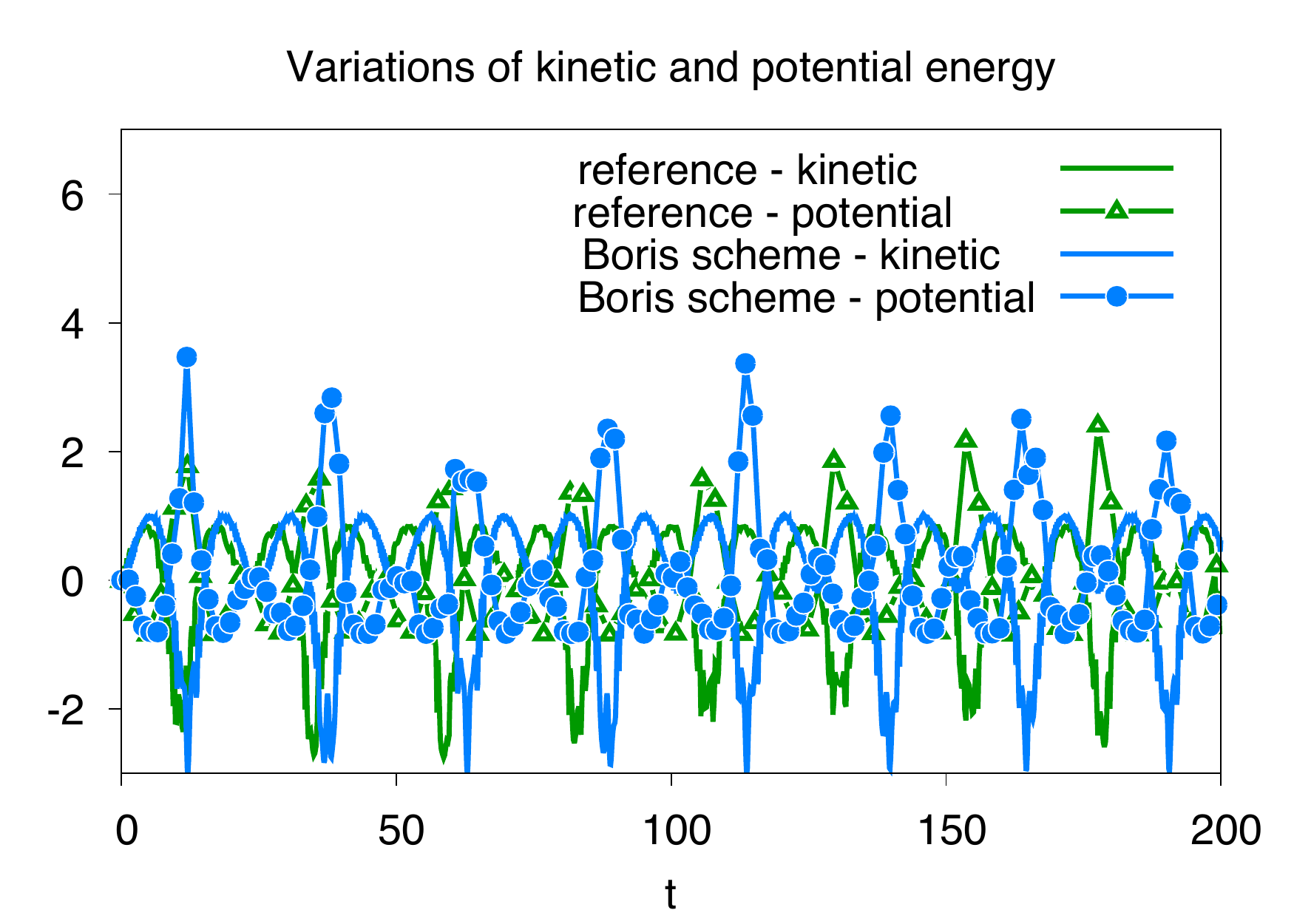}
   \\
\includegraphics[width=7.75cm]{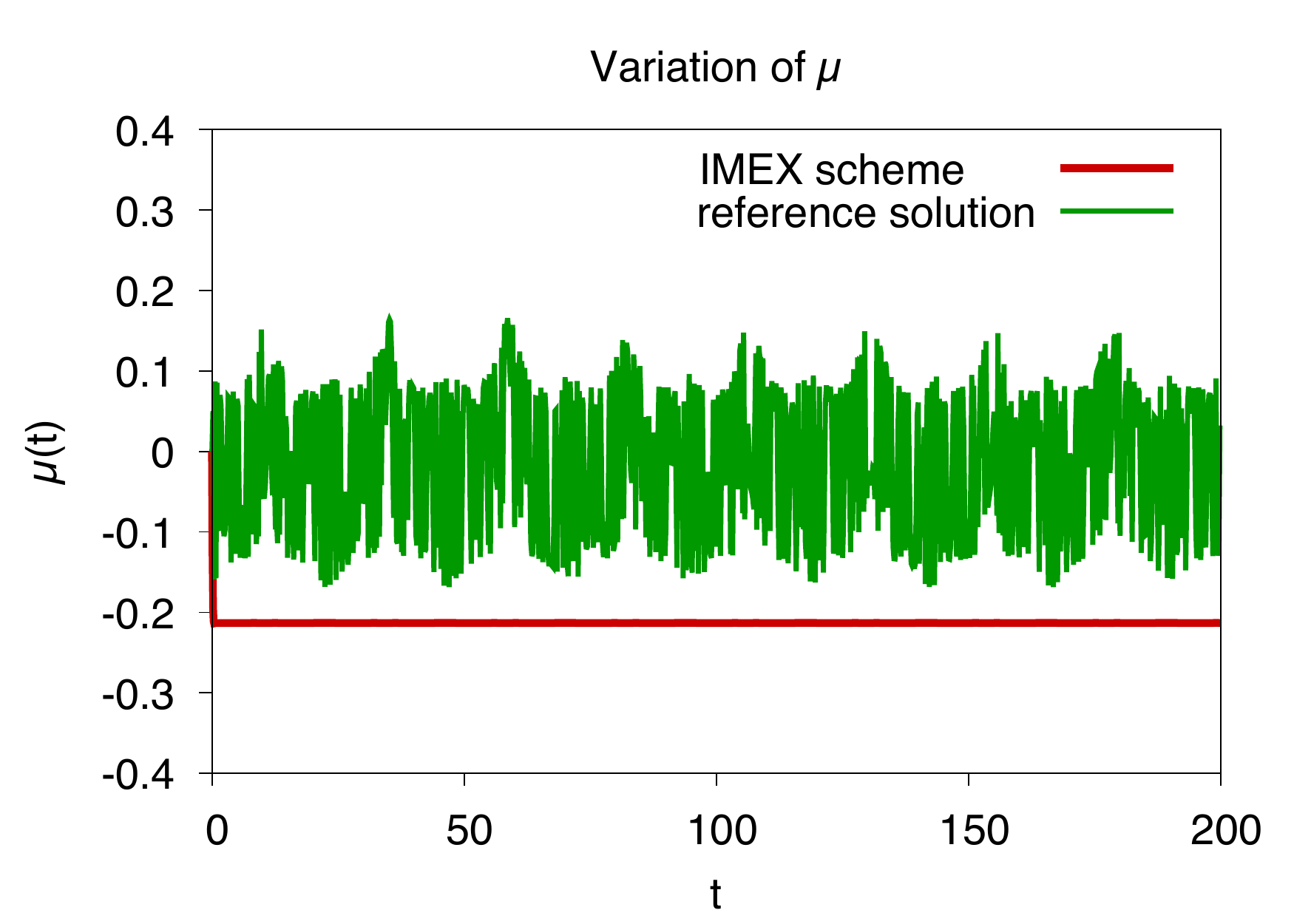}&
\includegraphics[width=7.75cm]{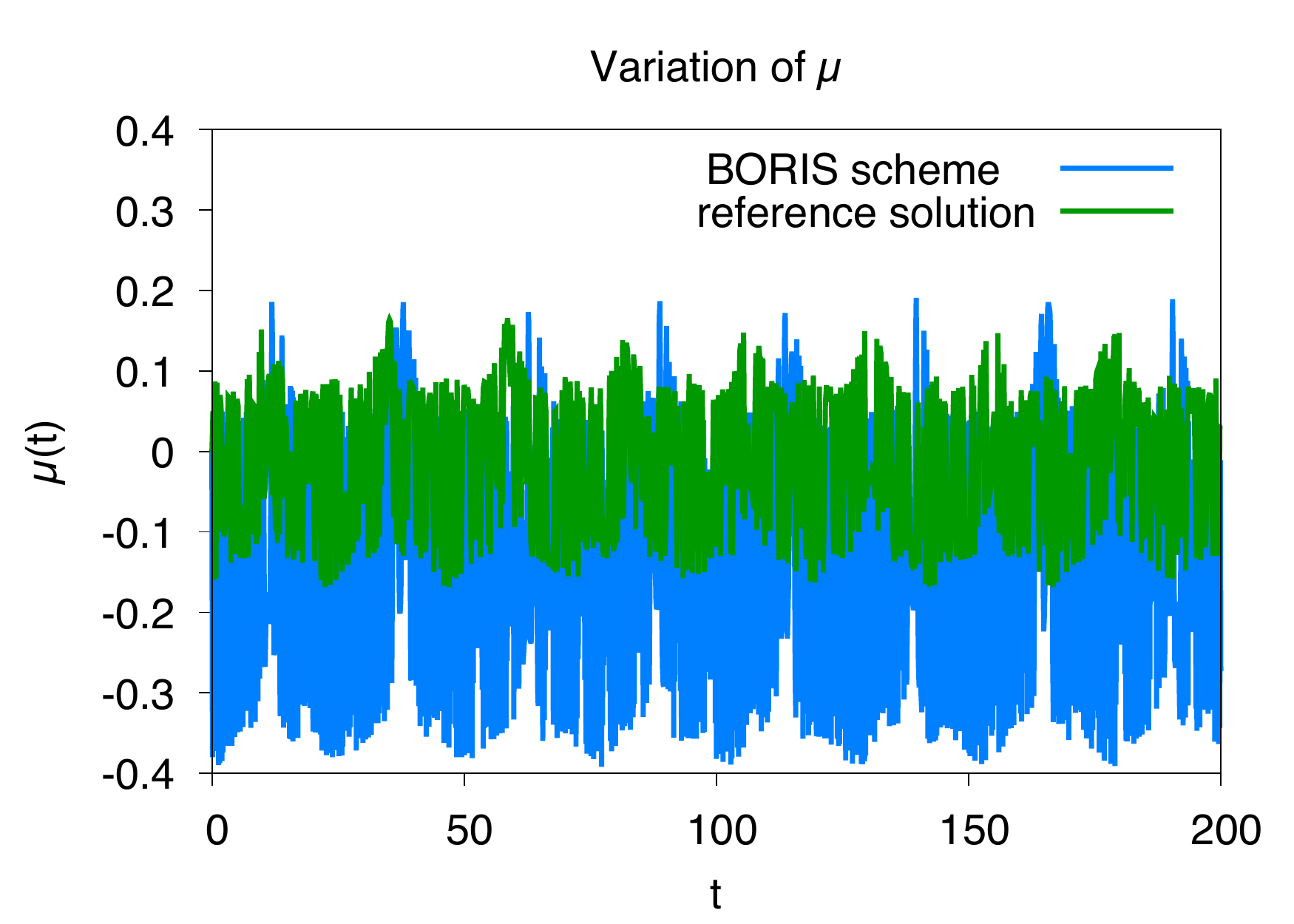}
 \\
(a)  & (b)  
\end{tabular}
\caption{\label{fig:5-1} {\bf Tokamak Equilibrium  $\eps=5\,\times\,10^{-2}$.}
  Time variation of kinetic \& potential energy and adiabatic invariant  $\mu$ obtained with (a) the second order scheme \eqref{scheme:3-1}-\eqref{scheme:3-2} and (b) the second order Boris scheme \cite{boris}, with $\Delta t=10^{-2}$.}
 \end{center}
\end{figure}

\begin{figure}[H]
\begin{center}
 \begin{tabular}{cc}
\includegraphics[width=7.75cm]{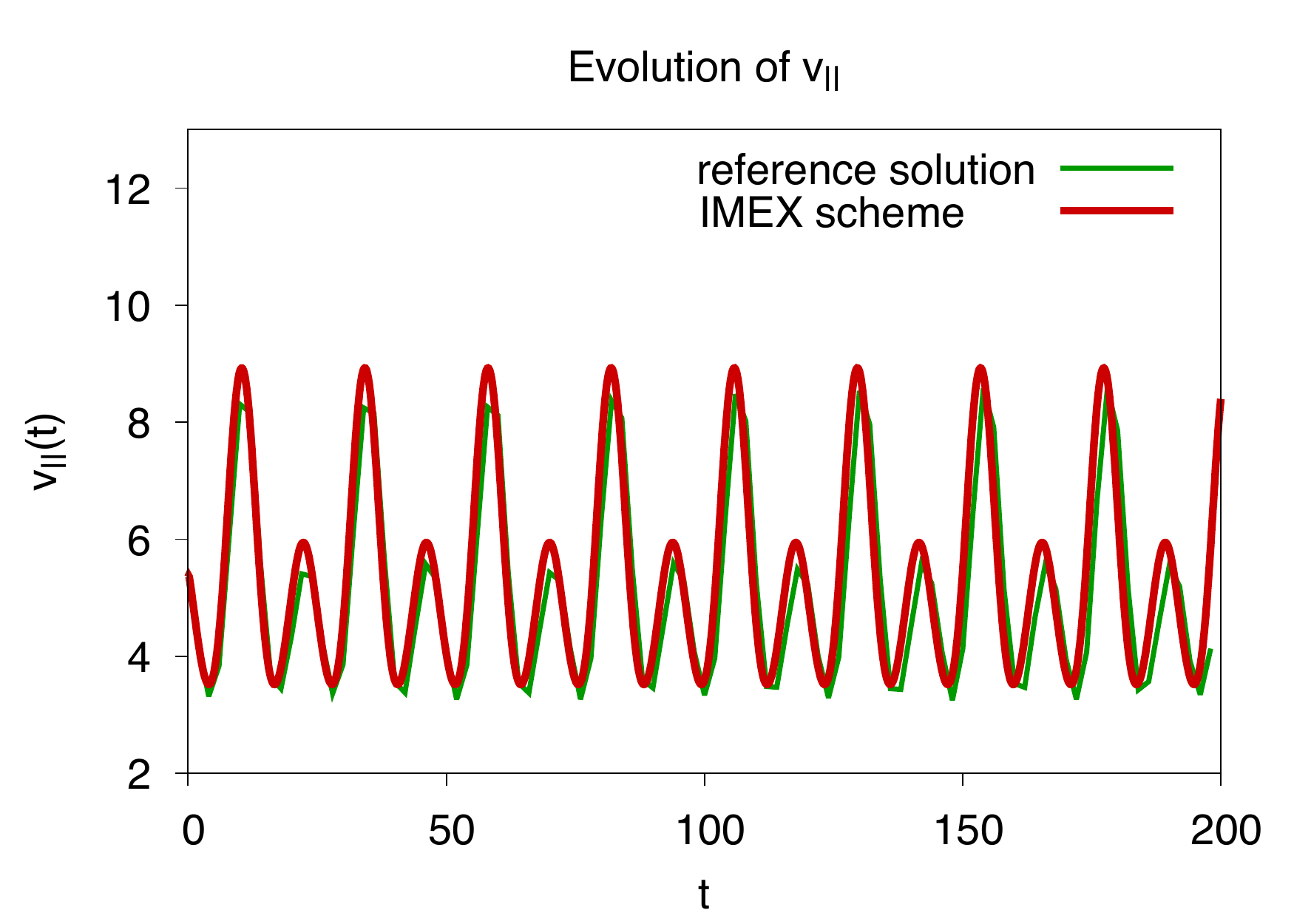} &    
\includegraphics[width=7.75cm]{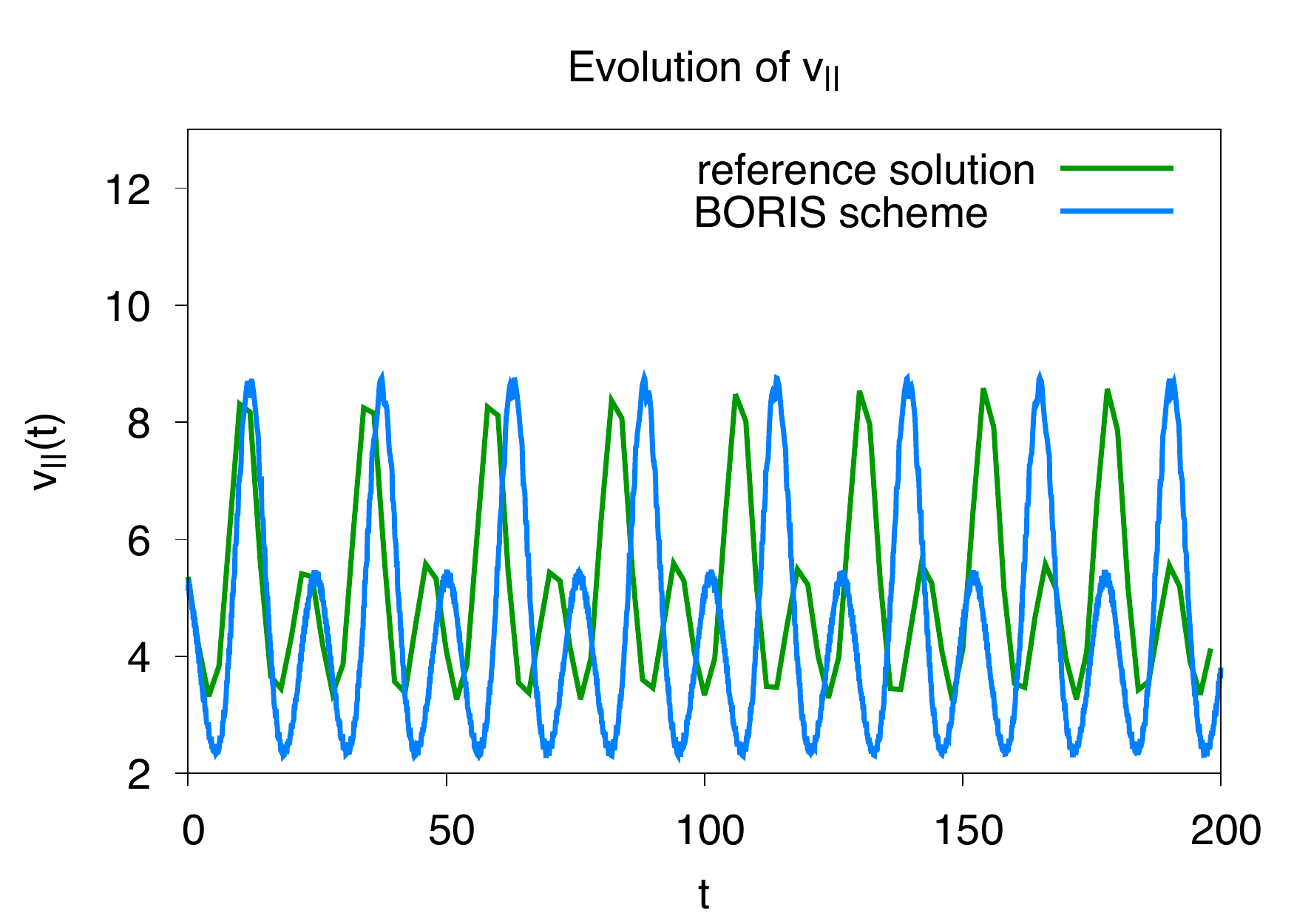}
   \\
\includegraphics[width=7.75cm]{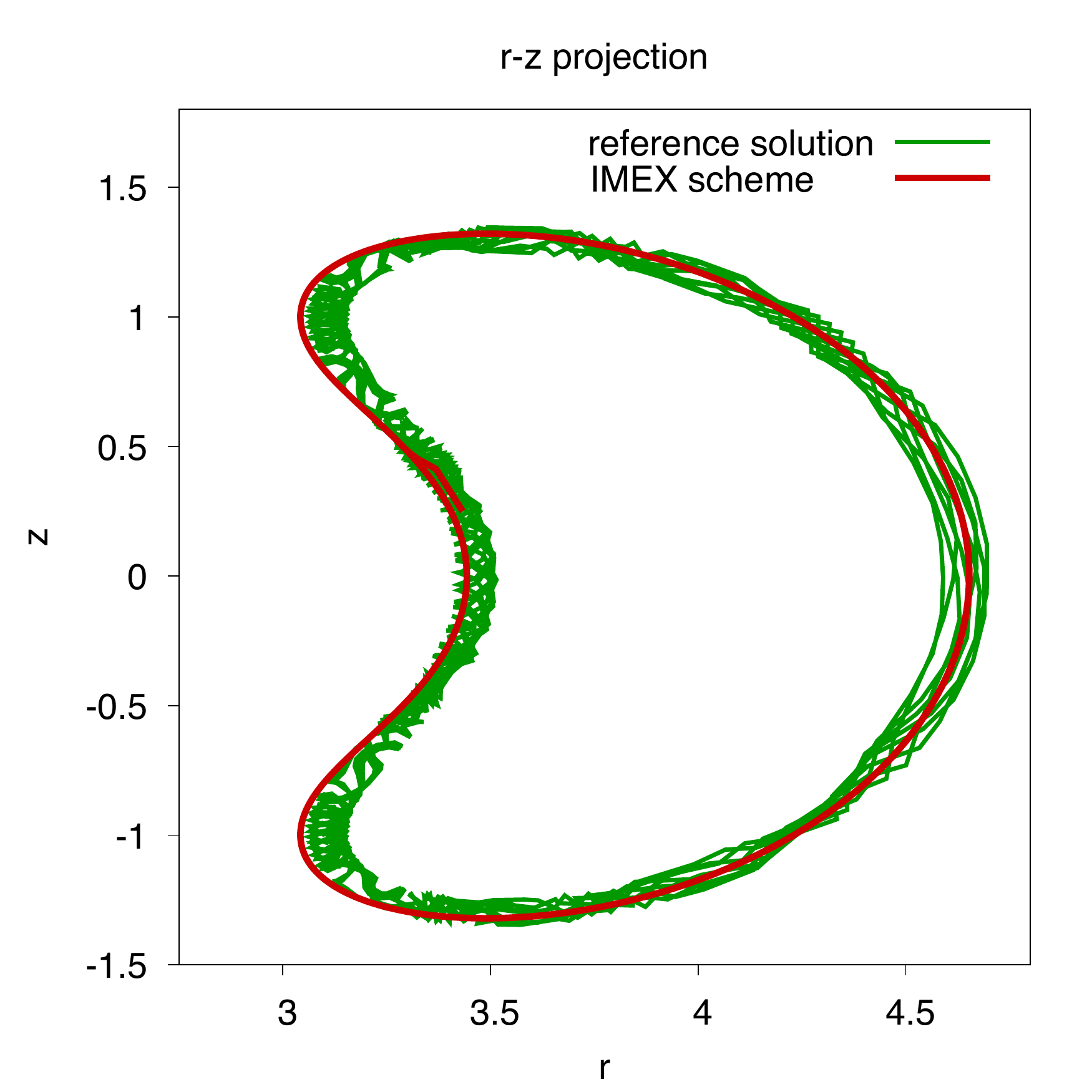}&
\includegraphics[width=7.75cm]{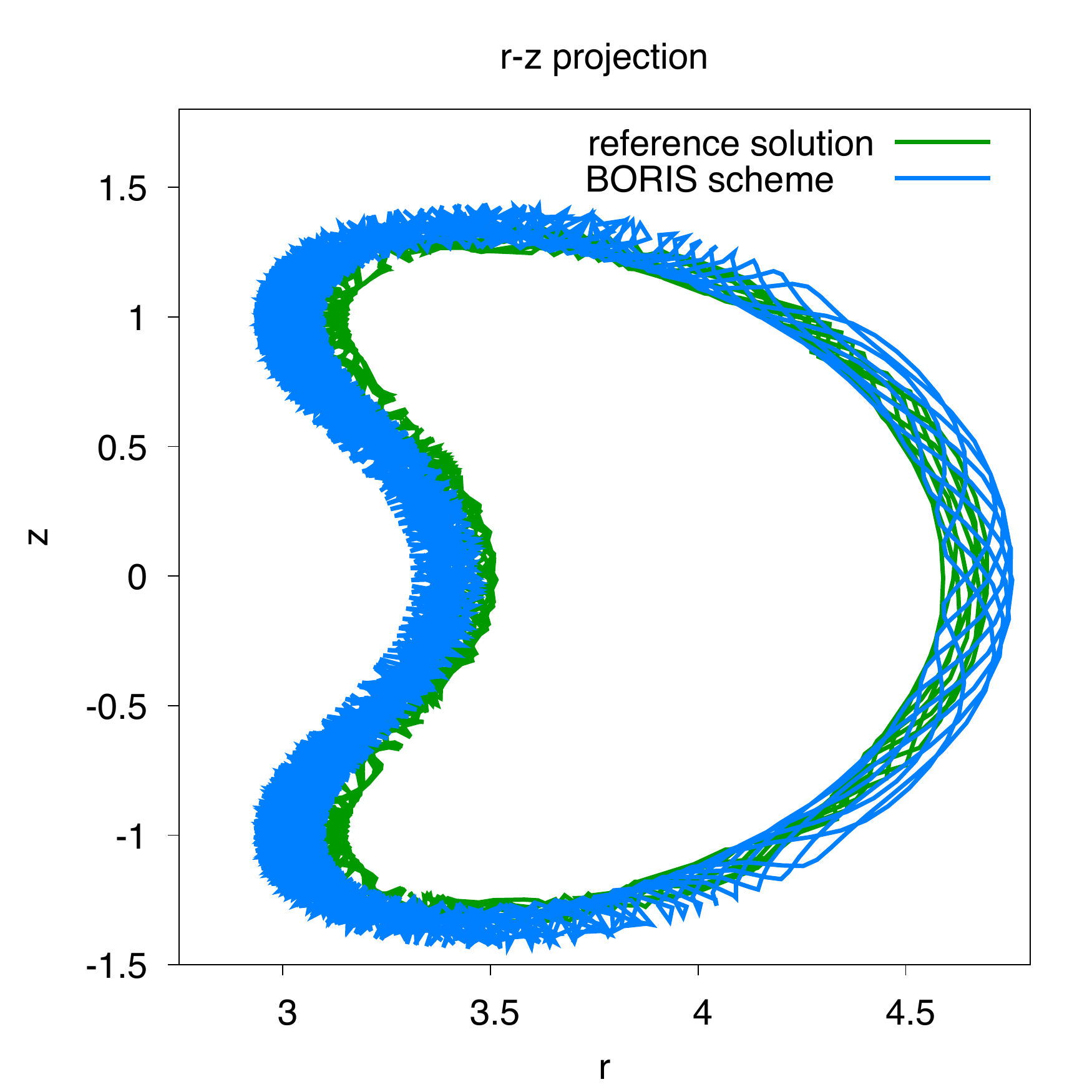}
   \\
(a)  & (b)  
\end{tabular}
\caption{\label{fig:5-2} {\bf Tokamak Equilibrium $\eps=5\,\times\,10^{-2}$.}
  Time evolution of
  $v_\parallel$ and $r$-$z$ projection obtained with (a) the second order scheme \eqref{scheme:3-1}-\eqref{scheme:3-2} and (b) the second order Boris scheme \cite{boris},   with $\Delta t=10^{-3}$.}
 \end{center}
\end{figure}

For a smaller $\eps= 10^{-2}$ (and the same time step $\Delta t=10^{-2}$), the
trajectory obtained by applying Boris scheme again oscillates with a spuriously
larger amplitude. Moreover,  for large time ($t\geq 100$) there is a
phase shift in time on the evolution of the potential energy and kinetic
energy. Let us stress that these spurious oscillations and phase shifts do decrease when taking smaller time steps.  However, with the same time step, the IMEX scheme \eqref{scheme:3-1}-\eqref{scheme:3-2} is much more stable and gives an accurate approximation of the trajectory even for large times. 

Incidentally we point out that on the latter one observes the junction of two bananas in spatial trajectories.

\begin{figure}[H]
\begin{center}
 \begin{tabular}{cc}
\includegraphics[width=7.75cm]{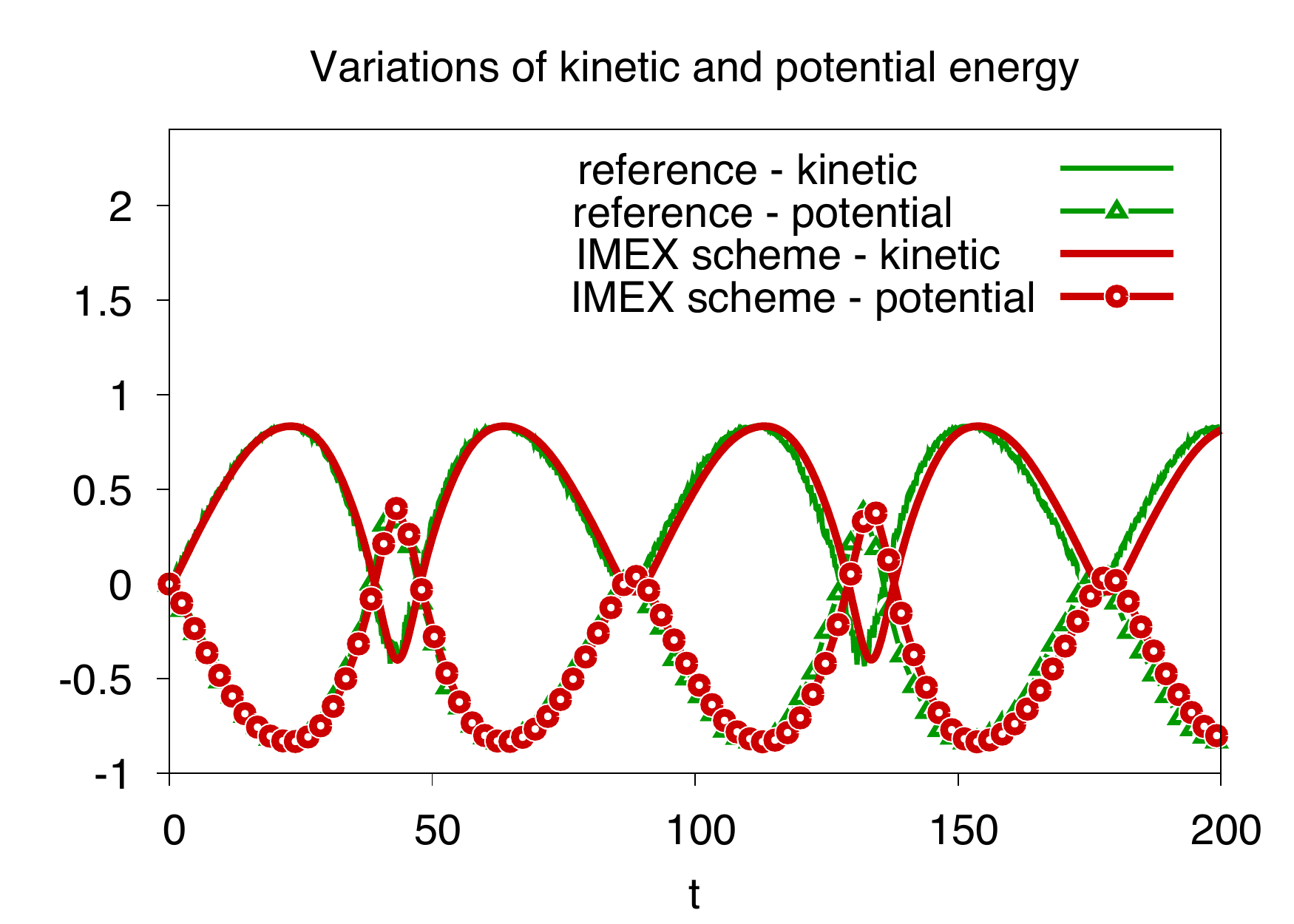} &    
\includegraphics[width=7.75cm]{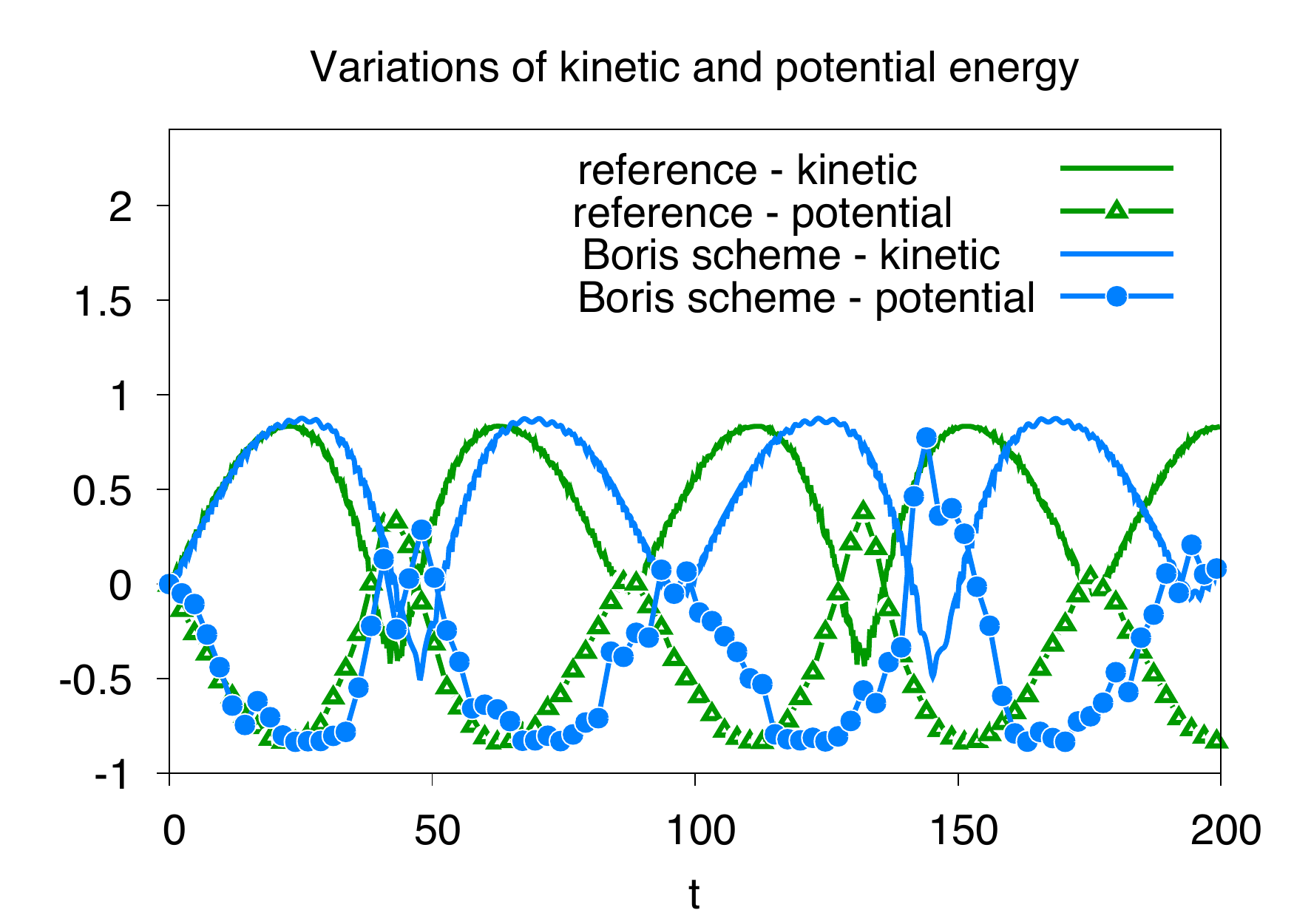}
\\
\includegraphics[width=7.75cm]{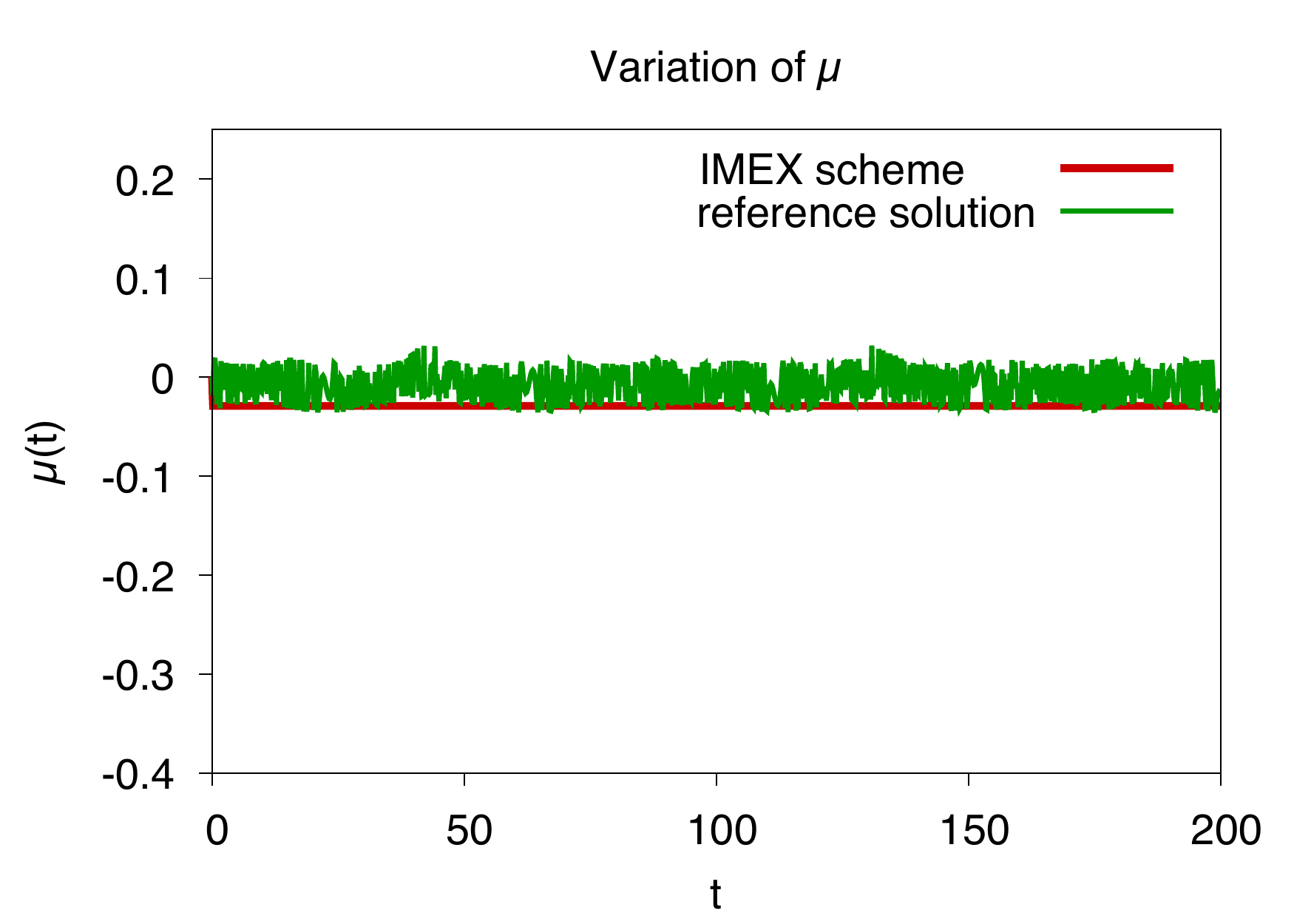}&
\includegraphics[width=7.75cm]{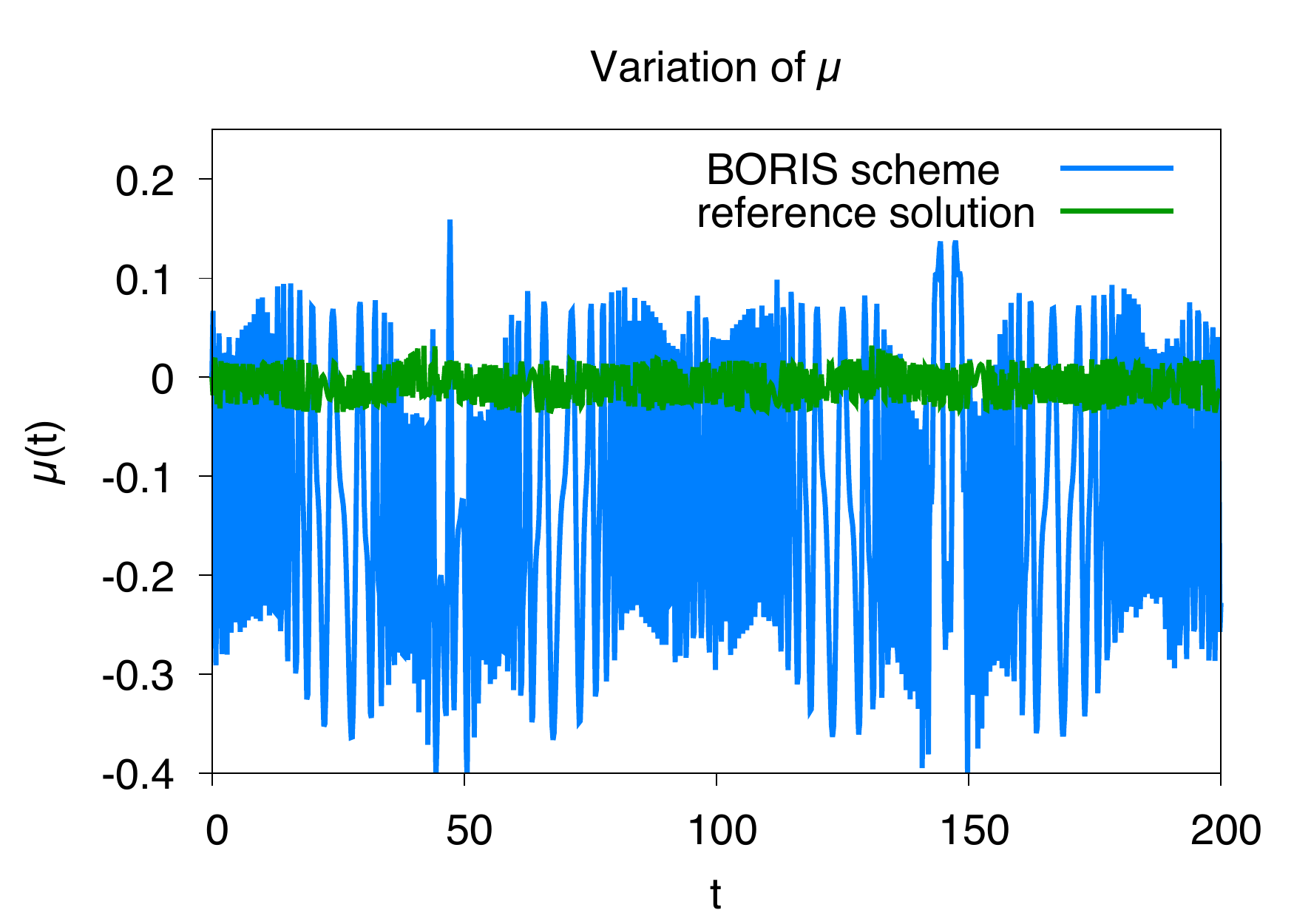}
\\
(a)  & (b)  
 \end{tabular}
 \caption{\label{fig:6-1} {\bf Tokamak Equilibrium $\eps=10^{-2}$.}
  Time variation of kinetic \& potential energy and adiabatic invariant  $\mu$ obtained with (a) the second order scheme \eqref{scheme:3-1}-\eqref{scheme:3-2} and (b) the second order Boris scheme \cite{boris}, with $\Delta t=10^{-2}$.}
 \end{center}
\end{figure}

\begin{figure}[H] 
\begin{center}
  \begin{tabular}{cc}
 \includegraphics[width=7.75cm]{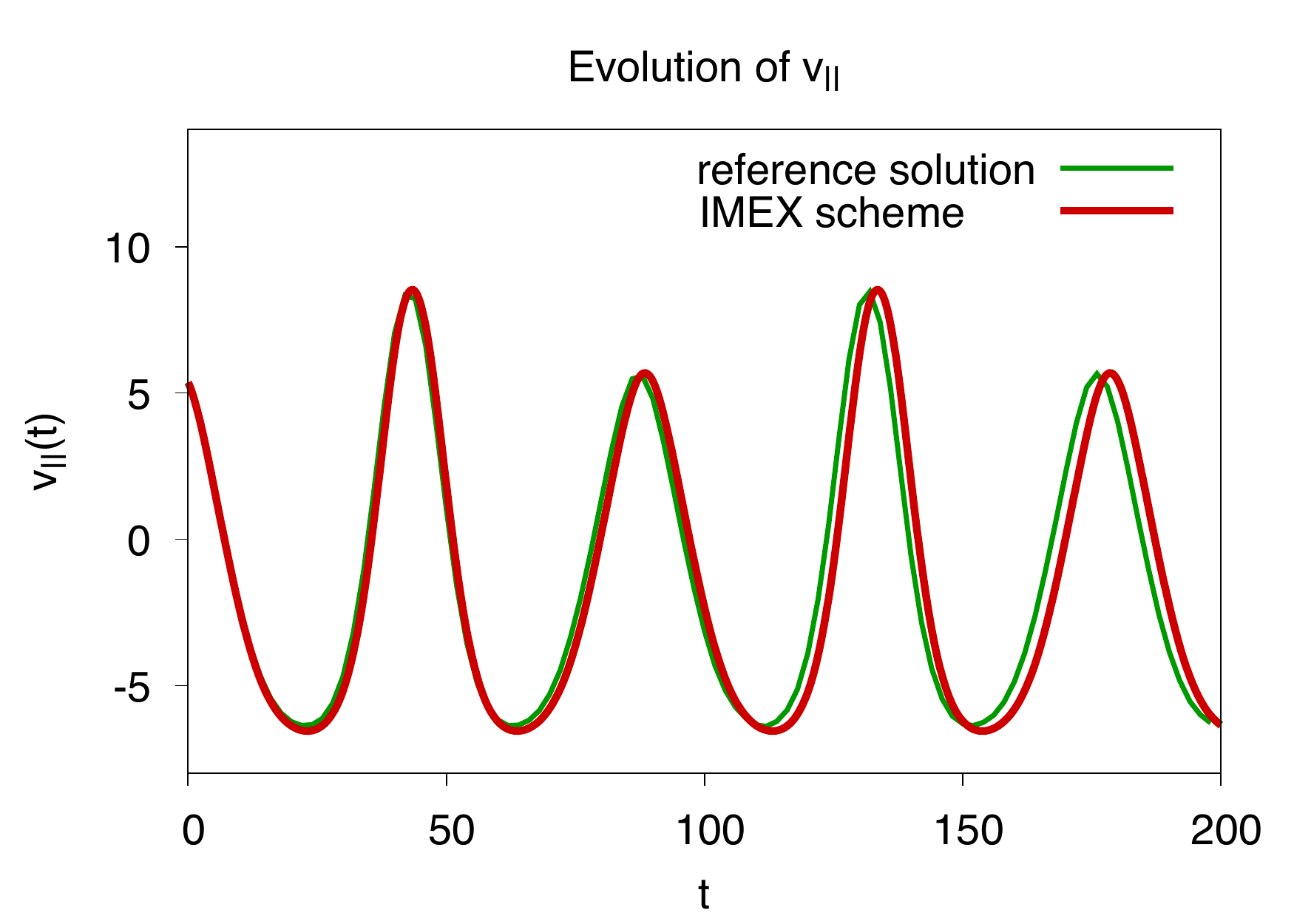} &    
\includegraphics[width=7.75cm]{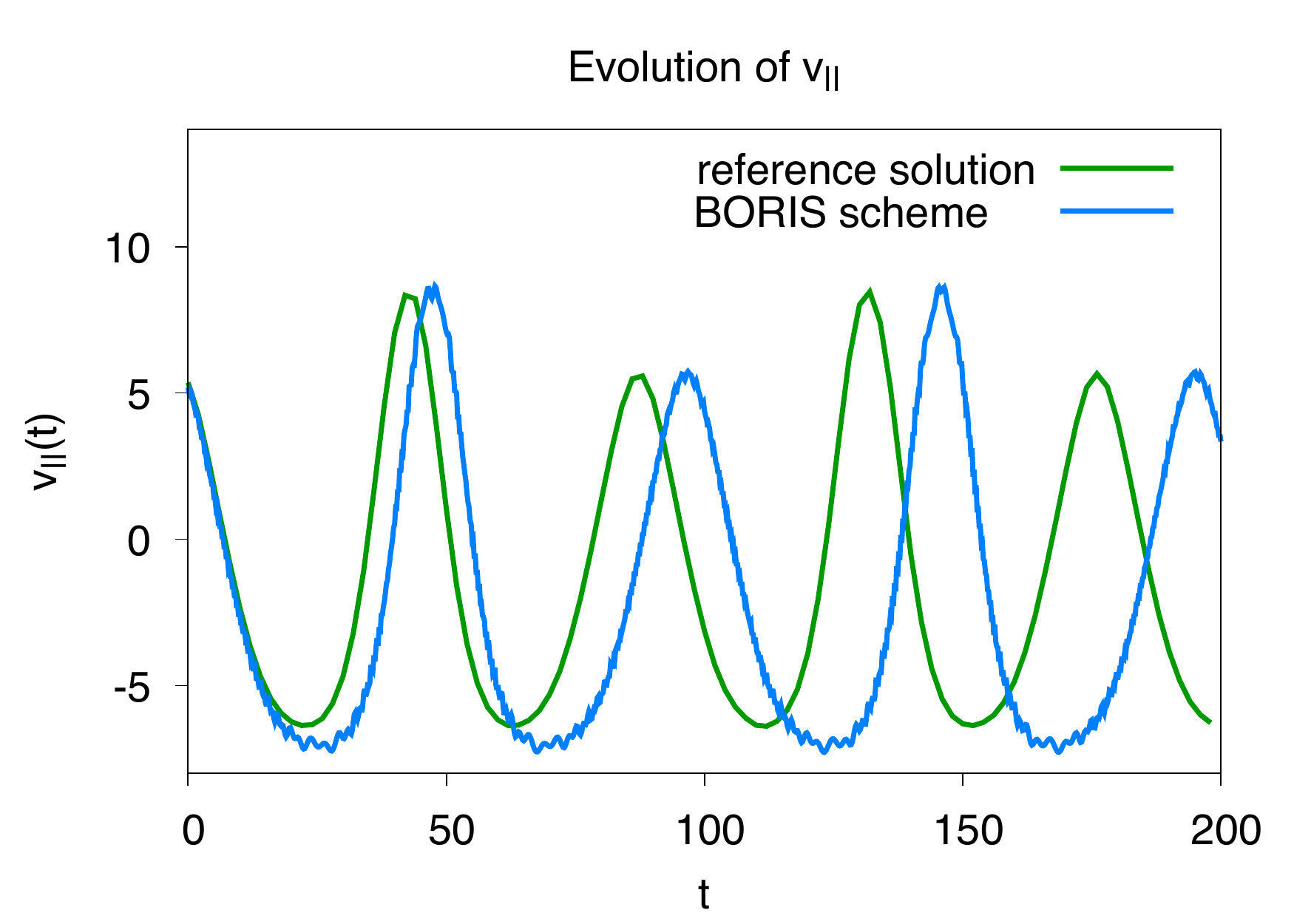}
 \\
\includegraphics[width=7.75cm]{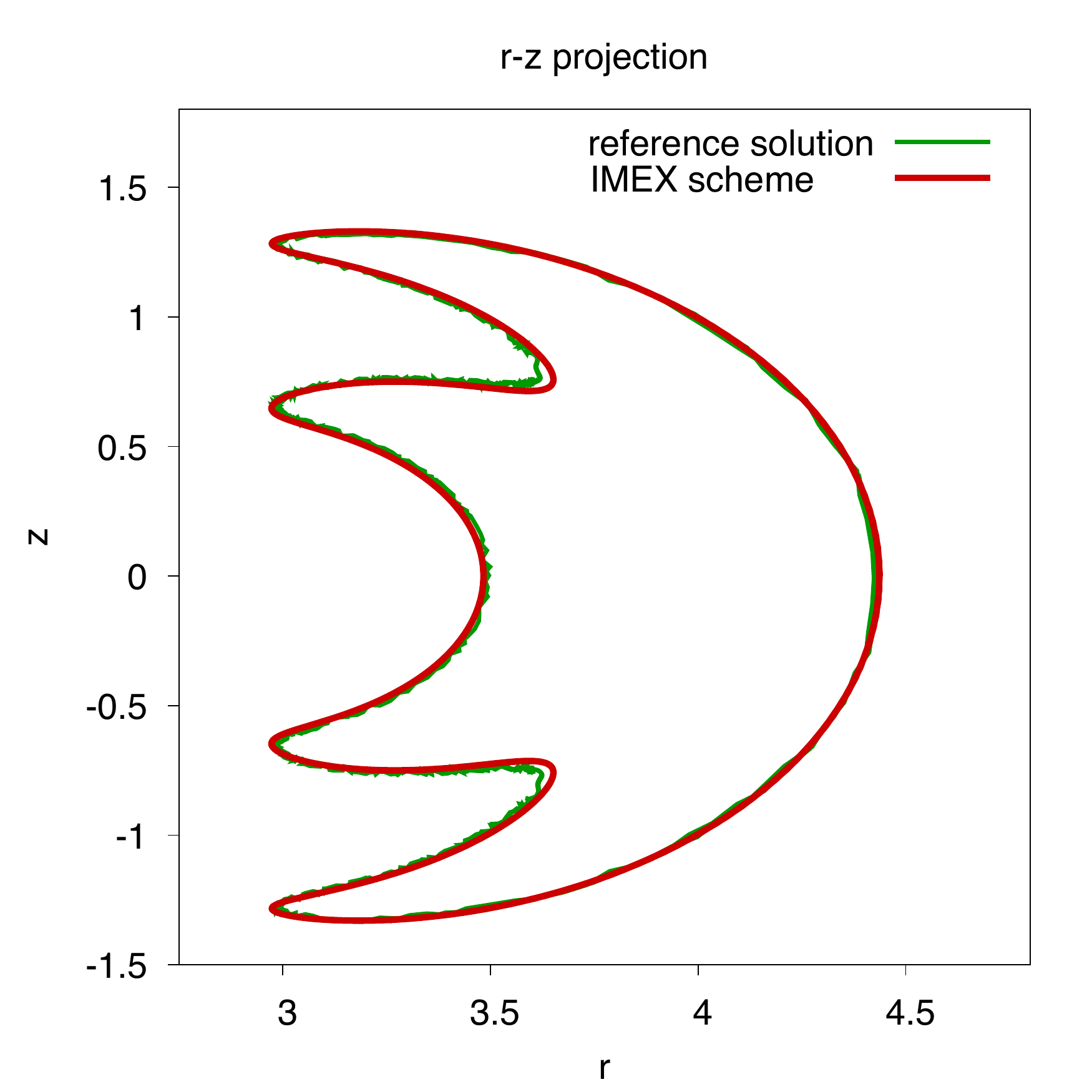}&
\includegraphics[width=7.75cm]{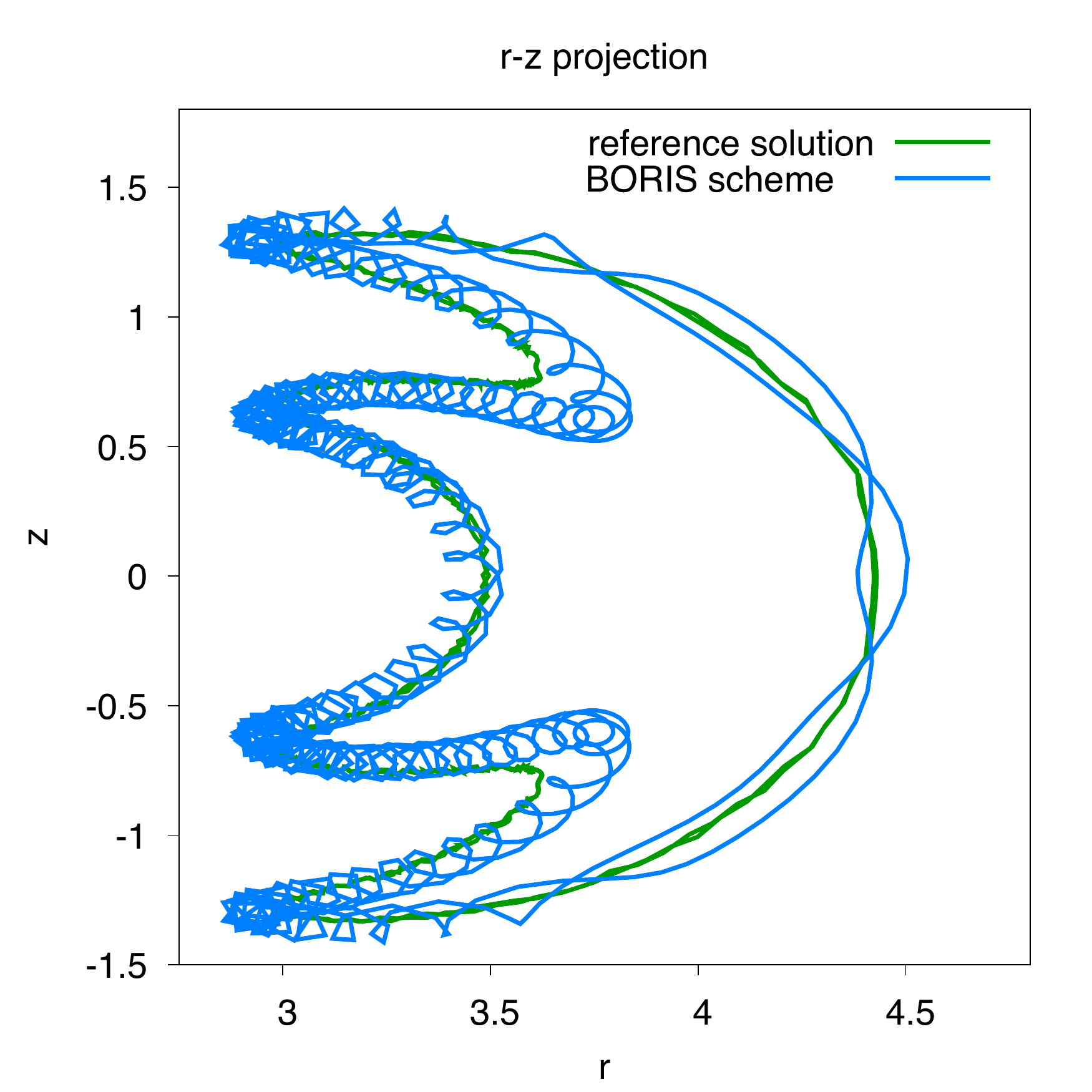}
   \\
(a)  & (b)  
\end{tabular}
\caption{\label{fig:6-2} {\bf Tokamak Equilibrium  $\eps=10^{-2}$.}
  Time evolution of
  $v_\parallel$ and $r$-$z$ projection obtained with (a) the second
  order scheme \eqref{scheme:3-1}-\eqref{scheme:3-2} and  (b) the second order Boris scheme \cite{boris}   with $\Delta t=10^{-2}$.}
 \end{center}
\end{figure}

Finally in Figures \ref{fig:7-1} and \eqref{fig:7-2} we report the
numerical results obtained for $\eps=10^{-3}$ (and $\Delta t=10^{-2}$). On the one hand, in this regime, the Boris scheme is still stable but poorly accurate. It produces some
oscillations on the different quantities as the potential energy and the
adiabatic invariant $\mu$ and the quantity $\vpar$ very rapidly desynchronizes.  On the other hand, the IMEX scheme \eqref{scheme:3-1}-\eqref{scheme:3-2} gives smooth and accurate
results essentially indistinguishable from the reference ones. It illustrates the robustness of our approach in term of stability and accuracy with respect to $\eps\ll 1$. 

\begin{figure}[H]
\begin{center}
\begin{tabular}{cc}
\includegraphics[width=7.75cm]{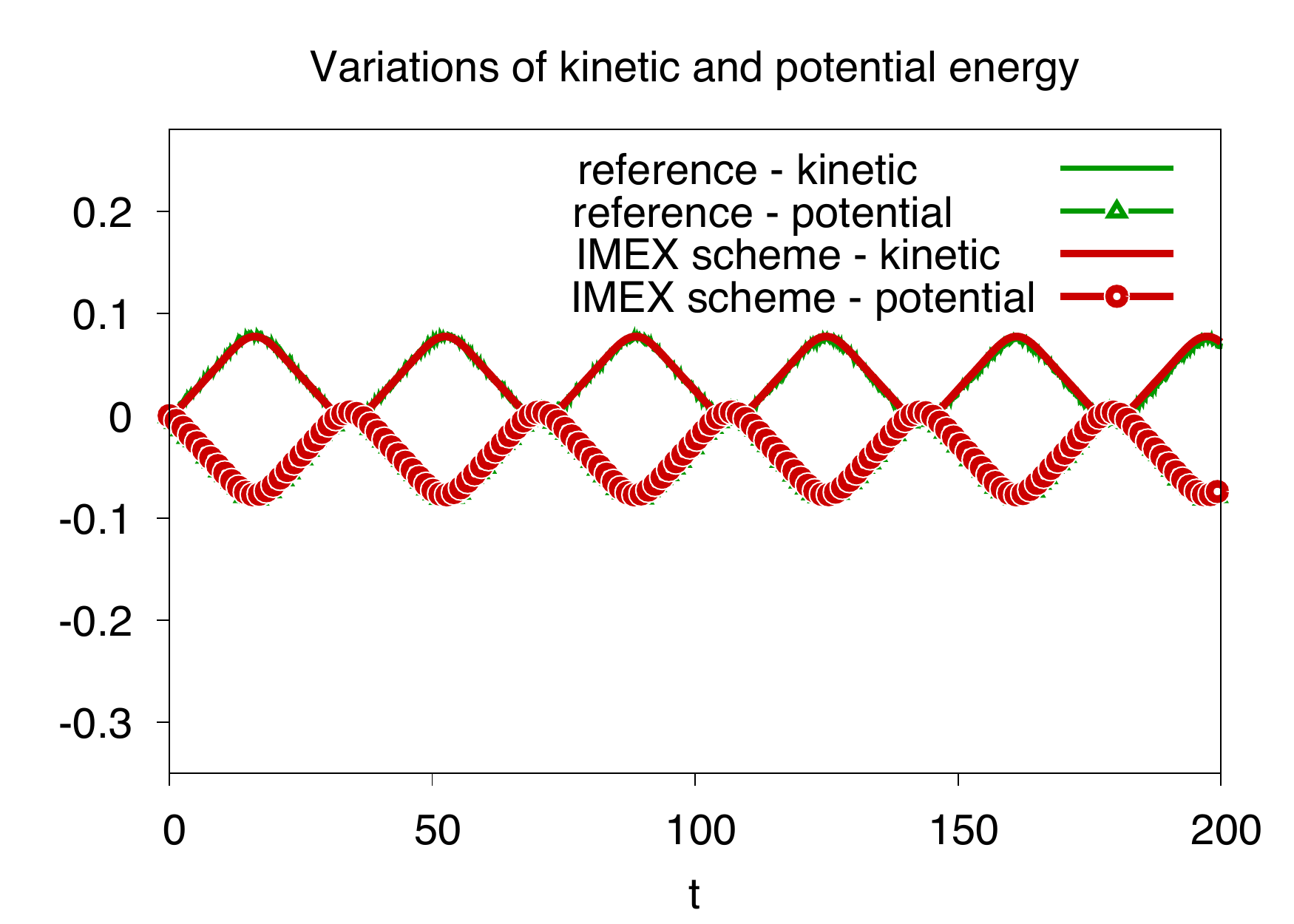} &    
\includegraphics[width=7.75cm]{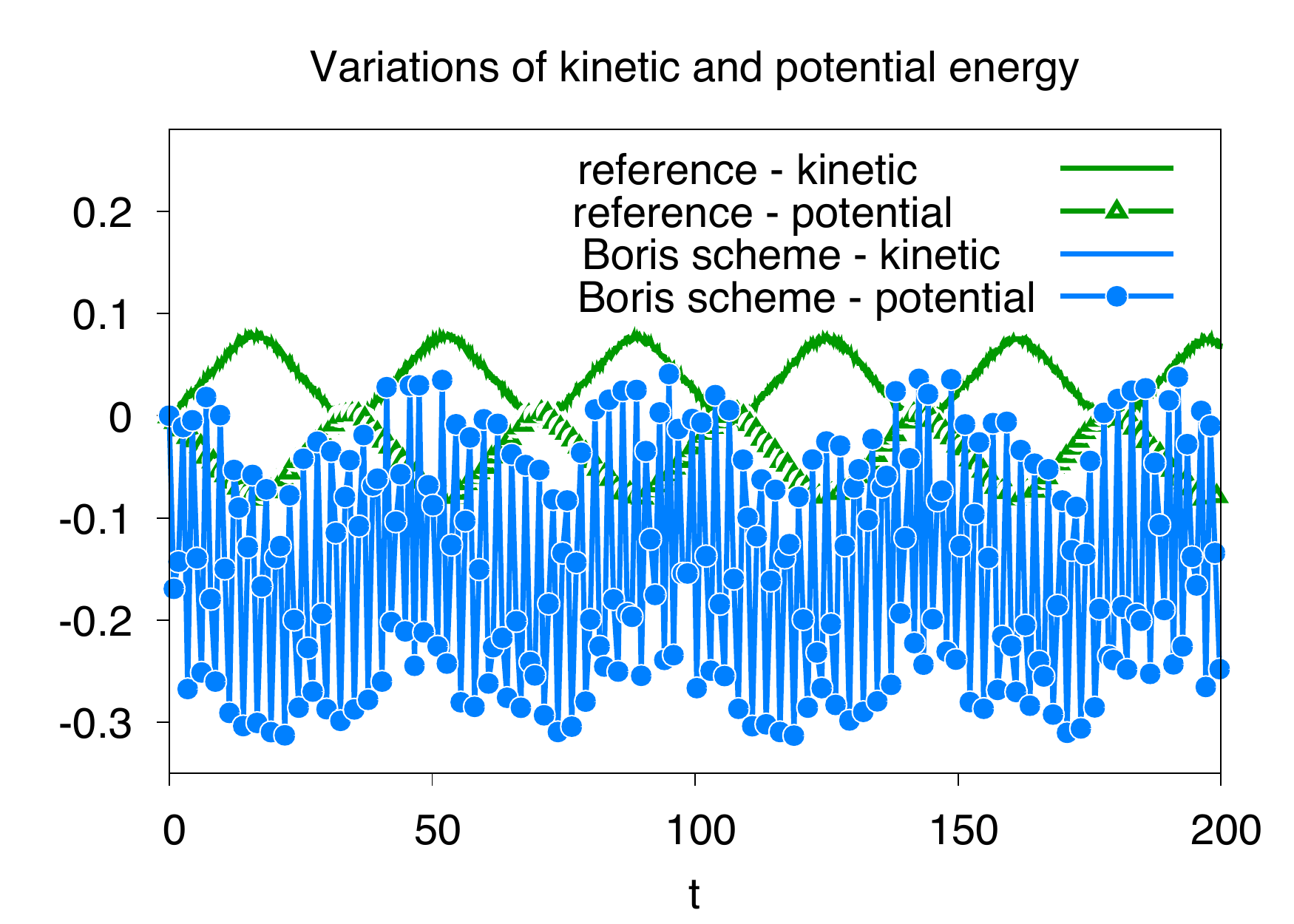}
   \\
\includegraphics[width=7.75cm]{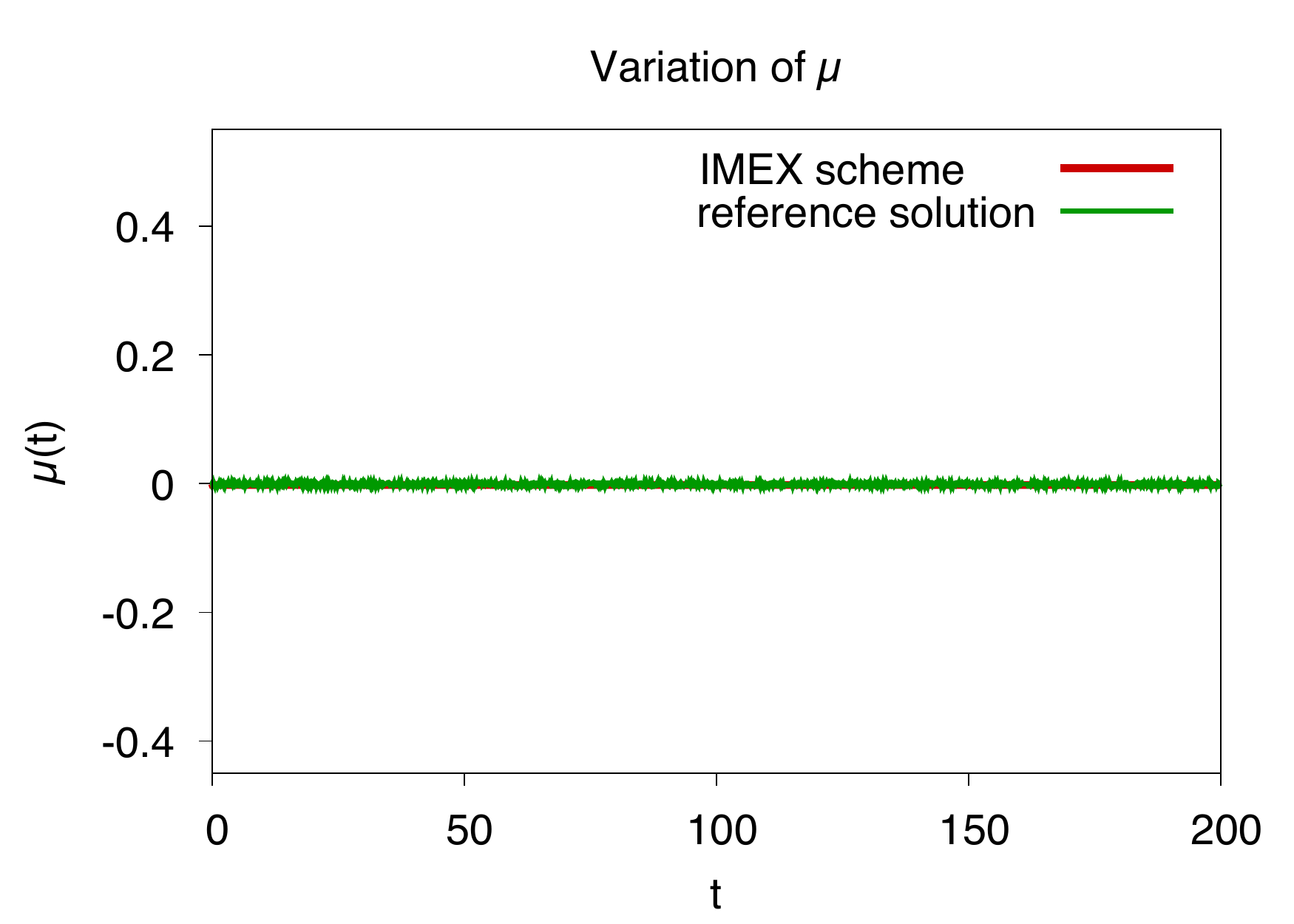}&
\includegraphics[width=7.75cm]{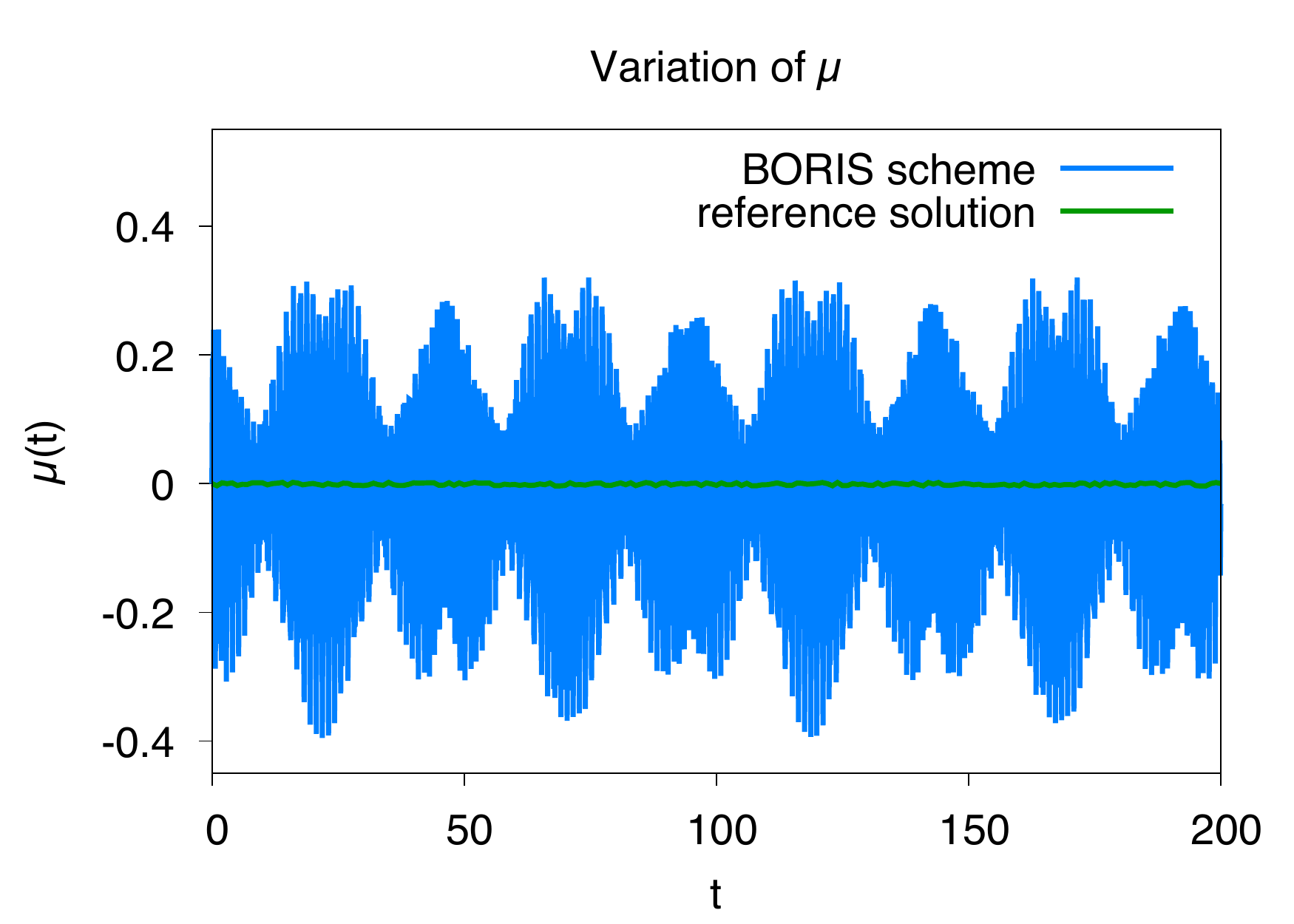}
   \\
(a)  & (b)  
\end{tabular}
\caption{\label{fig:7-1} {\bf Tokamak Equilibrium $\eps=10^{-3}$.}
  Time variation of kinetic \& potential energy and adiabatic
  invariant  $\mu$ obtained with (a) the second
  order scheme \eqref{scheme:3-1}-\eqref{scheme:3-2} and (b) the second order Boris scheme \cite{boris}, with $\Delta t=10^{-2}$.}
 \end{center}
\end{figure}

\begin{figure}[H]
\begin{center}
 \begin{tabular}{cc}
\includegraphics[width=7.75cm]{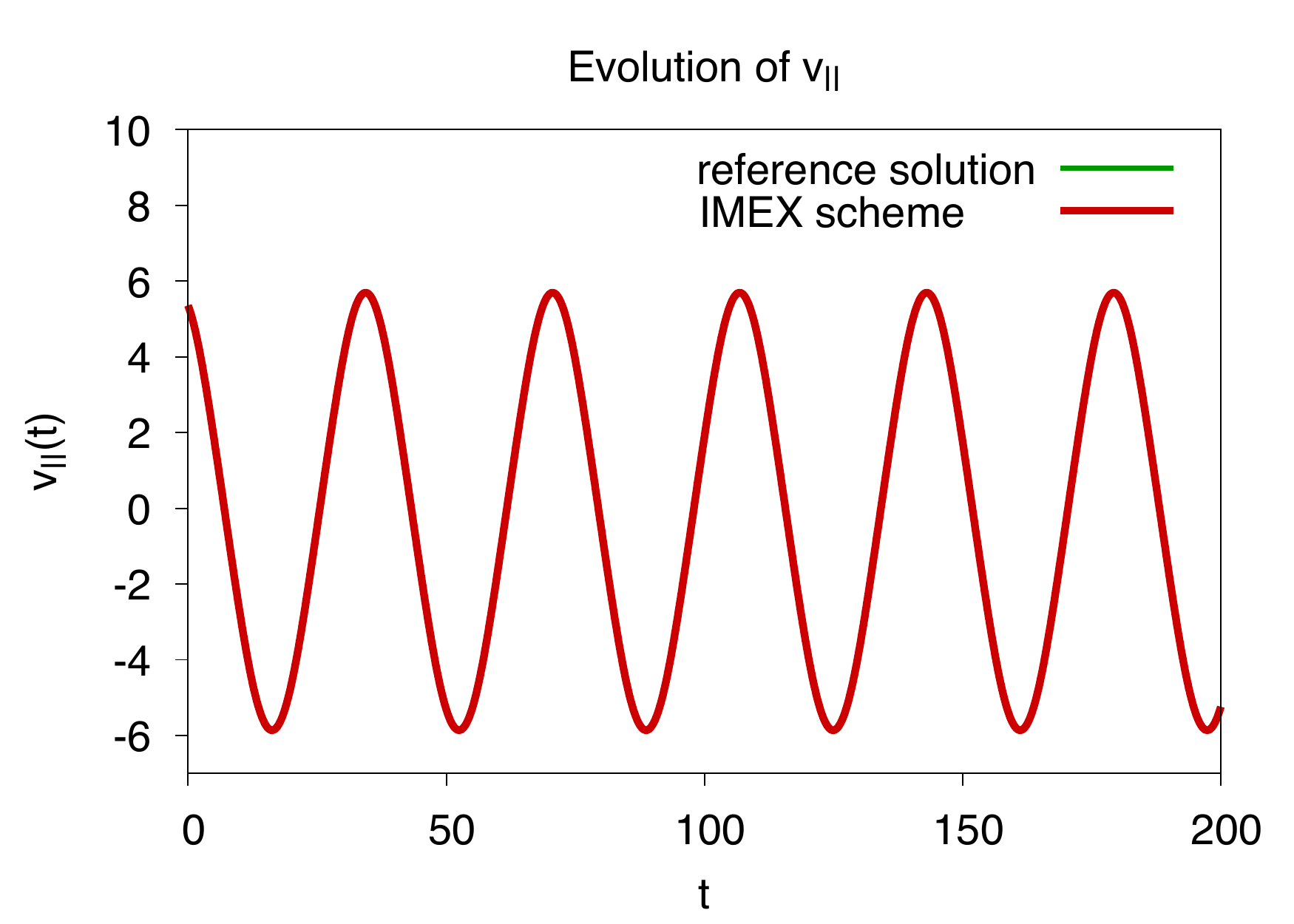} &    
\includegraphics[width=7.75cm]{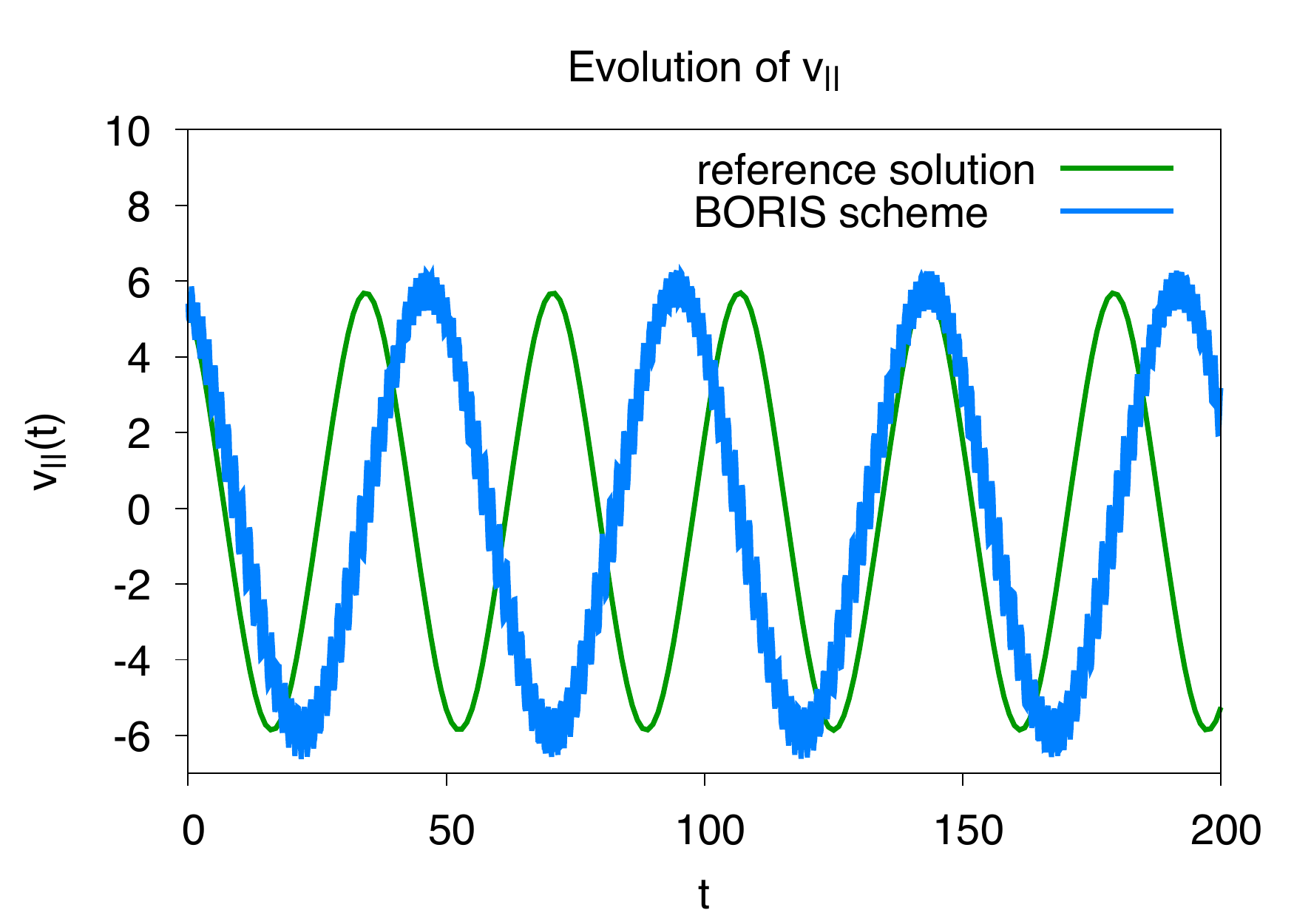}
   \\
\includegraphics[width=7.75cm]{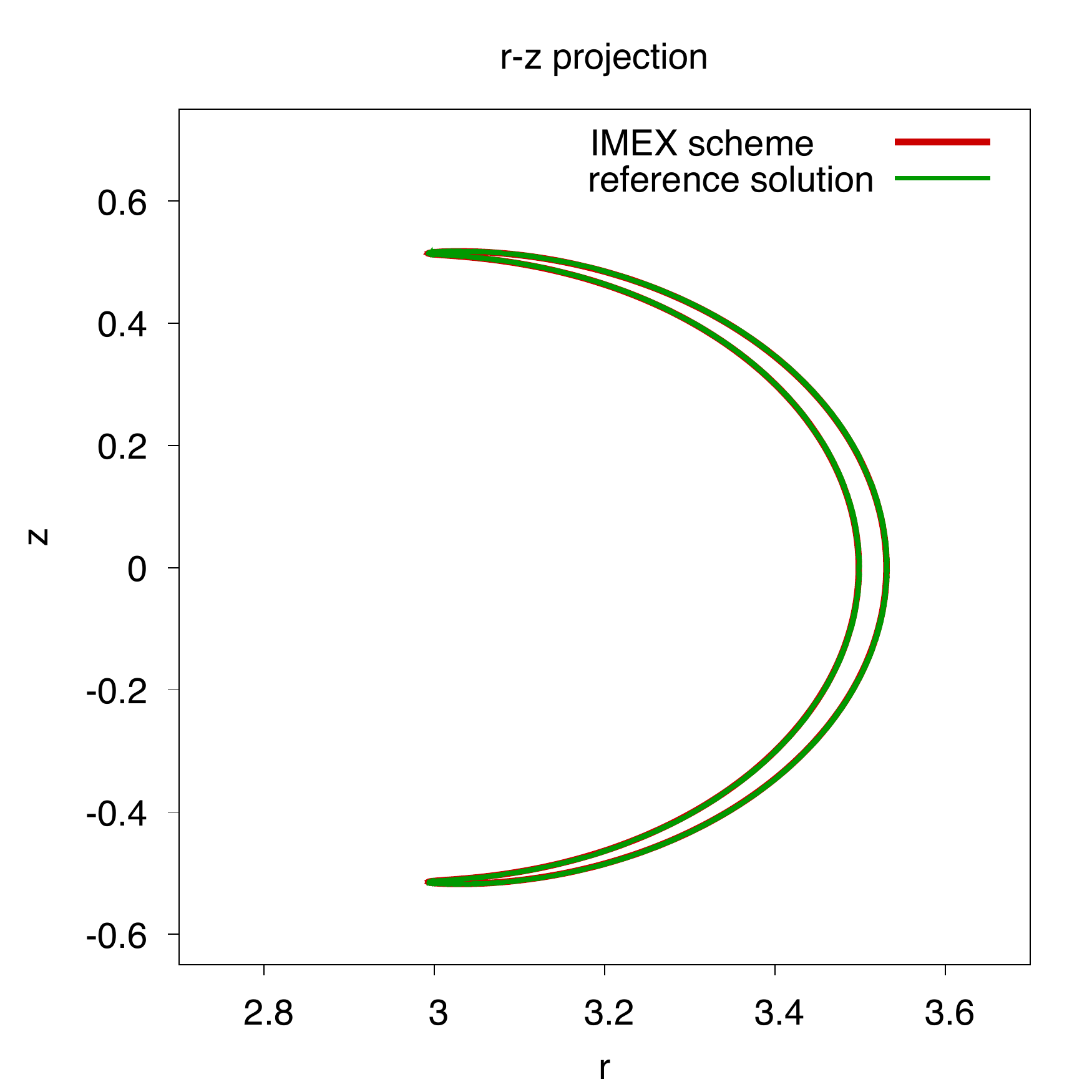}&
\includegraphics[width=7.75cm]{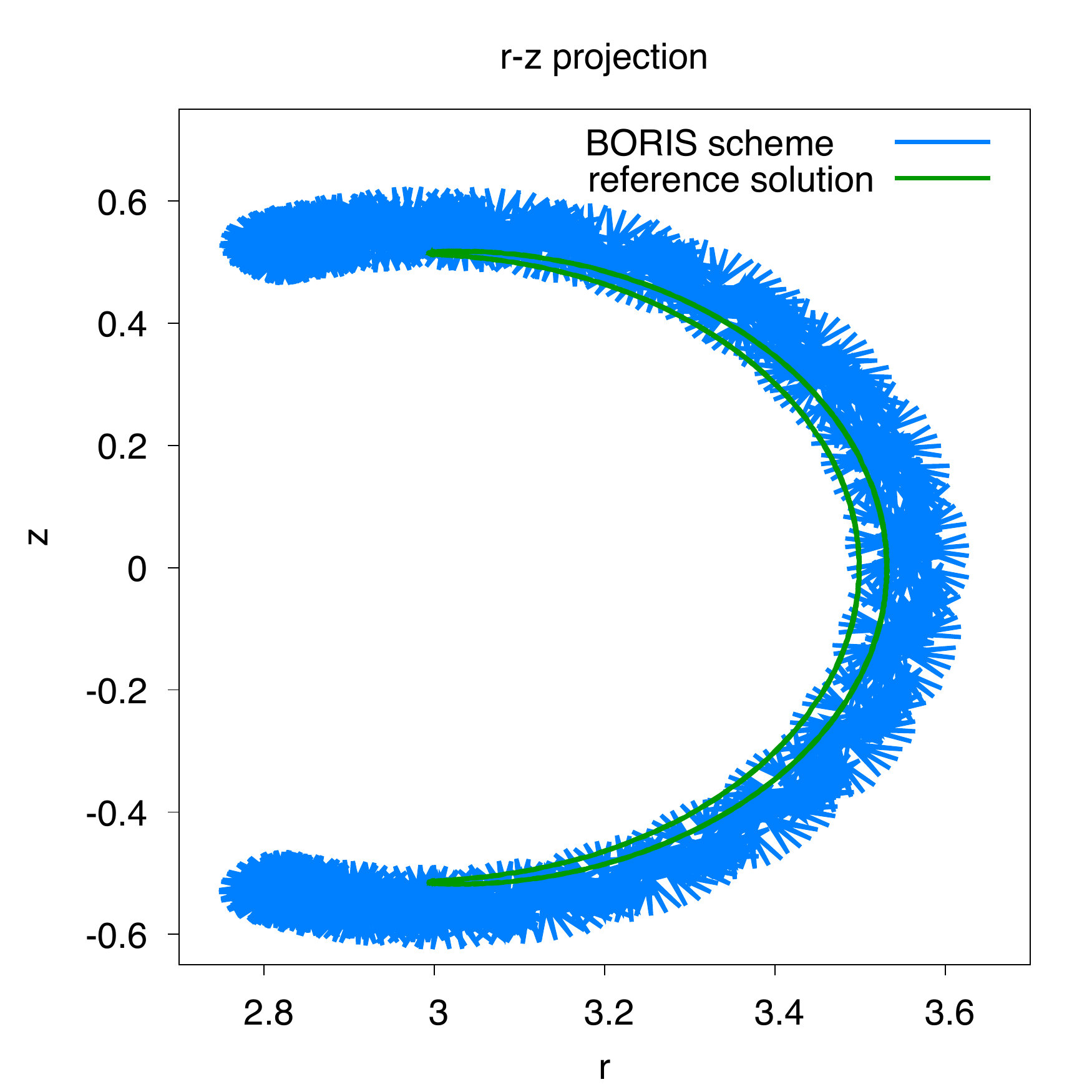}
   \\
(a)  & (b)  
\end{tabular}
\caption{\label{fig:7-2} {\bf Tokamak Equilibrium $\eps=10^{-3}$.}
  Time evolution of 
  $v_\parallel$ and $r$-$z$ projection obtained with (a) the second
  order scheme \eqref{scheme:3-1}-\eqref{scheme:3-2} and (b) the second order Boris scheme \cite{boris}, with $\Delta t=10^{-2}$.}
 \end{center}
\end{figure}

\section{Conclusion and perspectives}
\label{sec:6}
\setcounter{equation}{0}
In the present paper we have proposed a class of semi-implicit time discretization
techniques for particle-in cell simulations in torus configurations, mimicking tokamak fusion devices. The main feature of our approach is to guarantee the accuracy and stability on slow scale variables even when the amplitude of the magnetic field becomes large, thus allowing a capture of their correct long-time behavior including cases with non homogeneous magnetic fields and coarse time grids. Even on large time simulations the obtained numerical schemes also provide an acceptable accuracy on physical invariants (total energy for any $\eps$, adiabatic invariant when $\eps\ll1$) whereas fast scales are automatically filtered when the time step is large compared to $\eps^2$.

As a theoretical validation we have proved that the slow part of the
approximation converges when $\eps\rightarrow 0$ to the solution of a
limiting scheme for the asymptotic evolution, that preserves the
initial order of accuracy. Yet a full proof of uniform accuracy and a
classification of admissible schemes remains to be carried out. 

From a practical point of view, the next natural step would be to consider
the coupling with the Poisson equation for the computation of a self-consistent electric field and the consideration of even more realistic geometries. 

\section*{Acknowledgements}
FF was supported by the EUROfusion Consortium and has received funding
from the Euratom research and training programme 2014-2018 under grant
agreement No 633053. The views and opinions expressed herein do not
necessarily reflect those of the European Commission.

LMR expresses his appreciation of the hospitality of IMT, Universit\'e Toulouse III, during part of the preparation of the present contribution.

\bibliography{my_biblio.bib}

\begin{thebibliography}{10}

\bibitem{bellan_2006_fundamentals}
P.~M. Bellan.
\newblock {\em Fundamentals of plasma physics}.
\newblock Cambridge University Press, 2008.

\bibitem{boris}
J.~Boris.
\newblock Relativistic plasma simulation-optimization.
\newblock In {\em 4th conference on numerical simulation of plasma}, page~3,
  1970.

\bibitem{BFR:15}
S.~Boscarino, F.~Filbet, and G.~Russo.
\newblock High order semi-implicit schemes for time dependent partial
  differential equations.
\newblock {\em Journal of Scientific Computing}, 68(3):975--1001, 2016.

\bibitem{Burby2020}
J.~W. Burby and T.~J. Klotz.
\newblock I{NVITED}: {S}low manifold reduction for plasma science.
\newblock {\em Commun. Nonlinear Sci. Numer. Simul.}, 89:105289, 62, 2020.

\bibitem{cerfon2010}
A.~J. Cerfon and J.~P. Freidberg.
\newblock “one size fits all” analytic solutions to the grad--shafranov
  equation.
\newblock {\em Physics of Plasmas}, 17(3):032502, 2010.

\bibitem{cruz2}
P.~Chartier, N.~Crouseilles, M.~Lemou, F.~M\'{e}hats, and X.~Zhao.
\newblock Uniformly accurate methods for {V}lasov equations with
  non-homogeneous strong magnetic field.
\newblock {\em Math. Comp.}, 88(320):2697--2736, 2019.

\bibitem{cruz1}
P.~Chartier, N.~Crouseilles, M.~Lemou, F.~M\'{e}hats, and X.~Zhao.
\newblock Uniformly accurate methods for three dimensional {V}lasov equations
  under strong magnetic field with varying direction.
\newblock {\em SIAM J. Sci. Comput.}, 42(2):B520--B547, 2020.

\bibitem{chen2022}
G.~Chen and L.~Chac{\'o}n.
\newblock An implicit, conservative and asymptotic-preserving electrostatic
  particle-in-cell algorithm for arbitrarily magnetized plasmas in uniform
  magnetic fields.
\newblock {\em arXiv preprint arXiv:2205.09187}, 2022.

\bibitem{Cohen2007}
R.~Cohen, A.~Friedman, D.~Grote, and J.-L. Vay.
\newblock Large-timestep mover for particle simulations of arbitrarily
  magnetized species.
\newblock {\em Nuclear Instruments and Methods in Physics Research Section A:
  Accelerators, Spectrometers, Detectors and Associated Equipment},
  577(1):52--57, 2007.
\newblock Proceedings of the 16th International Symposium on Heavy Ion Inertial
  Fusion.

\bibitem{FR1}
F.~Filbet and L.~M. Rodrigues.
\newblock Asymptotically stable particle-in-cell methods for the
  {V}lasov-{P}oisson system with a strong external magnetic field.
\newblock {\em SIAM J. Numer. Anal.}, 54(2):1120--1146, 2016.

\bibitem{FR2}
F.~Filbet and L.~M. Rodrigues.
\newblock Asymptotically preserving particle-in-cell methods for inhomogeneous
  strongly magnetized plasmas.
\newblock {\em SIAM J. Numer. Anal.}, 55(5):2416--2443, 2017.

\bibitem{FR:JEP}
F.~Filbet and L.~M. Rodrigues.
\newblock Asymptotics of the three-dimensional {V}lasov equation in the large
  magnetic field limit.
\newblock {\em J. \'{E}c. polytech. Math.}, 7:1009--1067, 2020.

\bibitem{FR3}
F.~Filbet, L.~M. Rodrigues, and H.~Zakerzadeh.
\newblock Convergence analysis of asymptotic preserving schemes for strongly
  magnetized plasmas.
\newblock {\em Numerische Mathematik}, 149(3):549--593, 2021.

\bibitem{Hairer2017}
E.~Hairer and C.~Lubich.
\newblock Symmetric multistep methods for charged-particle dynamics.
\newblock {\em SMAI J. Comput. Math.}, 3:205--218, 2017.

\bibitem{Hairer2020}
E.~Hairer and C.~Lubich.
\newblock Long-term analysis of a variational integrator for charged-particle
  dynamics in a strong magnetic field.
\newblock {\em Numer. Math.}, 144(3):699--728, 2020.

\bibitem{hairer2022}
E.~Hairer, C.~Lubich, and Y.~Shi.
\newblock Large-stepsize integrators for charged-particle dynamics over
  multiple time scales.
\newblock {\em Numerische Mathematik}, pages 1--33, 2022.

\bibitem{HanKwan_PhD}
D.~Han-Kwan.
\newblock {\em Contribution {\`a} l'{\'e}tude math{\'e}matique des plasmas
  fortement magn{\'e}tis{\'e}s}.
\newblock PhD thesis, Universit{\'e} Pierre et Marie Curie-Paris VI, 2011.

\bibitem{haz_mei_03}
R.~Hazeltine and J.~Meiss.
\newblock {\em Plasma Confinement}.
\newblock Dover Publications, 2003.

\bibitem{haz_ware_78}
R.~Hazeltine and A.~Ware.
\newblock The drift kinetic equation for toroidal plasmas with large mass
  velocities.
\newblock {\em Plasma Phys.}, 20:673--678, 1978.

\bibitem{Herda_PhD}
M.~Herda.
\newblock {\em Analyse asymptotique et num{\'e}rique de quelques mod{\`e}les
  pour le transport de particules charg{\'e}es}.
\newblock PhD thesis, Universit{\'e} Claude Bernard Lyon 1, 2017.

\bibitem{Lutz_PhD}
M.~Lutz.
\newblock {\em {\'E}tude math{\'e}matique et num{\'e}rique d'un mod{\`e}le
  gyrocin{\'e}tique incluant des effets {\'e}lectromagn{\'e}tiques pour la
  simulation d'un plasma de Tokamak}.
\newblock PhD thesis, Universit{\'e} de Strasbourg, 2013.

\bibitem{miyamoto_2006_plasma}
K.~Miyamoto.
\newblock {\em Plasma physics and controlled nuclear fusion}, volume~38 of {\em
  Springer Series on Atomic, Optical, and Plasma Physics}.
\newblock Springer-Verlag Berlin-Heidelberg, 2006.

\bibitem{chacon2020}
L.~F. Ricketson and L.~Chac{\'o}n.
\newblock An energy-conserving and asymptotic-preserving charged-particle orbit
  implicit time integrator for arbitrary electromagnetic fields.
\newblock {\em Journal of Computational Physics}, 418:109639, 2020.

\bibitem{vu1995}
H.~Vu and J.~Brackbill.
\newblock Accurate numerical solution of charged particle motion in a magnetic
  field.
\newblock {\em Journal of Computational Physics}, 116(2):384--387, 1995.

\bibitem{Wang2020}
B.~Wang, X.~Wu, and Y.~Fang.
\newblock A two-step symmetric method for charged-particle dynamics in a normal
  or strong magnetic field.
\newblock {\em Calcolo}, 57(3):Paper No. 29, 21, 2020.

\bibitem{wang2021}
B.~Wang and X.~Zhao.
\newblock Error estimates of some splitting schemes for charged-particle
  dynamics under strong magnetic field.
\newblock {\em SIAM Journal on Numerical Analysis}, 59(4):2075--2105, 2021.

\bibitem{Webb2014}
S.~D. Webb.
\newblock Symplectic integration of magnetic systems.
\newblock {\em J. Comput. Phys.}, 270:570--576, 2014.

\end{thebibliography}
\bibliographystyle{abbrv}

\begin{flushleft}
  \signFF \end{flushleft}
\vspace*{-44mm}
\begin{flushright}
  \signLMR \end{flushright}

\end{document}